\documentclass[reqno]{amsart}
\usepackage[left=2.7cm,right=2.7cm,top=3.5cm,bottom=3cm]{geometry}
\usepackage[english]{babel}
\usepackage[utf8]{inputenc}
\usepackage[T1]{fontenc}
\usepackage{MnSymbol}
\usepackage{amsmath}
\usepackage{amsfonts}
\usepackage{amsthm}
\usepackage{upgreek}
\usepackage[usenames,dvipsnames]{color}
\usepackage{mathrsfs}
\usepackage[
    colorlinks = true,
    linkcolor = black,
    urlcolor  = black,
    citecolor = blue,
    bookmarksopen = true,
    backref = page
]{hyperref}

\theoremstyle{plain}
\newtheorem{thm}{Theorem}[section]
\newtheorem{cor}[thm]{Corollary}
\newtheorem{lem}[thm]{Lemma}
\newtheorem{prop}[thm]{Proposition}

\theoremstyle{definition}
\newtheorem{dfn}[thm]{Definition}
\newtheorem{eg}[thm]{Example}
\newtheorem{rmk}[thm]{Remark}

\newcommand{\field}[1]{\mathbb{#1}}
\newcommand{\Q}{\field{Q}}
\newcommand{\C}{\field{C}}

\newcommand{\R}{\field{R}}

\newcommand{\N}{\field{N}}
\newcommand{\Z}{\field{Z}}
\newcommand{\A}{\field{A}}
\newcommand{\F}{\field{F}}
\newcommand{\p}{\field{P}}
\newcommand{\G}{\field{G}}

\DeclareMathOperator{\Aut}{Aut}
\DeclareMathOperator{\rank}{rank}
\DeclareMathOperator{\Lie}{Lie}

\DeclareMathOperator{\Gal}{Gal}

\DeclareMathOperator{\Hom}{Hom}
\DeclareMathOperator{\id}{id}

\DeclareMathOperator{\Spec}{Spec}

\DeclareMathOperator{\Frac}{Frac}
\DeclareMathOperator{\Frob}{Frob}

\DeclareMathOperator{\pgl}{PGL}
\DeclareMathOperator{\PSL}{PSL}
\DeclareMathOperator{\sep}{sep}

\DeclareMathOperator{\gl}{GL}

\DeclareMathOperator{\Stab}{Stab}
\DeclareMathOperator{\pr}{pr}

\DeclareMathOperator{\Inn}{Inn}
\DeclareMathOperator{\Sl}{SL}
\newcommand{\et}{\mathrm{\acute{e}t}}

\newcommand{\im}{\text{Im}}

\newcommand{\cala}{\mathcal A}

\newcommand{\calu}{\mathcal U}
\newcommand{\calb}{\mathcal B}
\newcommand{\calc}{\mathcal C}

\newcommand{\dcal}{\mathcal D}

\newcommand{\calf}{\mathcal F}
\newcommand{\gcal}{\mathcal G}

\newcommand{\cali}{\mathscr I}

\newcommand{\lcal}{\mathcal L}

\newcommand{\calo}{\mathscr O}

\newcommand{\calt}{\mathscr T}
\newcommand{\calv}{\mathscr V}

\newcommand{\caly}{\mathscr Y}

\newcommand{\gota}{\mathfrak a}

\newcommand{\gotg}{\mathfrak g}

\newcommand{\gotm}{\mathfrak m}
\newcommand{\gotn}{\mathfrak n}
\newcommand{\gotp}{\mathfrak p}
\newcommand{\gotq}{\mathfrak q}

\newcommand{\goth}{\mathfrak h}

\renewcommand{\ge}{\geqslant}
\renewcommand{\le}{\leqslant}

\newcommand{\interior}[1]{%
	{\kern0pt#1}^{\mathrm{o}}%
}
\newtheorem{innercustomgeneric}{\customgenericname}
\providecommand{\customgenericname}{}
\newcommand{\newcustomtheorem}[2]{%
	\newenvironment{#1}[1]
	{%
		\renewcommand\customgenericname{#2}%
		\renewcommand\theinnercustomgeneric{##1}%
		\innercustomgeneric
	}
	{\endinnercustomgeneric}
}

\newcustomtheorem{customthm}{Theorem}

\numberwithin{equation}{section}

\title{Ax--Schanuel for the Drinfeld $j$-function}
\author{Gal Binyamini}
\address{Department of Mathematics, Weizmann Institute of Science, Israel}
\email{gal.binyamini@weizmann.ac.il}
\author{Dmitry Novikov}
\address{Department of Mathematics, Weizmann Institute of Science, Israel}
\email{dmitry.novikov@weizmann.ac.il}
\author{Francesco Maria Saettone}
\address{Department of Mathematics, Weizmann Institute of Science, Israel}
\email{francesco.saettone@weizmann.ac.il}

\begin{document}
	
\begin{abstract}
    We prove an analog of the Ax--Schanuel theorem for the Drinfeld $j$-function in odd characteristic. Roughly speaking, if the graph of $\boldsymbol j\colon\Omega^n\rightarrow\A^n_{\C_\infty}$ and its derivatives has an atypical intersection $\mathcal{V}$ with an algebraic variety, then $\mathcal{V}$ projects to a weakly-special subvariety in $\A^n_{\C_\infty}$.

    More generally, we prove a positive characteristic analog of the differential Galois theoretic Ax--Schanuel theorem of Bl\'azquez-Sanz, Casale, Freitag and Nagloo. Our main theorem for $\boldsymbol j$ follows by applying this result to a suitable foliation. A theorem of Pink allows us to conclude that the special varieties in the sense of Bl\'azquez-Sanz et al. in this context agree with the classical weakly-special subvarieties in the Drinfeld sense.    
\end{abstract}	

  \maketitle
  \setcounter{tocdepth}{2}
	\tableofcontents
    
\section{Introduction}

\medskip
\noindent\textbf{Functional transcendence.}\quad
A functional transcendence theorem governs the ways in which an algebraic variety can meet the graph of a transcendental uniformisation \emph{atypically}, i.e., in a larger dimension  than a naive count  predicts. The prototype is Ax's theorem for the complex exponential \cite{ax}: coordinates together with their exponentials have transcendence degree exceeding the rank of their differentials unless the coordinates obey a linear relation over the constants; geometrically, an algebraic variety meets the graph of $\exp$ atypically only along cosets of algebraic subgroups. For the elliptic modular function the corresponding result was established by Pila and Tsimerman \cite{pt}: the $j$ function, recorded together with its first two derivatives, is exactly as algebraically independent as the modular relations allow. For general Shimura varieties, Mok, Pila and Tsimerman \cite{pt-shimura} identified the sole source of atypicality with the weakly-special subvarieties. Theorems of this kind provide the functional transcendence input at the foundation of the Pila--Zannier method: an Ax--Schanuel statement is precisely what turns a counting estimate into geometry, and its strength dictates how far into the Zilber--Pink hierarchy the method can climb \cite{pilaAO,habegger-pila}.
 
\medskip
\noindent\textbf{Two approaches, and positive characteristic.}\quad
In characteristic zero, two routes to Ax--Schanuel theorems have proved especially effective. The first, developed by Pila--Tsimerman for the modular $j$-function \cite{pt} and by Mok--Pila--Tsimerman for Shimura varieties \cite{pt-shimura}, is $o$-minimal: the restriction of the uniformisation to a fundamental domain is definable in a tame (real) geometric structure, and the Pila--Wilkie counting theorem, applied to the orbit of the monodromy group, forces an atypical intersection to be stabilised by enough elements of the arithmetic group. The second, introduced by Bl\'azquez-Sanz, Casale, Freitag and Nagloo \cite{axschanuel}, is differential geometric: the uniformising differential equation is encoded as a flat connection on a principal bundle, atypical intersections are transported into the transverse direction of the horizontal foliation, and a sparsity theorem for algebraic subgroups---ultimately Lie theory in the semisimple group---confines the resulting transverse germ to a proper algebraic subgroup, over which the differential Galois group of the connection drops.
 
The present paper carries the second strategy into positive characteristic. A straightforward adaptation would find serious obstructions on both sides of the argument. The Drinfeld upper half-plane is a rigid-analytic space over $\C_\infty$, so the $o$-minimal approach has no immediate analog. In fact, the $o$-minimal approach relies on the real geometry underlying complex uniformisation: after identifying $\C$ with $\R^2$, the classical $j$-function restricted to an entire fundamental domain is definable in $\R_{\mathrm{an},\exp}$, as in \cite{pilaAO}. No analogous theory over $\C_\infty$ has yet been investigated. Although the rigid-analytic Pila--Wilkie theorem of Binyamini--Kato \cite{bk} provides an alternative counting tool, its affinoid setting does not by itself replace this global definability input (indeed, the Drinfeld fundamental domain  is not affinoid). On the differential side, ordinary derivations degenerate in characteristic $p$ (they annihilate all $p$-th powers, even nonconstant ones) and do not generate the differential algebra of the Ramanujan system in the Drinfeld setting. Even deeper than either obstruction is the failure of the Lie-theoretic mechanism behind sparsity: in characteristic $p$ an analytic subgroup germ is not governed by its tangent space, and sparsity may genuinely fail: indeed $\G_{a,\C_\infty}^{2}$ contains one-dimensional analytic subgroup germs that are Zariski dense (as shown in Example~\ref{eg}). We therefore replace the flat connection by the \emph{Hasse--Schmidt foliation} defined by the hyperderivatives of Drinfeld quasi-modular forms, the differential Galois group by the stabiliser of the Zariski closure of a Hasse--Schmidt leaf, and Lie-theoretic sparsity by a \emph{$p$-sparsity} theorem for $\pgl_{2,\C_\infty}^{\,n}$. With these substitutions the ideas of \cite{axschanuel} admit a positive characteristic extension.
 
\medskip
\noindent\textbf{The Drinfeld $j$-function.}\quad
By now, a well-known ``dictionary'' shows that  arithmetic of $F=\F_q(T)$ runs parallel to that of $\Q$, with $A=\F_q[T]$ in the role of $\Z$ and the completion $\C_\infty$ of an algebraic closure of $F_\infty=\F_q(\!(1/T)\!)$ in the role of $\C$. Rank-$2$ Drinfeld $A$-modules can be thought of as analogs of elliptic curves, and their rigid-analytic uniformisation presents each of them as $\C_\infty/(Az+A)$ for a point $z$ of the Drinfeld upper half-plane
\[
\Omega=\p^1(\C_\infty)-\p^1(F_\infty),
\]
the non-archimedean counterpart of the complex upper half-plane, now a rigid-analytic space. The modular invariant $j(z)=g(z)^{q+1}/\Delta(z)$ is rigid-analytic and $\gl_2(A)$-invariant, and identifies the coarse moduli line with the affine line, $Y(1)_{\C_\infty}\simeq\A^1_{\C_\infty}$; coordinatewise it gives the uniformisation
\[
\boldsymbol j\colon\Omega^n\rightarrow\A^n_{\C_\infty}
\]
of the $n$-fold self-product.
 
For us the decisive feature of $j$ is  its hyperdifferential algebra. The classical $j$ satisfies a third-order algebraic differential equation (equivalently, the Ramanujan system tying $E_2,E_4,E_6$ to their derivatives) and it is precisely this equation that drives the geometric proof of modular Ax--Schanuel. The Drinfeld $j$ obeys a rather faithful analog in which hyperderivatives replace derivatives. Let $g$ and $h$ be the Drinfeld modular forms of weights $q-1$ and $q+1$, normalised so that $\Delta=-h^{q-1}$, and let $E$ be the false Eisenstein series, the analog of the classical quasi-modular $E_2$. As studied in \cite{gekelerj}, writing $D_1$ for the first (normalised) hyperderivative, one has what we call the Drinfeld--Ramanujan system
\[
D_1E=E^2,\qquad D_1g=-(Eg+h),\qquad D_1h=Eh,
\]
and, by the theorem of Bosser and Pellarin \cite{bp}, the ring $\C_\infty[E,g,h]$ is stable under every hyperderivative $D_a$, while $E,g,h$ are algebraically independent over $\C_\infty$. This closed hyperdifferential system is the structure on which the whole argument turns; it is the positive characteristic replacement for the flat connection.
 
\medskip
\noindent\textbf{Main result.}\quad
The natural Ax--Schanuel statement records, alongside each $j(z_i)$, its first two hyperderivatives; we phrase it in the coordinates of the Ramanujan system. Let $\widetilde{\caly}:=\p^1_{\C_\infty}\times\A^3_{\C_\infty}$, with coordinates $(z,E,g,h)$; let $\pi_j^n\colon\widetilde{\caly}^{\,n}\dashrightarrow\A^n_{\C_\infty}$ be the rational map induced coordinatewise by the $j$-invariant; and let $\lcal\subset\widetilde{\caly}^{\,n}$ be a product Ramanujan leaf, the analytic leaf described by the uniformisation (Section~3).
 
\begin{customthm}{A}
Assume that $q$ is odd. Let $V_0\subset\widetilde{\caly}^{\,n}$ be an irreducible algebraic subvariety, let $\lcal$ be a product Ramanujan leaf, and let
$\calv\subset V_0\cap\lcal$ be an irreducible analytic subvariety. Set
$V:=\overline{\calv}^{\,\mathrm{Zar}}\subset V_0$. If
\[
\dim V<\dim\calv+3n
\]
then $\overline{\pi_j^n(V)}^{\,\mathrm{Zar}}\subset\A^n_{\C_\infty}$ is contained in
a proper weakly-special subvariety.
\end{customthm}

\begin{rmk}
The hypothesis that $q$ is odd is unfortunately required at several points of the proof. Here are some of them.
It ensures that the group scheme $\boldsymbol\mu_2$ acts trivially on the Ramanujan phase space, so that the action descends to $\pgl_2$, and that the degree $2$ phase space quotient \eqref{e:mu2-quotient} is \'etale.
In characteristic $2$, the central kernel is nonreduced and acts nontrivially on $g$ and $h$, preventing the former descent.
Moreover, the proof of $p$-sparsity uses the simplicity of
$\Lie(\pgl_2)$, which fails in characteristic $2$
(Lemmata~\ref{l:frobenius-goursat} and~\ref{l:pgl-p-sparse}).
On the arithmetic side, we also use that
$\Sl_2(F_\gotp)\rightarrow\pgl_2(F_\gotp)$ has open image
(Corollary~\ref{c:adelic-open-finite-index}).

Despite these obstructions, we believe that the even characteristic case should hold as well.
\end{rmk}
 
The exceptional loci are the expected ones, the Drinfeld counterparts of the weakly-special subvarieties of the classical modular Ax--Schanuel theorem. For $\gamma\in\gl_2(F)$, let $T_\gamma\subset Y(1)^2$ be the Zariski closure of the rigid-analytic curve $z\mapsto\bigl(j(z),j(\gamma z)\bigr)$; these are the Hecke correspondences on the Drinfeld modular curve, and they record isogeny relations between rank-$2$ Drinfeld modules. A \emph{weakly-special subvariety} of $Y(1)^n\simeq\A^n_{\C_\infty}$ is an irreducible component of a locus obtained by imposing finitely many equations $x_i=a_i$, with $a_i\in\C_\infty$, and finitely many relations $(x_i,x_j)\in T_\gamma$. The alternatives offered by Theorem~A coincide with the classical ones: constant coordinates and isogenies between the corresponding rank-$2$ Drinfeld modules.
 
Since $3n=\dim_{\C_\infty}\pgl_{2,\C_\infty}^{\,n}$, the hypothesis $\dim V<\dim\calv+3n$ is the classical Ax--Schanuel atypicality condition: the intersection of $V$ with the Ramanujan leaf exceeds its expected dimension $\dim V-3n$, the defect being measured against the dimension of the symmetry group of the uniformisation. On the open locus where $jD_1j\neq0$, the Ramanujan coordinates recover the first two hyperderivatives of $j$,
\[
\C_\infty\Bigl(E,\tfrac hg,\,gh\Bigr)=\C_\infty\bigl(j,\,D_1j,\,D_1^2j\bigr).
\]
These are the coordinates on the
$\boldsymbol\mu_2$-quotient by $(g,h)\mapsto(-g,-h)$; see Section~3.
Thus Theorem~A is an Ax--Schanuel theorem \emph{with derivatives} for the Drinfeld $j$-function, in its natural coordinates. Two special cases orient the reader. For $n=1$ the theorem amounts to the Zariski density of the lifted Ramanujan leaf in $\widetilde{\caly}$, which essentially amounts to the algebraic independence result of Bosser and Pellarin \cite{bp}, which indeed enters the proof (Lemma~\ref{l:basic-lifted-leaf-dense}); the substance of Theorem~A lies at $n\ge2$, where every relation across coordinates is forced to be an isogeny. At the other extreme, choosing $V_0$ to constrain only the $z$-coordinates, leaving the Ramanujan coordinates free, yields the Ax--Lindemann theorem for $\boldsymbol j$, strengthening the hyperbolic Ax--Lindemann theorem of our previous work \cite{bns}.
 
Formally, Theorem~A is the juxtaposition of two results: the abstract Hasse--Schmidt Ax--Schanuel theorem (Theorem~\ref{t:ax-schanuel}), which places $\overline{\pi_j^n(V)}^{\,\mathrm{Zar}}$ inside a proper \emph{HS-special} subvariety, i.e., a maximal irreducible subvariety over which the Hasse--Schmidt Galois group is a proper subgroup of $\pgl_{2,\C_\infty}^{\,n}$, and the result identifying HS-special subvarieties with weakly-special ones (Theorem~\ref{t:HS-weakly}). The first half is more geometric, the second more arithmetic in nature; we now describe each, signposting the points at which characteristic $p$ forces the argument off the characteristic zero path.

\medskip
\noindent\textbf{Strategy of the proof.}\quad
The geometric part of the proof is inspired by the ideas of
Bl\'azquez-Sanz--Casale--Freitag--Nagloo \cite{axschanuel}. Its broad
architecture is the same, but the argument by which it is carried out is
substantially new. After passing to the finite $\boldsymbol\mu_2$-quotient
of the Ramanujan phase space and making an equivariant birational change of
coordinates, we work, over a smooth open $X^\circ$ of the $j$-image $X$, on
a principal homogeneous space
\[
P_X=X^\circ\times G\rightarrow X^\circ
\]
where $G:=\pgl_{2,\C_\infty}^{\,n}$. The Hasse--Schmidt leaves on $P_X$
project locally isomorphically to $X^\circ$, and the $G$-coordinate, in a horizontal trivialization, therefore
records  the geometry transverse to the foliation. An atypical
intersection $\calv\subset V\cap\lcal$ gives, after translation to the
identity, a rigid-analytic image germ
\[
\Sigma\subset(G^{\mathrm{an}},e)
\]
with
\[
\dim\Sigma\le\dim V-\dim\calv<\dim G.
\]
The global d\'evissage is then to show that $\Sigma$ is a subgroup germ, to
place it in a proper algebraic subgroup $H\subsetneq G$, and to pass to the
quotient by $H$. The quotient produces a proper Hasse--Schmidt invariant
subvariety, inside which a minimal invariant subvariety has
Hasse--Schmidt Galois group strictly smaller than $G$. It follows that $X$
is HS-non-generic and hence lies in a proper HS-special subvariety.

The resemblance with \cite{axschanuel} is thus architectural rather than
technical. In characteristic zero, a connection form captures the
transverse geometry to first order and converts it into a proper Lie
subalgebra $\goth\subsetneq\Lie(G)$. Neither part of this mechanism survives
in the present setting: ordinary derivations do not encode the Drinfeld
Ramanujan system, and in characteristic $p$ an analytic subgroup germ is
not determined by its tangent space. Our proof therefore replaces the
first-order Lie algebra $\goth$ by the entire transverse germ $\Sigma$ and
establishes its group structure directly, to all orders.
More precisely, after shrinking $\calv$, the normalized transverse image
germs $\Sigma_v(V)$ agree with $\Sigma$ for $v\in\calv$. 
Using two copies of an enlarged Hasse--Schmidt foliation (as defined in \ref{ss:enlarged}), we express containment of transverse image germs as a condition on the
leaf intersection dimension.
Proposition~\ref{p:atypical-closed} then shows that, for
fixed $v_0\in V^\circ$, the locus
$\{v\in V^\circ:\Sigma_v(V)\subseteq\Sigma_{v_0}(V)\}$
is Zariski closed in a fixed open subset $V^\circ\subset V$
(Lemma~\ref{l:translated-image-jets-algebraic}).
The Zariski density of $\calv$ and equality of dimensions
therefore give $\Sigma_v(V)=\Sigma$ throughout $V^\circ$.
Changing the normalization of the transverse projection and applying the rigid analytic identity principle then gives stability under right division, hence the subgroup germ property (Lemma~\ref{l:translated-germ-constant}). This is the counterpart of a crucial step in
\cite[Lemma~3.10]{axschanuel}, where the image of the restricted connection form is shown to be generically constant and to form a proper Lie subalgebra of $\Lie(G)$.

Another novelty is our $p$-sparsity theorem for $G=\pgl_{2,\C_\infty}^{\,n}$. Since analytic subgroup germs in
characteristic $p$ are not controlled by their tangent spaces, its proof
cannot proceed through Lie theory. Instead, we establish a rigidity
theorem directly for the germs themselves. In two factors, a proper
subgroup germ with dominant coordinate projections is forced to lie in a
Frobenius-twisted graph; Goursat's lemma then extends the conclusion to
arbitrary $n$. This replaces the Lie-theoretic sparsity input of
\cite{axschanuel} by a general characteristic-$p$ argument.

Once $\Sigma\subset(H^{\mathrm{an}},e)$ has been obtained, the original
d\'evissage resumes. Let
\[
\rho\colon P_X\rightarrow H\backslash P_X
\]
be the quotient map and set
$Y:=\overline{\rho(V)}^{\,\mathrm{Zar}}$.
The germ of $Y$ at a general point of $\rho(\calv)$ is contained in a
quotient leaf. Since $Y$ dominates $X^\circ$ and the quotient leaves have
dimension $\dim X$, one obtains $\dim Y=\dim X$; the Zariski closedness of
the atypicality loci then implies that $Y$ is Hasse--Schmidt invariant.
Therefore $\rho^{-1}(Y)$ is a proper invariant subvariety of $P_X$.
A minimal invariant subvariety contained in $\rho^{-1}(Y)$ has
Hasse--Schmidt Galois group strictly smaller than $G$, and therefore $X$ is
contained in a proper HS-special subvariety. This proves
Theorem~\ref{t:ax-schanuel}.

The proof is therefore not a formal characteristic-$p$ transcription of
\cite{axschanuel}. It preserves the conceptual shape of the characteristic zero strategy, but replaces both of its decisive engines by new ones: the connection form and Lie-algebra arguments by an analysis of transverse image germs using enlarged Hasse--Schmidt foliations, and ordinary sparsity by the analogous notion of $p$-sparsity, which takes the Frobenius into account.

\medskip
\noindent\textbf{From HS-special to weakly-special.}\quad
Let $X\subset\A^n_{\C_\infty}$ be HS-non-generic, and let $\cala$ be an irreducible analytic component of $\boldsymbol j^{-1}(X)$. The stabiliser $\Delta_\cala^+\subset\pgl_2(A)^n$ of $\cala$, which is the analytic monodromy of $X$, preserves the leaf attached to $\cala$, hence lands in the Hasse--Schmidt Galois group and has proper Zariski closure $H_\cala\subsetneq\pgl_{2,\C_\infty}^{\,n}$ (Lemma~\ref{l:HS-monodromy-proper}). Against this properness we play an openness: for every coordinate on which $X$ is nonconstant, the corresponding projection of the monodromy is \emph{adelically} open in $\Sl_2(\mathbf A_F^\infty)$, which we prove by computing the monodromy of the congruence tower $Y(\gotn\gota)\rightarrow Y(\gotn)$ through the \'etale fundamental group and strong approximation (Lemma~\ref{l:coordinate-adelic-open-corrected}). Each one-coordinate projection of $H_\cala$ is therefore $\pgl_{2,\C_\infty}$, and Goursat's lemma produces a pair $i\neq j$ of coordinates with $\pr_{ij}(H_\cala)\subsetneq\pgl_{2,\C_\infty}^{\,2}$. For this global algebraic subgroup, the abstract Goursat argument and Van der Waerden's classification of the automorphisms of $\pgl_2(\C_\infty)$ give a graph twisted by a field automorphism; algebraicity and Lemma~\ref{l:analytic-field-automorphism} force that automorphism to be a power of the Frobenius. A direct  descent of the conjugating element then gives a relation $\pr_j(\gamma)=\delta_0\,\Frob^m(\pr_i(\gamma))\,\delta_0^{-1}$ on a finite index subgroup, with $\delta_0\in\pgl_2(F)$ (Lemmata~\ref{l:F-frobenius-graph} and~\ref{l:step3-frobenius-relation}).
 
Two steps remain, and both are arithmetic. The Frobenius twist is eliminated by a congruence level point counting at a prime $\gotp\neq \infty$: writing $q_\gotp:=|A/\gotp|$, adelic openness gives $q_\gotp^{3e-O(1)}$ residue classes in the $j$-th coordinate modulo $\gotp^e$, whereas a $p^m$-Frobenius graph gives at most $q_\gotp^{3e/p^m+O(1)}$ such classes; hence $m=0$ (Lemma~\ref{l:no-positive-frobenius-twist} and Lemma~\ref{p:projected-pair-isogeny}). And the untwisted graph is converted into geometry: choosing a prime $\gotp$ not dividing $\gotn$ and spreading the situation out over a finitely generated field, we compare the monodromy of the $\gotp$-power tower with the projective $\gotp$-adic Tate representations of the two coordinate Drinfeld modules (Lemma~\ref{l:tate-comparison-identity-components}); the graph relation makes these representations conjugate (Corollary~\ref{l:graph-monodromy-projective-tate}), and the Tate conjecture for Drinfeld modules, due to Taguchi and Tamagawa \cite{taguchi,tamagawa}, combined with \cite{pinkmt}, converts conjugacy into an isogeny---that is, into a Hecke correspondence $T_\gamma$ containing the $(i,j)$-image of $X$ (Lemma~\ref{l:projective-tate-conjugacy-isogeny-corrected}, Lemma~\ref{l:isogeny-hecke}). This proves Theorem~\ref{t:HS-weakly}.
 
\medskip
\noindent\textbf{Characteristic-$p$ phenomena.}\quad
It may be interesting to single out at least two features of the argument, for neither has a characteristic-zero counterpart. The first is the fragility of sparsity itself. Over $\C$, an analytic subgroup germ of a linear algebraic group is faithfully recorded by its Lie algebra; over $\C_\infty$, additive power series produce Zariski dense proper subgroup germs, so that $p$-sparsity is false in general and holds for $\pgl_{2}^{\,n}$ by virtue of a genuinely non-infinitesimal argument. The second is the persistence of the Frobenius: at every stage where characteristic zero sees only inner automorphisms, characteristic $p$ sees the Frobenius twists $\Inn(\delta)\circ\Frob^m$---among the ``exceptional'' subgroups of the sparsity theorem, and again among the possible relations between coordinate monodromies. In our previous work \cite{bns} the Frobenius twist was excluded by a counting argument: a finite index subgroup of $\F_q[T]$ cannot fit inside the sparser subring $\F_q[T^{q^m}]$. Here the same sparsity is measured in finite congruence quotients. For $q_\gotp:=|A/\gotp|$, adelic openness supplies $q_\gotp^{3e-O(1)}$ monodromy classes modulo $\gotp^e$, whereas a $p^m$-Frobenius graph supplies at most $q_\gotp^{3e/p^m+O(1)}$ classes. Thus a positive Frobenius twist is too ``sparse'' to carry the open local monodromy. The two mechanisms are avatars of a single principle, the ``sparsity'' of $p$-th powers, surfacing once at the level of points and once at the level of groups.
 
\medskip
\noindent\textbf{Earlier works.}\quad
As already mentioned, in characteristic zero, the modular Ax--Schanuel theorem is due to Pila--Tsimerman \cite{pt}, its extension to Shimura varieties to Mok--Pila--Tsimerman \cite{pt-shimura}, and the differential proof that serves as our template to Bl\'azquez-Sanz--Casale--Freitag--Nagloo \cite{axschanuel}. In positive characteristic, functional transcendence of Ax--Schanuel type has so far been available in the linear regime: Kowalski \cite{kow} obtains such statements by means of iterative Hasse--Schmidt derivations (but not in the Drinfeld setting), and the companion paper \cite{bns} proves Ax--Lindemann theorems, for Drinfeld exponentials and for $\boldsymbol j$, by direct point counting; Theorem~A strengthens the latter to a full Ax--Schanuel with derivatives, and, to our knowledge, it is the first Ax--Schanuel theorem in positive characteristic beyond the linear setting. It also supplies the transcendence input which, in the classical Pila--Zannier scheme, underlies statements of Andr\'e--Oort type, for products of Drinfeld modular curves a theorem of Breuer \cite{breuerCM}.

On the arithmetic side, our debt is to Pink's study of the adelic and Tate module monodromy of Drinfeld modules \cite{pinkmt,breuerpink}, which we use to convert the resulting monodromy graph into an isogeny.

\medskip
\noindent\textbf{Future directions.}\quad
Theorem~A supplies the functional transcendence input needed to carry the
Pila--Zannier strategy beyond Andr\'e--Oort. In \cite{bns} we developed a common rigid analytic framework for the two basic Drinfeld uniformisations, the Drinfeld exponential and the $j$-function: combining
Ax--Lindemann with the rigid-analytic Pila--Wilkie theorem of Binyamini--Kato, we proved Manin--Mumford for a product of two Drinfeld modules of equal rank and Andr\'e--Oort for a product of two Drinfeld
modular curves. In Section~\ref{ss:alao}, we deduce the hyperbolic Ax--Lindemann in arbitrary dimension, extending \cite[Theorem~3.11]{bns}, which allows to recover Andr\'e--Oort theorem for arbitrary products of Drinfeld modular curves (indeed in odd characteristic). Together with the finer height estimates required by the higher strata, Theorem~A also points towards analogs of  Zilber--Pink in $Y(1)^n$.

On the Drinfeld modules side, the parallel next step is an Ax--Schanuel theorem for the Drinfeld exponential, which will hopefully be studied by the authors in a forthcoming work. Combined with the counting framework of \cite{bns} and suitable height estimates, such a theorem would open a route  to higher unlikely intersections in products of
Drinfeld modules. Already for the Carlitz module, it would provide a crucial functional transcendence input toward the (full) analog of the Zilber--Pink conjecture of Brownawell and Masser \cite{brownawell-masser} for powers of the Carlitz module.

\medskip
\noindent\textbf{Organization of the paper.}\quad
Section~2 collects the background on Drinfeld modular curves: the Drinfeld upper half-plane, level structures and the congruence tower, Conrad's theory of irreducible components of rigid spaces---our substitute for elementary complex analytic component arguments---and the weakly-special subvarieties of $Y(1)^n$. In Section~3 the theorem is stated and proved. Its first part develops Hasse--Schmidt foliations on algebraic varieties, proves the Zariski closedness of the atypicality loci (Proposition~\ref{p:atypical-closed}), constructs the Drinfeld--Ramanujan foliation together with its $\pgl_2$-equivariance and the $\boldsymbol\mu_2$-quotient and principal homogeneous space model, introduces the Hasse--Schmidt Galois group, and establishes the $p$-sparsity theorem. The second part carries out the d\'evissage  and proves the abstract Ax--Schanuel theorem (Theorem~\ref{t:ax-schanuel}). The final part is the arithmetic comparison of HS-special with weakly-special subvarieties (Theorem~\ref{t:HS-weakly}): analytic monodromy, congruence towers and adelic openness, the Tate module comparison, and the two-coordinate criterion.

\subsubsection*{Conventions}
By a variety we shall mean an integral, separated scheme of finite type. Unless explicitly stated otherwise, an intersection of algebraic
subvarieties of a fixed variety is endowed with its reduced induced
structure. The same convention applies to intersections of algebraic
subgroups. Fiber products, kernels, stabilizer group schemes, torsion
group schemes, and scheme-theoretic images retain their usual
scheme-theoretic meaning.

\subsubsection*{Acknowledgements}
This work was partially done while the authors were at the Institute for Advanced Study in Princeton, and they would like to thank the institute for its hospitality and for providing excellent working conditions. G.B. was supported by the Marvin V. and Beverly J. Mielke Endowed Fund and the Infosys Member Fund, and D.N. was supported by the Kovner Member Fund. G.B. and F.M.S. were also supported by the European Union (ERC, SharpOS, 101087910) and by the Israel Science Foundation (grant No. 2067/23). D.N. was also supported by the Israel Science Foundation grant 1167/17 and by Minerva grant 714141. 

\subsubsection*{AI statement}
All of the new ideas of this work are due to the authors. {\em Chat GPT-6 Astra} was used to proofread the manuscript and correct small technical errors.

\section{The Drinfeld \texorpdfstring{$j$}{j}-function}

We recall the basic notation for Drinfeld modular curves and for the Drinfeld modular forms used below.  Our conventions are standard;
see \cite{drinfellipt,gekeler,pap} for Drinfeld modules and Drinfeld
modular curves, and \cite{gekelerj,gekeler2,bp} for Drinfeld modular and quasi-modular forms.

\subsection{Drinfeld modular curves}

\subsubsection{The Drinfeld upper half-plane}

Let
\[
A:=\F_q[T],
\qquad
F:=\F_q(T),
\]
and let $\infty$ be the place of $F$ corresponding to $1/T$.  We write
$F_\infty$ for the completion of $F$ at $\infty$ and $\C_\infty$ for the
completion of an algebraic closure of $F_\infty$.  The Drinfeld upper
half-plane is the rigid analytic space
\[
\Omega:=\p^1(\C_\infty)-\p^1(F_\infty).
\]
The group $\gl_2(F_\infty)$ acts on $\Omega$ by M\"obius action.  For $\gamma=
\begin{bmatrix}
a&b\\ c&d
\end{bmatrix}
\in \gl_2(F_\infty)$,
we write
$\gamma z:=\frac{az+b}{cz+d}$ and  $\lambda_\gamma(z):=cz+d$.
For $z\in\Omega$ one has $\lambda_\gamma(z)\neq0$.
For $z\in\Omega$, the lattice
\[
\Lambda_z:=Az+A\subset \C_\infty
\]
defines an isomorphism class of rank-$2$ Drinfeld $A$-modules over $\C_\infty$. Fix once and for all a fundamental period
$\widetilde\pi\in\C_\infty^\times$ of the Carlitz module $T\mapsto T+\uptau$. We take $\varphi^z$ to be the representative
associated with the lattice $\widetilde\pi\Lambda_z$ and write
\[
\varphi_T^z=
T+g(z)\uptau+\Delta(z)\uptau^2
\qquad
\text{for}
\quad
\uptau(x)=x^q.
\]
This gives rigid holomorphic functions $g$ and $\Delta$ on $\Omega$.
The corresponding Drinfeld $j$-invariant is
\[
j(z):=\frac{g(z)^{q+1}}{\Delta(z)}.
\]
\subsubsection{Drinfeld modular curves}

Let $\Gamma(1):=\pgl_2(A)$.
The quotient
\[
\Gamma(1)\backslash \Omega
\]
is the analytification of the coarse moduli curve of rank-$2$ Drinfeld
$A$-modules.  This curve is denoted by $Y(1)$.
The function $j$ is invariant under $\Gamma(1)$ and so it induces the following rigid-analytic isomorphism
\[
Y(1)_{\C_\infty}\simeq \A^1_{\C_\infty}.
\]
Thus, throughout the paper, we identify $Y(1)_{\C_\infty}$ with
$\A^1_{\C_\infty}$ using the coordinate $j$.

For $n\ge1$ we write
\[
\boldsymbol j\colon\Omega^n\rightarrow \A^n_{\C_\infty},
\qquad
(z_1,\ldots,z_n)\mapsto
\bigl(j(z_1),\ldots,j(z_n)\bigr).
\]
Under the identification $Y(1)_{\C_\infty}\simeq\A^1_{\C_\infty}$, this is
the analytic uniformization map of $Y(1)^n$.

\subsubsection{Level structures}

Let $\gotn\subset A$ be a non-zero proper ideal.  We denote by $\widetilde\Gamma(\gotn):=
\ker\bigl(\Sl_2(A)\rightarrow \Sl_2(A/\gotn)\bigr)$
the principal congruence subgroup of level $\gotn$, and by
\[
\Gamma(\gotn)\subset \pgl_2(A)
\]
its image.  We shall also use $\Gamma^+$
for the image of $\Sl_2(A)$ in $\pgl_2(A)$.  Thus
$\Gamma(\gotn)\subset\Gamma^+$ is a finite index normal subgroup.

We recall the moduli interpretation of this level.  Let $\varphi$ be a
rank-$2$ Drinfeld $A$-module over an $\F_q$-scheme $S$.
We write $\varphi[\gotn]:=
\bigcap_{a\in\gotn}\ker(\varphi_a)$
for the $\gotn$-torsion group scheme.  Since $A=\F_q[T]$ is a principal ideal domain, if $\gotn=(N)$ then
$\varphi[\gotn]=\ker(\varphi_N)$.

A full level-$\gotn$ structure on $\varphi$ is an $A$-linear map
\[
\alpha\colon (\gotn^{-1}/A)^2\rightarrow \varphi(S)
\]
such that, for every prime ideal $\gotp\supset\gotn$, the restriction of
$\alpha$ to $(\gotp^{-1}/A)^2$ satisfies
\[
\varphi[\gotp]
=
\sum_{x\in(\gotp^{-1}/A)^2}[\alpha(x)]
\]
as Cartier divisors on $\mathbb G_{a,S}$.  Here $[\alpha(x)]$ denotes the
section of $\mathbb G_{a,S}$ defined by the point $\alpha(x)$.

Over $\C_\infty$, the group scheme $\varphi[\gotn]$ is finite \'etale of
rank $|A/\gotn|^2$.  In this case the preceding divisor condition is
equivalent to saying that $\alpha$ induces an isomorphism of finite
$A$-module schemes
$((\gotn^{-1}/A)^2)_{\C_\infty}
\stackrel{\sim}{\rightarrow}
\varphi[\gotn]$.

For later use, we also recall the associated Tate modules; see
\cite[Section~1]{pinkmt}. Let $\varphi$ be a rank-$2$ Drinfeld $A$-module of generic characteristic over a field $K\supset F$, and let $\gotp\ne\infty$ be a prime of $A$. Its $\gotp$-adic {\em Tate module} is
\[
T_\gotp(\varphi):=\varprojlim_e\varphi[\gotp^e](K^{\sep})
\]
a free $A_\gotp$-module of rank $2$, carrying a continuous action of
$G_K=\Gal(K^{\sep}/K)$. We write
\[
V_\gotp(\varphi):=T_\gotp(\varphi)\otimes_{A_\gotp}F_\gotp
\]
for the rational Tate module and
$\rho_{\varphi,\gotp}\colon G_K\rightarrow\gl_2(F_\gotp)$ for the resulting
representation, defined up to conjugacy. Its projectivization will be denoted by $\bar\rho_{\varphi,\gotp}\colon G_K\rightarrow\pgl_2(F_\gotp)$.
Compatible full level-$\gotp^e$ structures amount to compatible bases of
the finite quotients of $T_\gotp(\varphi)$.

We shall work with compatible geometric connected components of the
full-level modular curves. Fix once and for all a point
$z_0\in\Omega$.
For every non-zero proper ideal $\gotm\subset A$, let $Y(\gotm)$ denote the
geometric connected component of the full level-$\gotm$ Drinfeld modular
curve selected by $z_0$. We choose these components compatibly as $\gotm$
varies. Their rigid analytifications are identified with
\[
Y(\gotm)^{\mathrm{an}}
\simeq
\Gamma(\gotm)\backslash\Omega.
\]

Let
\[
\pi_\gotm\colon\Omega\rightarrow Y(\gotm)^{\mathrm{an}}
\]
be the quotient map. Forgetting the whole level structure gives a finite
morphism
\[
\nu_\gotm\colon Y(\gotm)\rightarrow Y(1).
\]
Under the preceding uniformizations,
\[
\nu_\gotm^{\mathrm{an}}\circ\pi_\gotm=j.
\]

More generally, if $\gotm'\subset\gotm$ are non-zero proper ideals, the
level-forgetting morphism between the compatible components is denoted by
\[
\nu_{\gotm',\gotm}\colon
Y(\gotm')\rightarrow Y(\gotm).
\]
These maps satisfy
$\nu_\gotm\circ\nu_{\gotm',\gotm}=\nu_{\gotm'}$
and, rigid analytically,
$\nu_{\gotm',\gotm}^{\mathrm{an}}\circ\pi_{\gotm'}=
\pi_\gotm$.

For $N\ge1$, we use the product notation
\[
\pi_\gotm^N
:=
\pi_\gotm\times\cdots\times\pi_\gotm
\colon
\Omega^N\rightarrow Y(\gotm)^{N,\mathrm{an}},
\]
and similarly $\nu_\gotm^N$ and $\nu_{\gotm',\gotm}^N$. Thus
\[
(\nu_\gotm^N)^{\mathrm{an}}\circ\pi_\gotm^N
=
\boldsymbol j
\]
and
\[
(\nu_{\gotm',\gotm}^N)^{\mathrm{an}}\circ\pi_{\gotm'}^N
=
\pi_\gotm^N.
\]

For every non-zero proper ideal $\gotm$, let
\[
\kappa_\gotm\colon
\widetilde\Gamma(\gotm)\rightarrow\Gamma(\gotm)
\]
be the natural projection.

We use the standard rigid quotient construction
$\Omega\rightarrow\Gamma(\gotm)\backslash\Omega$: the action is discontinuous, the
quotient map is admissible-locally trivial, and its fibers are the
$\Gamma(\gotm)$-orbits. See
Drinfeld~\cite[Proposition~6.2]{drinfellipt}
and Gekeler~\cite[Chapters~I--II]{gekeler}.

\begin{lem}\label{l:level-tower-galois}
Assume that $q$ is odd. For every non-zero proper ideal $\gotm\subset A$,
the group $\widetilde\Gamma(\gotm)$ acts freely on $\Omega$, the map
\[
\kappa_\gotm\colon
\widetilde\Gamma(\gotm)\rightarrow\Gamma(\gotm)
\]
is an isomorphism, and the quotient map
\[
\pi_\gotm\colon\Omega\rightarrow Y(\gotm)^{\mathrm{an}}
\]
is admissible-locally an isomorphism.
Moreover, if $\gotm'\subset\gotm$ are non-zero proper ideals, then
\[
\nu_{\gotm',\gotm}\colon Y(\gotm')\rightarrow Y(\gotm)
\]
is a connected finite \'etale Galois cover with group
\[
\Gamma(\gotm)/\Gamma(\gotm')
\simeq
\widetilde\Gamma(\gotm)/\widetilde\Gamma(\gotm').
\]
\end{lem}

\begin{proof}
The freeness of the action of the principal congruence subgroup
$\widetilde\Gamma(\gotm)$ on $\Omega$, for $\gotm\ne A$, is standard;
see \cite[Chapter~I]{gekeler}. The rigid-analytic uniformization and its
compatibility with level-forgetting morphisms are given in
\cite[Section~6]{drinfellipt} and \cite[Chapter~I]{gekeler}.

Since $q$ is odd, one has
$-I\notin\widetilde\Gamma(\gotm)$. Indeed, otherwise
$-I\equiv I\bmod{\gotm}$, so $2\in\gotm$, which is impossible because
$2\in A^\times$ and $\gotm$ is proper. Therefore
$\kappa_\gotm\colon
\widetilde\Gamma(\gotm)\rightarrow\Gamma(\gotm)$
is an isomorphism. The admissible-local description of the rigid quotient,
together with freeness, shows that $\pi_\gotm$ is admissible-locally an
isomorphism.

Now let $\gotm'\subset\gotm$. The subgroup
$\Gamma(\gotm')$ is normal in $\Gamma(\gotm)$. Hence the map between the
chosen analytic components is
$\Gamma(\gotm')\backslash\Omega
\rightarrow
\Gamma(\gotm)\backslash\Omega$.
It is a connected finite \'etale Galois cover with deck group
$\Gamma(\gotm)/\Gamma(\gotm')$. By compatibility of the chosen components,
this is the analytification of
$\nu_{\gotm',\gotm}\colon Y(\gotm')\rightarrow Y(\gotm)$.
Finally, the isomorphisms $\kappa_\gotm$ and $\kappa_{\gotm'}$ identify the
deck group with
$\widetilde\Gamma(\gotm)/\widetilde\Gamma(\gotm')$.
\end{proof}

\subsubsection{Rigid irreducible components over \texorpdfstring{$\C_\infty$}{C-infinity}}

We use Conrad's terminology for irreducible components of rigid analytic
spaces \cite{conrad-irred}. Throughout, rigid analytic spaces are taken over
the complete algebraically closed non-archimedean field $\C_\infty$. Let $X$
be a rigid analytic space over $\C_\infty$, and let
$\nu\colon \widetilde X\rightarrow X$
be the normalization of $X$. The irreducible components of $X$ are the
reduced analytic subsets
$\nu(\widetilde X_i)\subset X$,
where $\widetilde X_i$ runs through the connected components of
$\widetilde X$. We say that $X$ is irreducible if $X$ is non-empty and has a
unique irreducible component. An analytic subset $Z\subset X$ is irreducible
if it is irreducible with its reduced rigid analytic structure
\cite[Definition~2.2.2]{conrad-irred}.

We shall use the following consequences. First, if
$U\subset X$ is a non-empty admissible open, then every irreducible component
of $U$ is contained in a unique irreducible component of $X$, and the
intersection of an irreducible component of $X$ with $U$ is a possibly empty
union of irreducible components of $U$ \cite[Cor.~2.2.9]{conrad-irred}.
Second, if $X_0$ is a non-empty locally finite type $\C_\infty$-scheme, then
the irreducible components of $X_0^{\mathrm{an}}$ are  the
analytifications of the irreducible components of $X_0$
\cite[Theorem~2.3.1]{conrad-irred}. Third, if $X$ is normal, then its
irreducible components are its connected components. In particular, since
smooth rigid analytic spaces over $\C_\infty$ are normal, every non-empty
connected smooth rigid analytic space over $\C_\infty$ is irreducible.

Finally, if $x$ is a smooth point of a rigid analytic space $X$ over
$\C_\infty$, then $\calo_{X,x}$ is a regular local ring, hence a domain. Thus $\Spec(\calo_{X,x})$ is irreducible. By
\cite[p.~496]{conrad-irred}, this is
equivalent to saying that there is exactly one irreducible component of $X$
passing through $x$. This observation is used below whenever we pass between
smooth level branches and irreducible analytic components of preimages under
$\boldsymbol j$.

\begin{lem}
\label{l:smooth-branch-bridge}
Assume that $q$ is odd. Let
$X\subset Y(1)^N_{\C_\infty}$
be irreducible, let $\gotn\subset A$ be a non-zero proper ideal, and let
$B\subset(\nu_\gotn^N)^{-1}(X)$
be a non-empty smooth irreducible open subset. Consider $U:=(\pi_\gotn^N)^{-1}(B^{\mathrm{an}})$.
Then $U$ is a smooth admissible open subset of
$\boldsymbol j^{-1}(X^{\mathrm{an}})$. Every connected component
$\cala'$ of $U$ is irreducible and is contained in a unique irreducible
component
\[
\cala\subset\boldsymbol j^{-1}(X^{\mathrm{an}}).
\]
Moreover,
\[
\Stab_{\Gamma(\gotn)^N}(\cala')
\subset
\Stab_{(\Gamma^+)^N}(\cala).
\]
\end{lem}

\begin{proof}
Since $\boldsymbol j=(\nu_\gotn^N)^{\mathrm{an}}\circ\pi_\gotn^N$,
the space $U$ is an admissible open subset of
$\boldsymbol j^{-1}(X^{\mathrm{an}})$. By
Lemma~\ref{l:level-tower-galois}, $\pi_\gotn^N$ is
admissible-locally an isomorphism; hence $U$ is smooth.

Thus $U$ is normal, so its connected components are irreducible. By the
component results recalled above, each such component $\cala'$ is contained
in a unique irreducible component $\cala$ of
$\boldsymbol j^{-1}(X^{\mathrm{an}})$.

If $\gamma\in\Gamma(\gotn)^N$ stabilizes $\cala'$, then $\gamma\cala$ is
another irreducible component containing
$\gamma\cala'=\cala'$. Uniqueness gives $\gamma\cala=\cala$. Since
$\Gamma(\gotn)\subset\Gamma^+$, the claimed inclusion follows.
\end{proof}

\subsubsection{weakly-special subvarieties}

For $\gamma\in\gl_2(F)$, the analytic map
\[
\Omega\rightarrow Y(1)^2\simeq\A^2_{\C_\infty},
\qquad
z\mapsto \bigl(j(z),j(\gamma z)\bigr)
\]
has algebraic image.  We denote its Zariski closure by
\[
T_\gamma\subset Y(1)^2.
\]
These are the Hecke correspondences on the Drinfeld modular curve $Y(1)$.
Equivalently, they record isogeny relations between rank-$2$ Drinfeld
$A$-modules.

A {\em weakly-special subvariety} of
$Y(1)^n\simeq\A^n_{\C_\infty}$
is an irreducible component of a subvariety cut out by finitely many conditions of the following two kinds:
\[
x_i=a_i
\qquad
\text{for}
\quad
a_i\in\C_\infty
\]
and
\[
(x_i,x_j)\in T_\gamma
\qquad
\text{for}
\quad
\gamma\in\gl_2(F),\ \text{with }i\neq j.
\]

\begin{lem}
\label{l:isogeny-hecke}
Let $X\subset Y(1)^N_{\C_\infty}$ be irreducible. Suppose that, for some
$i\ne j$, the two generic rank-$2$ Drinfeld modules attached to the $i$-th and
$j$-th coordinate projections of $X$ are geometrically isogenous. Then $X$ is
contained in a proper weakly-special subvariety.
\end{lem}

\begin{proof}
After passing to sufficiently high level, take the level component above $X$ and
its generic point. A geometric isogeny between the two generic rank-$2$ Drinfeld
modules has finite $A$-degree. By the moduli interpretation of Drinfeld Hecke
correspondences, the generic point of $\pr_{ij}(X)$ lies on some
$T_\gamma\subset Y(1)^2$ with $\gamma\in\gl_2(F)$. Since $T_\gamma$ is Zariski
closed and $X$ is irreducible, we have 
$X\subset\pr_{ij}^{-1}(T_\gamma)$.
The irreducible component of $\pr_{ij}^{-1}(T_\gamma)$ containing $X$ is weakly
special by definition, and it is proper because $T_\gamma\ne Y(1)^2$.
\end{proof}

\section{Ax--Schanuel}

\subsection{Hasse--Schmidt structures}

\subsubsection{Hasse--Schmidt derivatives and jets}\label{s:hs-jets}

Let $U\subset \Omega$ be an admissible open, let
$f\in\calo_\Omega(U)$ be a rigid holomorphic function, and let $z\in U$.  The {\em Hasse--Schmidt derivatives}, or {\em hyperderivatives}, of $f$ are defined by 
\[
f(z+\varepsilon)
=\sum_{a\ge 0}(\dcal_a f)(z)\varepsilon^a
\]
for $|\varepsilon|_\infty$ sufficiently small.  Equivalently, if  $f(z)=\sum_{m\ge 0}a_mz^m$,
then $\dcal_a f(z)=\sum_{m\ge a}\binom{m}{a}a_mz^{m-a}$.

The linear operators $\{\dcal_a\}_{a\ge0}$ form an iterative derivation:
\begin{equation}\label{e:leibniz}
\dcal_0=\id,\qquad
\dcal_a(fg)=\sum_{b+c=a}(\dcal_bf)(\dcal_cg),
\qquad
\dcal_a\circ\dcal_b=\binom{a+b}{a}\dcal_{a+b}
\end{equation}
where the middle formula amounts to their Leibniz rule.
Moreover, the ring of holomorphic functions on $\Omega$ is closed under $\dcal_n$, as mentioned in \cite[p.~16]{bp}.

The $m$-th Hasse--Schmidt jet of $f$ at $z$ is the truncated  expansion
\[
\operatorname{jet}^m_z(f)
:=
\sum_{a=0}^{m}(\dcal_a f)(z)\varepsilon^a
\in \C_\infty[\varepsilon]/(\varepsilon^{m+1}).
\]
In particular, for $f\in\calo_{U,z}$, one has  $f\in\gotm_z^{m+1}$ if and only if $(\dcal_a f)(z)=0$ for every $0\le a\le m$, where $\gotm_z$ is the maximal ideal of $\calo_{U,z}$.

As in \cite{bp}, we shall use the normalized hyperderivatives
\[
D_a:=(-\widetilde\pi)^{-a}\dcal_a.
\]
For a tuple $F=(f_1,\ldots,f_s)\in\calo_\Omega(U)^s$
we write
\[
J_D^m(F)(z)
:=
\bigl((D_a f_i)(z)\bigr)_{
\substack{1\le i\le s\\0\le a\le m}}
\in \A_{\C_\infty}^{s(m+1)}.
\]

\subsubsection{Hasse--Schmidt foliation}\label{ss:hs-foliation}

Let $(R,\mathfrak m)$ be a Noetherian local ring, and let $I\subset R$ be an
ideal. For $k\ge 0$, set
\[
J_k(R):=R/\gotm^{k+1}
\]
and let $j_k\colon R\rightarrow J_k(R)$
be the canonical projection. We denote the length of $R/I$ by 
$\ell_R(R/I)$.

We record the following standard algebraic lemma.

\begin{lem}\label{l:multiplicity-local}
We have $\ell_R(R/I)>k$ if and only if $\ell_R\bigl(J_k(R)/j_k(I)\bigr)>k$.
\end{lem}

\begin{proof}
Since $J_k(R)/j_k(I)\simeq R/(I+\gotm^{k+1})$,
the ``if'' implication is immediate.

Conversely, suppose $\ell_R\bigl(R/(I+\gotm^{k+1})\bigr)\le k$. Let us set $A:=R/I$ and $\gotn:=\gotm A$.
Then
\[
A/\mathfrak n^{k+1}\simeq R/(I+\mathfrak m^{k+1}).
\]
If $\gotn^{k+1}\neq 0$, then the chain
$A\supset \gotn\supset \gotn^2\supset\cdots\supset \gotn^{k+1}$
has $k+1$ nonzero successive quotients $\gotn^i/\gotn^{i+1}$, for $0\le i\le k$. Indeed, if $\gotn^i/\gotn^{i+1}=0$ for some $i\le k$, then $\gotn^i=\gotn\gotn^i$, and Nakayama's lemma gives $\gotn^i=0$, contradicting $\gotn^{k+1}\neq0$. Hence $\ell_R(A/\gotn^{k+1})\ge k+1$,
contrary to the assumption. Therefore $\gotn^{k+1}=0$, i.e., $\gotm^{k+1}\subset I$. Thus
\[
R/I\simeq R/(I+\mathfrak m^{k+1})
\]
and consequently $\ell_R(R/I)\le k$.
\end{proof}

We now introduce the form of Hasse--Schmidt foliation that will be used later on.
Let $X$ be an algebraic variety over $\C_\infty$. A rank-$r$ {\em Hasse--Schmidt foliation} $\calf$ on $X$ is given, on every affine open
$U=\Spec A\subset X$, by operators
$D_\alpha\colon A\rightarrow A$ for $\alpha\in\N^r$. Let $\mathbf t=(t_1,\ldots,t_r)$ be a tuple of  formal variables, and let us set $\mathbf t^\alpha:=t_1^{\alpha_1}\cdots t_r^{\alpha_r}$ for 
$\alpha=(\alpha_1,\ldots,\alpha_r)\in\N^r$. Writing
\[
\theta_{\mathbf t}(f)
:=
\sum_{\alpha\in\N^r}D_\alpha(f)\mathbf t^\alpha,
\]
we require that $\theta_{\mathbf t}\colon A\rightarrow A[\![\mathbf t]\!]$ be a
$\C_\infty$-algebra morphism, compatible with restriction, and that
\[
\operatorname{ev}_{\mathbf t=0}\circ\theta_{\mathbf t}=\id_A,
\qquad
(\theta_{\mathbf s}\widehat\otimes 1)\circ\theta_{\mathbf t}
=
\theta_{\mathbf{s+t}},
\]
where $\theta_{\mathbf s}\widehat\otimes 1$ acts coefficientwise on
$A[\![\mathbf t]\!]$.
When these expansions converge on rigid-analytic neighbourhoods, their
images are called local analytic leaves.

For $x\in U$ and $m\ge0$, we define the $m$-th Hasse--Schmidt
jet of $F\in A$ along the leaf of $\calf$ through $x$ by
\[
\operatorname{jet}^m_{\calf,x}(F)
:=
\sum_{|\alpha|\le m}D_\alpha(F)(x)t^\alpha
\in
\C_\infty[t_1,\ldots,t_r]/(t_1,\ldots,t_r)^{m+1}.
\]
We also write
\[
\operatorname{jet}^\infty_{\calf,x}(F)
:=
\sum_{\alpha\in\N^r}D_\alpha(F)(x)t^\alpha
\in
\C_\infty[\![t_1,\ldots,t_r]\!]
\]
for the full formal expansion.

Let $V\subset X$ be a closed subvariety, with ideal sheaf $\cali_V\subset\calo_X$. For $x\in V(\C_\infty)$, we set
\[
I_x:=
\bigl(
\operatorname{jet}^\infty_{\calf,x}(f):f\in \cali_{V,x}
\bigr)
\subset
\C_\infty[\![t_1,\ldots,t_r]\!],
\]
and define $\dim(\calf\cap V)_x
:=\dim
\big(\C_\infty[\![t_1,\ldots,t_r]\!]/I_x\big)$.

The finite jet determinant argument in the proof below is inspired by the
multiplicity-operator construction of
\cite[Section~2.1]{bn}, adapted here to
Hasse--Schmidt jets and to positive dimensional intersections.

\begin{prop}\label{p:atypical-closed}
For every $k\ge0$, the set
\[
A_k(V,\calf)
:=\bigl\{
x\in V(\C_\infty):\dim(\calf\cap V)_x\ge k
\bigr\}
\]
is Zariski closed in $V$.
\end{prop}

\begin{proof}
The claim is local on $V$. Indeed, if $U\subset X$ is an open subset, then
for every $x\in V\cap U$ the germ $(\calf\cap V)_x$ is the same as the germ
$(\calf|_U\cap (V\cap U))_x$. Hence $A_k(V,\calf)\cap U=A_k(V\cap U,\calf|_U)$.
Since being Zariski closed is local on $V$, it is enough to prove the claim after replacing $X$ by an affine open neighbourhood of an arbitrary point of
$V$. Thus we may assume that $X=\Spec A$ is affine. In this case $V$ is defined by a finitely generated ideal $I_V\subset A$.

We set $R:=\C_\infty[\![t_1,\ldots,t_r]\!]$ with $\gotm:=(t_1,\ldots,t_r)$.
For $m\ge0$, we also set $R_m:=R/\gotm^{m+1}$. The quotient $R_m$ is a finite dimensional $\C_\infty$-vector space. We fix the monomial basis $\{t^\alpha:|\alpha|\le m\}$ and write
$N_m:=\dim_{\C_\infty}R_m$.

We first reduce the general case to the case of positive dimensional intersection. Let $1\le k\le r$, and let
\[
\lambda=(\ell_1,\ldots,\ell_{k-1})
\]
be a tuple of independent $\C_\infty$-linear forms in
$t_1,\ldots,t_r$. For $x\in V$, consider
\[
I_x^{\lambda}:=I_x+(\ell_1,\ldots,\ell_{k-1})\subset R.
\]
We claim that
\begin{equation}\label{e:intersection}
A_k(V,\calf)=\bigcap_{\lambda}
\bigl\{x\in V:\dim R/I_x^{\lambda}\ge 1\bigr\}
\end{equation}
where $\lambda$ runs over all such tuples of linear forms. Indeed, set $A_x:=R/I_x$ and let $\gotq$ be its maximal ideal. Let
$L\subset\gotq$ be the $\C_\infty$-span of the images of
$t_1,\ldots,t_r$, so that $L$ generates $\gotq$. If $\dim A_x\ge k$, then quotienting by one element of the maximal ideal can lower the dimension by at most one (see, for instance,
\cite[Lemma~10.60.13]{stacks}). Applying this successively gives
\[
\dim A_x/(\ell_1,\ldots,\ell_{k-1})
\ge
\dim A_x-(k-1)
\ge 1.
\]
Thus we get $\dim R/I_x^\lambda\ge1$ for every $\lambda$. On the other hand, suppose $\dim A_x=d<k$. Note that one can choose
$d$ linear forms $m_1, \ldots ,m_d \in L$ whose images form a system of parameters of $A_x$. Indeed, assume that $m_1,\ldots,m_i$ have already been chosen and that $\dim A_x/(m_1,\ldots,m_i)=d-i>0$. Let $\gotp_1,\ldots,\gotp_N$ be the minimal primes of this quotient. The image of $L$ is not contained in any $\gotp_j$, since it generates the maximal ideal. Hence, since $\C_\infty$ is infinite, prime avoidance allows us to choose a linear form $m_{i+1}$ whose image avoids all the $\gotp_j$, see \cite[Lemma~10.15.2]{stacks}. By the same dimension estimate, quotienting by $m_{i+1}$ lowers the dimension by one. After
$d$ steps we get 
$\dim A_x/(m_1,\ldots,m_d)=0$.
Extending $m_1,\ldots,m_d$ to a tuple
$\lambda=(\ell_1,\ldots,\ell_{k-1})$ of independent linear forms, we still have
\[
\dim R/I_x^\lambda=0\ ,
\]
since this is a quotient of the zero dimensional ring
$A_x/(m_1,\ldots,m_d)$. Hence $x$ does not belong to the intersection
\eqref{e:intersection}. This proves the claim.

It is therefore enough to prove that, for every fixed tuple
$\lambda=(\ell_1,\ldots,\ell_c)$ of linear forms, the set
\[
T_{\lambda}
:=\bigl\{
x\in V:\dim R/(I_x+(\ell_1,\ldots,\ell_c))\ge1
\bigr\}
\]
is Zariski closed.

Fix such a tuple $\lambda$. Choose generators
$f_1,\ldots,f_s$ for $I_V$. For $x\in V$, let
$I_{x,m}^{\lambda}\subset R_m$ be the image of
$I_x+(\ell_1,\ldots,\ell_c)$ in $R_m$. Since the Hasse--Schmidt expansion gives a morphism, $I_{x,m}^{\lambda}$ is generated as an ideal of $R_m$ by
\[
\operatorname{jet}^m_{\calf,x}(f_1),\ldots,
\operatorname{jet}^m_{\calf,x}(f_s),
\ell_1,\ldots,\ell_c.
\]
As a $\C_\infty$-vector subspace of $R_m$, the ideal
$I_{x,m}^{\lambda}$ is therefore spanned by
\[
t^\beta\operatorname{jet}^m_{\calf,x}(f_i)
\qquad \text{for}\quad |\beta|\le m,\ 1\le i\le s
\]
and by
\[
t^\beta\ell_j
\qquad \text{for}\quad|\beta|\le m,\ 1\le j\le c
\]
where all products are taken modulo $\gotm^{m+1}$.
Using the monomial basis $\{t^\alpha:|\alpha|\le m\}$ of $R_m$, we form the matrix $M_{m,\lambda}(x)$ whose columns are the coefficient vectors of these
spanning elements. The entries coming from
$t^\beta\operatorname{jet}^m_{\calf,x}(f_i)$ are regular functions of $x$,
as they are among the functions $D_\alpha(f_i)$ with $|\alpha|\le m$.
The entries coming from $t^\beta\ell_j$ are constants. Therefore all entries
of $M_{m,\lambda}(x)$ are regular functions on $V$.

By construction,
\[
\ell_{\C_\infty}(R_m/I_{x,m}^{\lambda})
=
N_m-\rank M_{m,\lambda}(x).
\]
Hence the condition $\ell_{\C_\infty}(R_m/I_{x,m}^\lambda)>m$
is equivalent to $\rank M_{m,\lambda}(x)\le N_m-m-1$.
This is a determinant condition, i.e., it is cut out by the vanishing of the $(N_m-m)$-minors of $M_{m,\lambda}(x)$. Thus, for fixed $m$ and $\lambda$, the locus
\[
T_{\lambda,m}
:=
\bigl\{
x\in V:
\ell_{\C_\infty}(R_m/I_{x,m}^\lambda)>m
\bigr\}
\]
is Zariski closed in $V$.

Now, let $J_x^\lambda:=I_x+(\ell_1,\ldots,\ell_c)\subset R$.
The quotient $R/J_x^\lambda$ has positive dimension if and only if it has
infinite length. By Lemma~\ref{l:multiplicity-local}, for every $m\ge0$ we have
\[
\ell_R\big(R/J_x^\lambda\big)>m
\qquad
\text{if and only if}
\qquad
\ell_{\C_\infty}\big(R_m/I_{x,m}^\lambda\big)>m.
\]
Therefore
\[
\dim R/J_x^\lambda\ge1
\quad
\text{if and only if}
\quad
\ell_{\C_\infty}(R_m/I_{x,m}^\lambda)>m
\]
for every $m\ge0$. This shows that
\[
T_\lambda
=\bigcap_{m\ge0}T_{\lambda,m}.
\]
Each $T_{\lambda,m}$ is Zariski closed, hence $T_\lambda$ is Zariski closed. Finally, $A_k(V,\calf)$ is, by \eqref{e:intersection}, an intersection of Zariski closed subsets of $V$, and is therefore Zariski closed. The cases $k=0$ and $k>r$ are  respectively $A_0(V,\calf)=V$ and $A_k(V,\calf)=\emptyset$.
\end{proof}

\begin{lem}\label{l:hs-invariance}
Let $X$ carry a rank-$r$ Hasse--Schmidt foliation $\calf$ for $r\ge1$, and let $V\subset X$ be a closed subvariety. Then $V\subset A_r(V,\calf)$ if and only if
$D_\alpha(\cali_V)\subset\cali_V$ for all $\alpha\in\N^r$.
\end{lem}
\begin{proof}
The statement is local, so assume $X=\Spec A$ and let $I_V\subset A$ be the ideal of $V$. Let us also consider $R:=\C_\infty[\![t_1,\ldots,t_r]\!]$. 

Suppose  $V\subset A_r(V,\calf)$. Let $f\in I_V$ and $x\in V$. By definition, the ideal $I_x=\bigl(\operatorname{jet}^{\infty}_{\calf,x}(h):h\in I_V\bigr)\subset R$ satisfies $\dim R/I_x\ge r$. Since $R$ has dimension $r$, this gives $\dim R/I_x=r$. Since $R$ is a domain, every nonzero ideal of $R$ has positive height. Hence $I_x=0$. Therefore
$\operatorname{jet}^{\infty}_{\calf,x}(f)
=\sum_{\alpha\in\N^r}D_\alpha(f)(x)t^\alpha=0$.
Thus $D_\alpha(f)(x)=0$ for every $\alpha$ and every $x\in V$. Since $V$ is
reduced, we get $D_\alpha(f)\in I_V$ for every $\alpha\in\N^r$.

On the other hand, suppose that $D_\alpha(I_V)\subset I_V$ for every
$\alpha\in\N^r$. Then, for every $f\in I_V$ and every $x\in V$, all coefficients $D_\alpha(f)(x)$ vanish. Hence $I_x=0$, so $\dim R/I_x=\dim R=r$.
Therefore $x\in A_r(V,\calf)$ for every $x\in V$.
\end{proof}

\begin{cor}\label{c:zariski-closure-leaf}
Let $X$ carry a rank-$r$ Hasse--Schmidt foliation $\calf$, and let
$V_0\subset X$ be an algebraic subvariety. Let $\calv$ be an irreducible
analytic subset contained in $V_0\cap\lcal$, where $\lcal$ is a  leaf of $\calf$, and set $V:=\overline{\calv}^{\mathrm{Zar}}\subset V_0$.
If $\dim\calv=m$, then
$V\subset A_m(V,\calf)$.
\end{cor}

\begin{proof}
The smooth locus of $\calv$ is dense in $\calv$. At every point of this smooth
locus, the intersection of $V$ with the leaf has dimension at least $m$.
Hence $\calv^{\mathrm{sm}}\subset A_m(V,\calf)$.
By Proposition~\ref{p:atypical-closed}, the set $A_m(V,\calf)$ is Zariski
closed in $V$. Since $\calv^{\mathrm{sm}}$ is Zariski dense in
$V=\overline{\calv}^{\mathrm{Zar}}$, we get $V\subset A_m(V,\calf)$.
\end{proof}

\subsubsection{Ramanujan foliation}

Let $h=P_{q+1,1}$ be the Poincar\'e series normalized as in
\cite{gekelerj,bp}, using the fixed Carlitz period $\widetilde\pi$,
so that $\Delta=-h^{q-1}$. Set
\[
\tilde E:=\frac{1}{h}\frac{dh}{dz}\qquad
\text{and}
\qquad
E:=-\widetilde\pi^{-1}\tilde E.
\]
Thus $E$ is the false Eisenstein series in the normalization of
\cite{bp}. Following \cite[(2), p.~4]{bp}, which in turn follows \cite[(8.6), Theorem~9.1 and p.~686]{gekelerj}, we
consider the Ramanujan-type system
\begin{equation}\label{e:ramanujan}
\begin{cases}
D_1E=E^2\\
D_1g=-(Eg+h)\\
D_1h=Eh
\end{cases}
\end{equation}
where we recall that $D_1=(-\widetilde\pi)^{-1}\dfrac{d}{dz}$. By
\cite[Theorem~2]{bp}, for every $n\ge 0$ one has
\begin{equation}\label{e:algebraicity}
D_n\bigl(\C_\infty[E,g,h]\bigr)\subset \C_\infty[E,g,h].
\end{equation}

Let $\caly:=\A^3_{\C_\infty}$
with coordinates $(E,g,h)$, and define
\[
C\colon\Omega\rightarrow \caly,
\qquad
z \mapsto (E(z),g(z),h(z)).
\]
For a connected admissible open $U\subset\Omega$, we call
\[
\lcal_U:=C(U)\subset\caly
\]
a (local) {\em Ramanujan leaf}. The curve $C(\Omega)$ shall be called the basic Ramanujan leaf.

For $\gamma\in\gl_2(A)$, write $\gamma=\begin{bmatrix}a&b\\ c&d\end{bmatrix}$ and
$\lambda_\gamma(z):=cz+d$.
Since $z\in\Omega$, one has $\lambda_\gamma(z)\neq0$.

Let $\widetilde{\caly}:=\p^1_{\C_\infty}\times \caly$
with coordinates $(z,E,g,h)$, and let
\[
\widetilde{\lcal}_U
:=
\bigl\{(z,E(z),g(z),h(z)):z\in U\bigr\}
\subset \widetilde{\caly}
\]
be the lifted local Ramanujan leaf.

By \eqref{e:algebraicity}, the operators $D_n$ define a rank-$1$ Hasse--Schmidt foliation $\calf$ on $\caly=\A^3_{\C_\infty}$, i.e.,  the Hasse--Schmidt jet expansion has one formal parameter.  On the affine chart with coordinate $z$, this lifted foliation is given by $D_0z=z$, $D_1z=(-\widetilde\pi)^{-1}$, and  $D_nz=0$ for $n\ge2$, together with the operators $D_n$ on $\C_\infty[E,g,h]$. Thus \eqref{e:algebraicity} is what makes this an algebraic Hasse--Schmidt foliation in the sense of Section~\ref{ss:hs-foliation} and Proposition~\ref{p:atypical-closed}.

 We consider the following rational action of $\pgl_{2,\C_\infty}$ on
$\widetilde{\caly}$. For $[\gamma]\in\pgl_{2,\C_\infty}$, choose
a representative $\gamma=
\begin{bmatrix}
a&b\\ c&d
\end{bmatrix}
\in \Sl_{2,\C_\infty}$
and set $\lambda_\gamma(z):=cz+d$. We define
\begin{equation}\label{e:rationaltransformation}
[\gamma].(z,E,g,h)
:=
\left(
\gamma z,\,
\lambda_\gamma(z)^2
\left(E-\frac{c}{\widetilde\pi\,\lambda_\gamma(z)}\right),\,
\lambda_\gamma(z)^{q-1}g,\,
\lambda_\gamma(z)^{q+1}h
\right).
\end{equation}
This is independent of the determinant one representative.

We recall the following standard result.

\begin{lem}\label{l:dense}
Let $\Gamma\le \pgl_2(A)$ be of finite index. Then $\Gamma$ is Zariski dense in $\pgl_{2,\C_\infty}$. 
\end{lem}
\begin{proof}
As the index $[\pgl_2(A):\Gamma]=n$ is finite, we choose representatives
$\gamma_1,\dots,\gamma_n\in\pgl_2(A)$ such that
$\pgl_2(A)=\bigsqcup_{i=1}^n\gamma_i\Gamma$. Taking Zariski closures and using
the Zariski density of $\pgl_2(A)$ in $\pgl_{2,\C_\infty}$, we obtain
\[
\pgl_{2,\C_\infty}
=
\overline{\pgl_2(A)}
\subseteq
\bigcup_{i=1}^n\gamma_i\overline{\Gamma}.
\]
Since $\pgl_{2,\C_\infty}$ is irreducible, we conclude.
\end{proof}

Consider the rational map
\[
r\colon\pgl_{2,\C_\infty}\times\widetilde{\caly}\dashrightarrow
\widetilde{\caly},
\qquad
([\gamma],x)\mapsto r_\gamma(x)
\]
defined by \eqref{e:rationaltransformation}, and let $\calu$ be the open subset on which $r$ is regular. Let $\Gamma_r:=\bigl\{\big([\gamma],x,r_\gamma(x)\big):([\gamma],x)\in\calu\bigr\}$ denote its graph.

\begin{lem}\label{l:leaf-preserving-closed}
There is a Zariski closed subset $\calb\subset\calu$ such that
$([\gamma],x)\in\calb$ if and only if the product Hasse--Schmidt leaf through
$([\gamma],x,r_\gamma(x))$ has positive dimensional intersection with
$\Gamma_r$.
\end{lem}

\begin{proof}
On $\calu\times\widetilde{\caly}$ consider the product Hasse--Schmidt
foliation $\gcal$ on $\pgl_{2,\C_\infty}\times\widetilde{\caly}\times\widetilde\caly$ which is trivial in the $\pgl_{2,\C_\infty}$-direction and is equal to
$\calf$ on both copies of $\widetilde{\caly}$. This foliation is algebraic by
\eqref{e:algebraicity}. Since $\widetilde{\caly}$ is separated, the graph
$\Gamma_r$ is closed in $\calu\times\widetilde{\caly}$.

By Proposition~\ref{p:atypical-closed}, applied with $k=1$ to
$\Gamma_r\subset\calu\times\widetilde{\caly}$, the set $A_1(\Gamma_r,\gcal)$ is Zariski closed in $\Gamma_r$. Since the projection
$\Gamma_r\rightarrow\calu$ is an isomorphism, we denote by
$\calb\subset\calu$ the image of $A_1(\Gamma_r,\gcal)$ under this
identification. This is Zariski closed in $\calu$.
\end{proof}

Thus $([\gamma],x)\in\calb$  means that, infinitesimally along the Hasse--Schmidt leaf through $x$, the relation $y=r_\gamma(x)$ remains inside a Ramanujan leaf in the second factor.

\begin{lem}\label{l:basic-lifted-leaf-dense}
The basic lifted Ramanujan leaf is Zariski dense in $\widetilde{\caly}$.
\end{lem}

\begin{proof}
It is enough to work on the affine chart of $\p^1$ with coordinate $Z$.
Suppose that
\[
P\bigl(z,E(z),g(z),h(z)\bigr)=0
\]
for some nonzero $P\in\C_\infty[Z,E,g,h]$. Write $P(Z,E,g,h)=\sum_{i=0}^d P_i(E,g,h)Z^i$ with $P_d\neq0$. For every $a\in A$, the translation $z\mapsto z+a$ belongs to $\Sl_2(A)$; for this element one has $c=0$ and $\lambda_\gamma=1$. Hence one immediately has $E(z+a)=E(z)$, $g(z+a)=g(z)$, and $h(z+a)=h(z)$. Therefore
\[
P\bigl(z+a,E(z),g(z),h(z)\bigr)=0
\]
for every $a\in A$. Fixing $z$, the left-hand side is a polynomial in $a$
with infinitely many zeros. Thus it is identically zero as a polynomial in
$a$. Its leading coefficient is $P_d(E(z),g(z),h(z))$, so $P_d(E(z),g(z),h(z))=0$ for all $z\in\Omega$. By \cite[Theorem~1] {bp}, the functions $E,g,h$ are algebraically independent over $\C_\infty$, contradicting $P_d\neq0$. Thus no such nonzero $P$ exists.
\end{proof}

\begin{lem}\label{l:foliation-equivariant}
Assume that $q$ is odd. The Ramanujan foliation is
$\pgl_{2,\C_\infty}$-equivariant in the sense that each
$r_\gamma$ sends formal leaves to formal leaves,
wherever both $z$ and $\gamma z$ are finite.
\end{lem}

\begin{proof}
Let $r_\gamma$ denote the rational transformation in
\eqref{e:rationaltransformation}. For $\gamma\in\Sl_2(A)$, the
transformation formula for $E$ in \cite[(11)]{bp}, together with the modular
transformation laws for $g$ and $h$, gives
\begin{equation}\label{e:rt}
r_\gamma\bigl(z,E(z),g(z),h(z)\bigr)
= \bigl(\gamma z,E(\gamma z),g(\gamma z),h(\gamma z)\bigr).
\end{equation}
Thus every $\gamma\in\Sl_2(A)$ sends the basic lifted Ramanujan leaf to the basic lifted Ramanujan leaf.

Let $\calb\subset\calu$ be the closed subset given by
Lemma~\ref{l:leaf-preserving-closed}. For fixed $x\in\widetilde{\caly}$, set
\[
\calu_x:=\big\{[\gamma]\in\pgl_{2,\C_\infty}:([\gamma],x)\in\calu\big\}
\qquad
\text{and}
\qquad
B_x:=\big\{[\gamma]\in\calu_x:([\gamma],x)\in\calb\big\}.
\]
Since $B_x$ is the inverse image of the Zariski closed set $\calb$ under the regular
map $\calu_x\rightarrow\calu$, $[\gamma]\mapsto([\gamma],x)$, it is Zariski closed
in $\calu_x$.

Suppose now that $x$ lies on the basic lifted Ramanujan leaf, say
$x=(z_0,E(z_0),g(z_0),h(z_0))$. If $\gamma\in\Sl_2(A)$ and
$[\gamma]\in\calu_x$, then \eqref{e:rt} holds as an identity of germs at
$z_0$. Hence the product Hasse--Schmidt leaf through
$([\gamma],x,r_\gamma(x))$ has positive dimensional intersection with
$\Gamma_r$.  By Lemma~\ref{l:leaf-preserving-closed}, we have
$([\gamma],x)\in\calb$, hence $[\gamma]\in B_x$.

Therefore $B_x$ contains the image of $\Sl_2(A)$ in $\pgl_2(A)$, intersected with
$\calu_x$. This image has finite index in $\pgl_2(A)$.

Hence, by Lemma~\ref{l:dense}, this
image is Zariski dense in $\pgl_{2,\C_\infty}$. Therefore its intersection with the open set $\calu_x$ is Zariski dense in $\calu_x$. Since $B_x$ is Zariski closed in $\calu_x$, we get $B_x=\calu_x$.

Now fix $[\gamma]\in\pgl_{2,\C_\infty}$ and consider
\[
\calu_\gamma:=\{x\in\widetilde{\caly}:([\gamma],x)\in\calu\}
\qquad
\text{and}
\qquad
C_\gamma:=\{x\in\calu_\gamma:([\gamma],x)\in\calb\}.
\]
Again $C_\gamma$ is the inverse image of $\calb$ under the regular map
$\calu_\gamma\rightarrow\calu$, $x\mapsto([\gamma],x)$, so it is Zariski closed in $\calu_\gamma$. The previous paragraph shows that $C_\gamma$ contains the
intersection of $\calu_\gamma$ with the basic lifted Ramanujan leaf. By
Lemma~\ref{l:basic-lifted-leaf-dense}, this intersection is Zariski dense in
$\calu_\gamma$. Hence $C_\gamma=\calu_\gamma$. 
\end{proof}

\begin{rmk}
The preceding descent is genuinely scheme theoretic. The kernel of
$\Sl_{2,\C_\infty}\rightarrow\pgl_{2,\C_\infty}$ is indeed the group scheme $\boldsymbol\mu_2$. For every
$\C_\infty$-algebra $R$ and every
$\zeta\in\boldsymbol\mu_2(R)$, the central element $\zeta I$ acts through \eqref{e:rationaltransformation} by
$(z,E,g,h)\mapsto
\bigl(z,E,\zeta^{q-1}g,\zeta^{q+1}h\bigr)$.
If $q$ is odd, then $q-1$ and $q+1$ are even, so the kernel acts trivially and the rational $\Sl_{2,\C_\infty}$-action descends
to $\pgl_{2,\C_\infty}$.

On the other hand, in even characteristic, the formula holds again on $\C_\infty$-points. Indeed, since $\C_\infty$ is algebraically closed, every class in $\pgl_2(\C_\infty)$ has a determinant one
representative, and this representative is unique since we have
$\boldsymbol\mu_2(\C_\infty)=\{1\}$. However,
$\boldsymbol\mu_2$ is then nonreduced. For instance, if $R=\C_\infty[\varepsilon]/(\varepsilon^2)$
and $\zeta=1+\varepsilon$, then $\zeta\in\boldsymbol\mu_2(R)$, and, since $q-1$ and $q+1$ are odd, $\zeta I$ acts by
\[
(z,E,g,h)\mapsto
\bigl(z,E,(1+\varepsilon)g,(1+\varepsilon)h\bigr).
\]
Thus, in even characteristic, the pointwise action does not descend scheme theoretically to a  $\pgl_{2,\C_\infty}$-action on the original phase space.
\end{rmk}

\subsubsection{The quotient phase space and a principal homogeneous space}

We call a surjective left $G$-space $\pi\colon P\rightarrow X$ a \emph{principal
homogeneous $G$-space over $X$} if the canonical map
\[
G\times P\rightarrow P\times_XP,
\qquad
(g,x)\mapsto(g\cdot x,x)
\]
is an isomorphism.

Assume that $q$ is odd.  On the dense open
\[
\widetilde{\caly}^{\,\times}:=
\p^1_{\C_\infty}\times\A^1\times\G_{m}^2
\subset\widetilde{\caly}
\]
consider the involution
\[
\iota(z,E,g,h):=(z,E,-g,-h).
\]
It commutes with the rational $\pgl_{2,\C_\infty}$-action and is distinct
from $\boldsymbol\mu_2$ removed in the passage $\Sl_2\rightarrow\pgl_2$.  The function
\[
s:=g^{(q+1)/2}h^{-(q-1)/2}
\]
is $\pgl_{2,\C_\infty}$-invariant, and
\[
s^2=g^{q+1}h^{1-q}=-j,
\qquad
s\circ\iota=-s.
\]
Since $s$ is $\pgl_{2,\C_\infty}$-invariant and
$s\circ\iota=-s$, a general fiber of the  $j$-function contains at least two distinct $\pgl_{2,\C_\infty}$-orbits, exchanged by $\iota$. Hence the original $j$-map is not a principal homogeneous
$\pgl_{2,\C_\infty}$-space over the $j$-line.  After quotienting by $\iota$, Lemma~\ref{l:local-principal-model} gives a
$\pgl_{2,\C_\infty}$-equivariant birational identification over the $j$-line with the principal homogeneous space $B\times\pgl_{2,\C_\infty}$.

We set
\[
u:=\frac hg
\qquad
\text{and}
\qquad
v:=gh
\]
and consider
\[
\widetilde{\caly}^{\,\sharp}
:=
\p^1_{\C_\infty}\times\A^1_E\times\G_{m,u}\times\G_{m,v}.
\]
The morphism
\begin{equation}\label{e:mu2-quotient}
\vartheta\colon\widetilde{\caly}^{\,\times}\rightarrow
\widetilde{\caly}^{\,\sharp},
\qquad
(z,E,g,h)\mapsto\left(z,E,\frac hg,gh\right),
\end{equation}
is the finite \'{e}tale quotient of degree $2$ by
$\langle\iota\rangle\simeq\boldsymbol\mu_2$.
Indeed,
\[
\C_\infty[E,g^{\pm1},h^{\pm1}]^\iota
=
\C_\infty[E,u^{\pm1},v^{\pm1}],
\qquad
g^2=\frac vu,
\qquad h=ug.
\]
On the modular leaf $h\neq0$
because $\Delta=-h^{q-1}$ is nonzero on $\Omega$, and if a coordinate has $g\equiv0$, then the corresponding $j$-coordinate is identically zero and the final weakly-special conclusion is immediate.

Define
\begin{equation}\label{e:quotient-j}
\pi_j^\sharp(z,E,u,v):=-vu^{-q}.
\end{equation}
Then
\[
\pi_j^\sharp\circ\vartheta
=
\pi_j
\qquad
\text{with}
\qquad
\pi_j(z,E,g,h):=-g^{q+1}h^{1-q}.
\]
The action \eqref{e:rationaltransformation} descends to
\begin{equation}\label{e:quotient-action}
[\gamma]\cdot(z,E,u,v)=
\left(
\gamma z,
\lambda_\gamma(z)^2
\left(E-\frac{c}{\widetilde\pi\lambda_\gamma(z)}\right),
\lambda_\gamma(z)^2u,
\lambda_\gamma(z)^{2q}v
\right),
\end{equation}
and $\pi_j^\sharp$ is invariant under this action.

The Ramanujan equations give
\[
\frac{D_1j}{j}=-u,
\qquad
D_1u=2Eu+u^2,
\qquad
D_1^2j=2E\,D_1j.
\]
Therefore on the open locus $jD_1j\neq0$,
\begin{equation}\label{e:E,g,h,j}
u=-\frac{D_1j}{j},
\qquad
v=\frac{(D_1j)^q}{j^{q-1}},
\qquad
E=\frac{D_1^2j}{2D_1j},
\end{equation}
and thus
\[
\C_\infty(E,u,v)
=
\C_\infty(j,D_1j,D_1^2j).
\]
Thus $\widetilde{\caly}^{\,\sharp}$ is  the birational ``phase''
space determined by the second order jet of $j$.

\begin{lem}\label{l:hs-descends-quotient}
For every $n\ge0$ one has
$D_n\circ\iota=\iota\circ D_n$ on $\C_\infty[E,g,h]$.
Hence the iterative Hasse--Schmidt structure extends to
$\C_\infty[E,g^{\pm1},h^{\pm1}]$ and descends through
\eqref{e:mu2-quotient} to
$\C_\infty[E,u^{\pm1},v^{\pm1}]$.
\end{lem}

\begin{proof}
A monomial $E^a g^b h^c$ has weight
$w=2a+(q-1)b+(q+1)c$
and type $m\equiv a+c\bmod{q-1}$, while $\iota$ acts on it by
$(-1)^{b+c}$.  If $E^{a'}g^{b'}h^{c'}$ is another monomial of the same weight and type, set $A:=a'-a$, $B:=b'-b$ and $C:=c'-c$. Then $A+C=(q-1)k$ and $2A+(q-1)B+(q+1)C=0$, so $B+C=-2k$.  Thus $b+c\bmod2$ is constant on every weight--type
component.

By \cite[Theorem~2]{bp}, we have that $D_n$ sends weight--type $(w,m)$ to
$(w+2n,m+n)$.  If $E^{a'}g^{b'}h^{c'}$ occurs in
$D_n(E^ag^bh^c)$, put again
$A:=a'-a$, $B:=b'-b$ and $C:=c'-c$.
Then $A+C=n+(q-1)k$ and $2A+(q-1)B+(q+1)C=2n$, and again $B+C=-2k$, so $D_n$ preserves the $\iota$-parity and
commutes with $\iota$.

Set $R:=\C_\infty[E,g,h]$, $R':=R[g^{-1},h^{-1}]$, and, as in
Section~\ref{ss:hs-foliation}, write
\[
\theta_T(f):=\sum_{n\ge0}D_n(f)T^n.
\]
Since $\theta_T(g)$ and $\theta_T(h)$ have constant terms $g$ and $h$,
they are units in $R'[\![T]\!]$.  The universal property of localization
therefore extends $\theta_T$ uniquely to $R'$, with
$\theta_T(g^{-1})=\theta_T(g)^{-1}$ and
$\theta_T(h^{-1})=\theta_T(h)^{-1}$.

Let $\widetilde\theta_S\colon R'[\![T]\!]\rightarrow R'[\![S,T]\!]$ be obtained by applying $\theta_S$ coefficientwise.  The iterativity identity on $R$ is
$\widetilde\theta_S\circ\theta_T=\theta_{S+T}$.
Both sides are algebra homomorphisms $R'\rightarrow R'[\![S,T]\!]$.  They agree on $R$ by the normalized form of \eqref{e:leibniz}; moreover, both send $g^{-1}$ and $h^{-1}$ to the inverses of their respective values on $g$ and $h$.  Hence they agree on all of $R'$.

The two maps $\iota\circ\theta_T$ and $\theta_T\circ\iota$ agree on $R$, and hence on $R'$.  Thus
$\theta_T\bigl((R')^\iota\bigr)
\subset
(R')^\iota[\![T]\!]$.
Using $(R')^\iota=\C_\infty[E,u^{\pm1},v^{\pm1}]$ we obtain the descended iterative Hasse--Schmidt structure on the quotient ring.
\end{proof}

Let
\[
B:=\G_{m,\C_\infty}
\qquad
\text{and}
\qquad
P^{\mathrm{fr}}:=B\times \pgl_{2,\C_\infty}
\]
where $\pgl_{2,\C_\infty}$ acts on $P^{\mathrm{fr}}$ by left multiplication on the second factor.  We shall call $P^{\mathrm{fr}}$ the \emph{frame space}.  Define the section
\[
\sigma\colon B\rightarrow\widetilde{\caly}^{\,\sharp},
\qquad
\sigma(j):=(0,0,1,-j).
\]

Consider
\[
\pgl_{2,\C_\infty}^\circ:=
\left\{
\left[\begin{matrix}a&b\\ c&d\end{matrix}\right]\in \pgl_{2,\C_\infty}:
 d\neq0
\right\},
\]
which is a dense open subset of $\pgl_{2,\C_\infty}$.

\begin{lem}\label{l:local-principal-model}
The map
\[
\Phi\colon P^{\mathrm{fr}}\dashrightarrow\widetilde{\caly}^{\,\sharp},
\qquad
(j,[\gamma])\mapsto[\gamma]\cdot\sigma(j),
\]
is a $\pgl_{2,\C_\infty}$-equivariant birational map over $B$.
\end{lem}

\begin{proof}
Let $\widetilde{\caly}^{\,\sharp,\circ}
:=
\A^1_z\times\A^1_E\times\G_{m,u}\times\G_{m,v}
\subset\widetilde{\caly}^{\,\sharp}$.
We show that $\Phi$ restricts to 
$B\times \pgl_{2,\C_\infty}^\circ
\simeq \widetilde{\caly}^{\,\sharp,\circ}$.
Take $[\gamma]\in \pgl_{2,\C_\infty}^\circ$ and choose a determinant one representative
$\gamma=
\begin{bmatrix}a&b\\ c&d\end{bmatrix}$.
Since $\lambda_\gamma(0)=d$, formula~\eqref{e:quotient-action} gives
\begin{equation}\label{e:frame-map-formula}
\Phi(j,[\gamma])
=\left(
\frac bd,
-\frac{cd}{\widetilde\pi},
d^2,
-jd^{2q}
\right).
\end{equation}
This expression is independent of the choice of the determinant-one
representative, whose only ambiguity is multiplication by $-1$.

On the other hand, for $(z,E,u,v)\in\widetilde{\caly}^{\,\sharp,\circ}$, we set
$j:=-vu^{-q}$
and define
$M(z,E,u):=
\begin{bmatrix}
u^{-1}(1-\widetilde\pi Ez)&z\\
-\widetilde\pi Eu^{-1}&1
\end{bmatrix}$.
A direct computation gives
$\det M(z,E,u)=u^{-1}\neq0$,
so $[M(z,E,u)]\in \pgl_{2,\C_\infty}^\circ$.  This defines a regular map
\[
\Psi\colon\widetilde{\caly}^{\,\sharp,\circ}\rightarrow B\times \pgl_{2,\C_\infty}^\circ,
\qquad
(z,E,u,v)\mapsto\bigl(j,[M(z,E,u)]\bigr).
\]
Choose $t\in\C_\infty^\times$ with $t^2=u$.  Then $tM(z,E,u)$ has
determinant one, and formula~\eqref{e:frame-map-formula} gives $\frac bd=z$, $-\frac{cd}{\widetilde\pi}=E$, $d^2=u$ and  $-jd^{2q}=v$. Hence $\Phi\circ\Psi$ is the identity.

For the converse, start with
$\gamma=\begin{bmatrix}a&b\\ c&d\end{bmatrix}\in\Sl_2$ and let
$(z,E,u,v)=\Phi(j,[\gamma])$.  Since $u=d^2$ and $v=-jd^{2q}$, $-vu^{-q}=j$. Moreover, $ad-bc=1$, so
$1-\widetilde\pi Ez=1+bc=ad$. Therefore $M(z,E,u)=
\frac1d
\begin{bmatrix}a&b\\ c&d\end{bmatrix}$.
Thus $[M(z,E,u)]=[\gamma]$, so $\Psi\circ\Phi$ is also the identity.
This proves birationality.  Finally,
\[
\Phi(j,g[\gamma])=g\cdot\Phi(j,[\gamma])
\]
proves $\pgl_{2,\C_\infty}$-equivariance, and $\pi_j^\sharp\circ\Phi(j,[\gamma])=j$ proves that $\Phi$ is over $B$.
\end{proof}

Taking products, $(\widetilde{\caly}^{\,\sharp})^n$ is
$\pgl_{2,\C_\infty}^n$-equivariantly birational over $B^n$ to the principal homogeneous
$\pgl_{2,\C_\infty}^n$-space $B^n\times \pgl_{2,\C_\infty}^n$.

By Lemma~\ref{l:hs-descends-quotient}, the descended Hasse--Schmidt
expansion pulls back through $\Phi$ on $B\times\pgl_{2,\C_\infty}^\circ$, where
\[
\theta_T(j)=j-ujT+O(T^2).
\]
Since $uj$ is a unit, the formal inverse function theorem gives a unique
regular reparameterization $\eta_S$ with $\eta_S(j)=j+S$.
Lemma~\ref{l:foliation-equivariant} gives preservation of the formal leaves. Since the action fixes $j$,
uniqueness makes $\eta_S$ exactly $\pgl_{2,\C_\infty}$-equivariant wherever
both sides are defined. Composing the unique horizontal lifts of successive
$j$-increments also proves iterativity. Translating this reparameterized
expansion gives structures on the opens
$B\times\gamma\pgl_{2,\C_\infty}^\circ$, for
$\gamma\in\pgl_2(\C_\infty)$, which cover $P^{\mathrm{fr}}$; uniqueness
makes them agree on overlaps. We denote the resulting
$\pgl_{2,\C_\infty}$-equivariant iterative Hasse--Schmidt foliation by
$\calf^{\mathrm{fr}}$.
For products, we take products of these principal homogeneous spaces and of the foliations.

\subsubsection{The Hasse--Schmidt Galois group}

Set
\[
G:=\pgl_{2,\C_\infty}^n.
\]
Let $X\subset\A^n_{\C_\infty}$ be irreducible and assume first that $X$ is
not contained in a coordinate hyperplane $x_i=0$.  Choose a non-empty smooth open
\[
X^\circ\subset X\cap B^n
\]
and set
\begin{equation}\label{e:frame-space-X}
P_X
:= X^\circ\times_{B^n}(B^n\times G)
= X^\circ\times G
\qquad
\text{and}
\qquad
\pi_X:P_X\rightarrow X^\circ.
\end{equation}
This is the principal homogeneous $G$-space used from now on.

\begin{lem}\label{l:hs-base-change}
After shrinking $X^\circ$, assume that it is affine and fix an \'etale morphism
\[
t=(t_1,\ldots,t_r):X^\circ\rightarrow\A^r_{\C_\infty}
\qquad 
\text{with}
\quad
r:=\dim X^\circ.
\]
Then $P_X$ carries a $G$-equivariant iterative Hasse--Schmidt foliation
$\calf_{X,t}$ of rank $r$.  The projection $\pi_X$ identifies every local leaf germ with a neighbourhood of its image in $X^\circ$.
\end{lem}

\begin{proof}
The reparameterized expansion $\eta_S$ constructed above is exactly
$\pgl_{2,\C_\infty}$-equivariant. Taking products, every formal change
$j_i\mapsto j_i+S_i$ has a unique horizontal lift $\eta_{\mathbf S}$
on the frame space.

Let $\calt_{\boldsymbol\varepsilon}$ be the Taylor map in the chosen
\'etale coordinates, i.e., $\calt_{\boldsymbol\varepsilon}(t_\nu)
=t_\nu+\varepsilon_\nu$;
it exists by formal \'etaleness
\cite[Lemma~37.8.10]{stacks}.  Substituting
$S_i=
\calt_{\boldsymbol\varepsilon}(j_i|_{X^\circ})
-j_i|_{X^\circ}$
in the product expansion $\eta_{\mathbf S}$ defines an expansion
on $P_X$.  It is well defined because
$\left.\eta_{\mathbf S}(j_i)\right|_
{S_i=\calt_{\boldsymbol\varepsilon}(j_i)-j_i}=
\calt_{\boldsymbol\varepsilon}(j_i)$.
It has rank $r$, and at every finite order it is the unique horizontal lift of $\calt_{\boldsymbol\varepsilon}$.  This uniqueness implies iterativity, $G$-equivariance, and compatibility on overlapping frame charts.  Hence the local expansions glue to $\calf_{X,t}$.

In each factor, every frame point lies on a translate of the
 image of the lifted local Ramanujan leaf in the frame space. Since $D_1j=-uj\neq0$
on the modular leaf and the action fixes $j$, these translates
admit convergent $j$-parameters by
\cite[\S10, Proposition~10.8]{abhyankar}.
By uniqueness, their Taylor expansions are the formal horizontal
lifts constructed above. Taking products and restricting to
$X^\circ$ gives the required convergent local leaves.
Thus every local analytic leaf
projects isomorphically onto a neighbourhood in $X^\circ$.
\end{proof}

\begin{rmk}\label{r:hs-coordinate-choice}
The operators in $\calf_{X,t}$ depend on $t$.  By the uniqueness in the
proof, a second \'etale coordinate system changes the leaf parameters by
an automorphism of
$\C_\infty[\![T_1,\ldots,T_r]\!]$ with invertible linear part.  Hence the
unparameterized leaf germs, the dimensions defining $A_k(V,\calf)$, Hasse--Schmidt invariance, minimal invariant subvarieties, and the Galois conjugacy class below are independent of $t$.  We shall write $\calf_X$ whenever the choice is irrelevant.
\end{rmk}

The global leaves of the product foliation on $(P^{\mathrm{fr}})^n$ are, by definition, the products
\[
\lcal_1\times\cdots\times\lcal_n
\]
of global one factor leaves. A global leaf of $\calf_X$ is an irreducible analytic component of
\[
P_X
\cap
(\lcal_1\times\cdots\times\lcal_n).
\]
By Lemma~\ref{l:hs-base-change}, their
projections to $X^\circ$ are locally isomorphisms; hence their images
contain a non-empty admissible open and are Zariski dense by
\cite[Lemma~2.2.3 and Theorem~2.3.1]{conrad-irred}.

From now on, $P_X$ is equipped with $\calf_X$.  If $X\subset\{x_i=0\}$ for some $i$, we declare $X$ HS-non-generic.  This
case is already weakly-special.

A closed subvariety $W\subset P_X$ is called \emph{Hasse--Schmidt invariant} if every Hasse--Schmidt leaf through a point of $W$ is contained in $W$.

\begin{lem}\label{l:leaf-closure-invariant}
Let $\lcal$ be a Hasse--Schmidt leaf over $X^\circ$, and let
$Z:=\overline{\lcal}^{\,\mathrm{Zar}}\subset P_X$.  Then $Z$ is Hasse--Schmidt invariant.
\end{lem}

\begin{proof}
Corollary~\ref{c:zariski-closure-leaf} gives
$Z\subset A_{\dim X}(Z,\calf_X)$, and Lemma~\ref{l:hs-invariance} shows that the
ideal of $Z$ is stable under every Hasse--Schmidt operator.  By Lemma~\ref{l:hs-base-change}, the Taylor expansion of each local analytic leaf chart is the corresponding Hasse--Schmidt expansion.  Hence every
local leaf chart through a point of $Z$ lies in $Z$.  A  leaf is a union of  local charts, so it also lies in $Z$.
\end{proof}

A non-empty irreducible closed Hasse--Schmidt invariant subvariety
$Z\subset P_X$ is called \emph{minimal} if it contains no proper non-empty irreducible closed Hasse--Schmidt invariant subvariety.

\begin{lem}\label{l:hs-minimal-leaf}
Let $\pi\colon P\rightarrow X$ be a principal homogeneous $G$-space with a $G$-equivariant
Hasse--Schmidt foliation, and assume that every  leaf dominates $X$.  For an irreducible closed $Z\subset P$, the following are equivalent:
\begin{enumerate}
\item $Z$ is minimal Hasse--Schmidt invariant;
\item $Z=\overline{\lcal_x}^{\,\mathrm{Zar}}$ for every $x\in Z(\C_\infty)$;
\item $Z$ is the Zariski closure of a Hasse--Schmidt leaf.
\end{enumerate}
\end{lem}

\begin{proof}
This is the argument of \cite[Lemma~2.2]{axschanuel}.  Lemma~\ref{l:leaf-closure-invariant} gives $(1)\Rightarrow(2)$, and $(2)\Rightarrow(3)$ is immediate.  Conversely, if $Z=\overline{\lcal_x}^{\,\mathrm{Zar}}$ and $W\subset Z$ is a non-empty irreducible closed invariant subvariety, 
then $W\rightarrow X$ is dominant as $W$ contains a leaf.
By Chevalley's theorem, $\pi(W)$ contains a non-empty Zariski
open subset of $X$, which meets $\pi(\lcal_x)$.
Thus we may choose $q\in W$ and $g\in G$ with $gq\in\lcal_x$.
Equivariance identifies the local leaf germs, so the
rigid-analytic identity principle gives
$g\overline{\lcal_q}^{\,\mathrm{Zar}}=Z$.
Since $\overline{\lcal_q}^{\,\mathrm{Zar}}\subseteq W\subseteq Z$, we obtain $\dim W=\dim Z$, hence $W=Z$.
\end{proof}

For a minimal Hasse--Schmidt invariant subvariety $Z\subset P_X$, define
\[
\Gal_{\mathrm{HS}}(Z):=\Stab_G(Z).
\]

\begin{lem}\label{l:hs-galois-principal}
Let $Z\subset P_X$ be minimal Hasse--Schmidt invariant.  Then
$\Gal_{\mathrm{HS}}(Z)$ is an algebraic subgroup of $G$, and, after possibly shrinking $X^\circ$, the map $Z\rightarrow X^\circ$ is a principal homogeneous $\Gal_{\mathrm{HS}}(Z)$-space over $X^\circ$.
\end{lem}

\begin{proof}
Set $H:=(\Gal_{\mathrm{HS}}(Z))_{\mathrm{red}}$.
The scheme theoretic stabilizer of $Z$ is a closed subgroup scheme of $G$;
since $\C_\infty$ is perfect, its reduction $H$ is a smooth algebraic
subgroup. Its action preserves $Z$.  Since $Z$ is the closure of a leaf dominating $X^\circ$, the map $Z\rightarrow X^\circ$ is dominant; after shrinking $X^\circ$, we assume that it is surjective.

The defining isomorphism
$G\times P_X\xrightarrow{\sim}P_X\times_{X^\circ}P_X$ defined by $(g,p)\mapsto(g\cdot p,p)$
shows that two points in one fiber differ by a unique $g\in G$.  If
$p,q\in Z(\C_\infty)$ and $p=g\cdot q$, then $gZ$ and $Z$ are minimal
invariant subvarieties meeting at $p$; hence
Lemma~\ref{l:hs-minimal-leaf} gives $gZ=Z$.  Thus $H(\C_\infty)$ acts
transitively on the geometric points of every fiber of $Z\rightarrow X^\circ$.

At a smooth point $z\in Z(\C_\infty)$, Hasse--Schmidt invariance and Lemma~\ref{l:hs-base-change} imply that
$d(\pi_X|_Z)_z$ is surjective. Since $X^\circ$ is smooth,
the Jacobian criterion shows that $Z\rightarrow X^\circ$ is smooth at $z$, and hence at a general point.  Its smooth locus has image containing a
non-empty open subset of $X^\circ$; on every fiber meeting this locus,
transitivity of $H(\C_\infty)$ implies that the whole fiber is smooth.
After shrinking $X^\circ$, we may therefore assume that $Z\rightarrow X^\circ$ is
smooth.

Under the above isomorphism, the closed subschemes
$H\times Z$ and $Z\times_{X^\circ}Z$ have the same
$\C_\infty$-points.  Both are smooth over the reduced scheme $Z$, and hence reduced.  They
are therefore equal, and
we obtain
$H\times Z\xrightarrow{\simeq}Z\times_{X^\circ}Z$.
The action of the full scheme-theoretic stabilizer on $Z$, compared
with this isomorphism and the ambient principal homogeneous space isomorphism, forces the inclusion
$\Stab_G(Z)\times Z\hookrightarrow G\times Z$ to factor through
$H\times Z$. Since $Z\rightarrow \Spec\C_\infty$ is faithfully flat, this gives
$\Stab_G(Z)\subseteq H$, hence equality. Thus $\Gal_{\mathrm{HS}}(Z)=H$
is reduced, and the  isomorphism $H\times Z\simeq Z\times_{X^\circ}Z$ gives the principal
homogeneous space structure.
\end{proof}

\begin{lem}\label{l:hs-shrinking}
Let $Z\subset P_X$ be minimal invariant and let
$U\subset X^\circ$ be a non-empty open.  Then
$Z_U:=Z\cap P_U$ is minimal invariant and
\[
\Gal_{\mathrm{HS}}(Z_U)=\Gal_{\mathrm{HS}}(Z).
\]
In particular, the Galois conjugacy class is independent of the  open.
\end{lem}

\begin{proof}
Choose a maximal leaf $\lcal$ with
$Z=\overline{\lcal}^{\,\mathrm{Zar}}$.  Let $\lcal_U$ be a non-empty connected component of $\lcal\cap P_U$.  It is a non-empty admissible open of the irreducible leaf $\lcal$, so it is Zariski dense in $\lcal$ by \cite[Lemma~2.2.3]{conrad-irred}.  Its closure in $P_U$ is therefore $Z\cap P_U$, which is minimal by Lemma~\ref{l:hs-minimal-leaf}.  Since the $G$-action fixes the base,
$\Stab_G(Z\cap P_U)=\Stab_G(Z)$. Both stabilizers are reduced by Lemma~\ref{l:hs-galois-principal}, so the claim follows.
\end{proof}

If $Z'$ is another minimal invariant subvariety over $X^\circ$, choose
$p\in Z$ and $q\in Z'$ in one fiber.  The unique $g\in G$ with $q=gp$
satisfies $Z'=gZ$, so
\[
\Gal_{\mathrm{HS}}(Z')
=g\,\Gal_{\mathrm{HS}}(Z)\,g^{-1}.
\]
We denote this conjugacy class by $\Gal_{\mathrm{HS}}(X)$.  By
Lemma~\ref{l:hs-shrinking}, it is independent of the chosen open.
We call $X$ \emph{HS-generic} if this class is $G$, and
\emph{HS-non-generic} otherwise.  A maximal irreducible HS-non-generic
subvariety of
$\A^n_{\C_\infty}$ is called \emph{HS-special}.

\begin{lem}\label{l:hs-quotient}
Let $H\subset G$ be a closed subgroup.  The foliation $\calf_X$ descends
uniquely to an iterative Hasse--Schmidt foliation on
\[
H\backslash P_X=X^\circ\times(H\backslash G),
\]
and its local analytic leaves are the images of local leaves of $P_X$.
\end{lem}

\begin{proof}
Let
$q\colon P_X\rightarrow Y:=H\backslash P_X$
be the quotient map, and set
$D_m:=\Spec(\C_\infty[\mathbf T]/(\mathbf T)^{m+1})$. Consider
$\alpha_m\colon P_X\times D_m\rightarrow P_X$, the $m$-th truncation of the Hasse--Schmidt expansion. The $H$-equivariance gives $\alpha_m(hp,\mathbf T)=h\,\alpha_m(p,\mathbf T)$.
Hence the two pullbacks of $q\circ\alpha_m$ to $(P_X\times D_m)\times_{Y\times D_m}(P_X\times D_m)$
coincide.  Since
$q\times\id_{D_m}\colon P_X\times D_m\rightarrow Y\times D_m$
is faithfully flat and quasi-compact, $q\circ\alpha_m$ descends uniquely
to a morphism
$\overline\alpha_m\colon Y\times D_m\rightarrow Y$
by faithfully flat descent for morphisms
\cite[Lemma~35.13.7]{stacks}.
Uniqueness makes the morphisms $\overline\alpha_m$ compatible under
truncation, so they define a Hasse--Schmidt expansion on $Y$.  The identity
and iterativity relations descend by the same uniqueness. Finally, on a horizontal trivialization $P_X|_U\simeq U\times G$,
the quotient map is $(x,g)\mapsto(x,Hg)$.  Thus the local leaves on $Y$ are precisely the images of the local leaves on $P_X$.
\end{proof}

\subsubsection{Sparsity in positive characteristic}

\begin{dfn}\label{d:p-sparse}
Let $G$ be a smooth connected algebraic group over $\C_\infty$. A reduced irreducible rigid analytic germ $\Sigma\subset(G^{\text{an},e})$
is called a {\em subgroup germ} if the morphism 
\[
\delta\colon G\times G\rightarrow G 
\qquad
(a,b)\mapsto a^{-1}b
\]
satisfies
$\delta(\Sigma\times\Sigma)\subset\Sigma$
as germs at $(e,e)$.
Moreover, we say that $G$ is $p$-{\em sparse} if every subgroup germ $\Sigma\subset(G^{\text{an},e})$ satisfying $\dim\Sigma<\dim G$
is contained in the germ of a proper connected algebraic subgroup, i.e., there exists $H\subsetneq G$ such that $\Sigma\subset(H^{\text{an}},e)$.
\end{dfn}
\begin{eg}\label{eg}
The group $\mathbb G_{a,\C_\infty}^2$ is not $p$-sparse. Indeed, consider, for instance,
\[
f(T):=\sum_{n\ge0}T^{p^{2^n}}
\]
which converges for $T$ small enough. Since every monomial $T^{p^{2^n}}$ is additive,
the germ
\[
\Sigma:=\big\{\big(x,f(x)\big):|x|_\infty<1\big\}\subset\G^{2}_{a,\C_\infty}
\]
is a one-dimensional subgroup germ. Naturally, $f$ is transcendental over $\C_\infty(T)$.
Thus $\Sigma$ is Zariski dense in $\G_a^2$, and therefore it is
not contained in the germ of any proper algebraic subgroup. Hence
$\G_{a,\C_\infty}^2$ is not $p$-sparse.
\end{eg}

We now provide a structure result for analytic automorphisms. 

\begin{lem}
\label{l:analytic-field-automorphism}
Let
$\sigma_2,\ldots,\sigma_n\in\Aut(\C_\infty)$.
Suppose that the germ at the origin of
\[
\Gamma_\sigma:=\bigl\{(x,\sigma_2(x),\ldots,\sigma_n(x)):x\in\C_\infty\bigr\}
\subset\A_{\C_\infty}^n
\]
is a rigid analytic curve germ. Then, for every $i=2,\ldots,n$
\[
\sigma_i=\Frob^{r_i}
\]
for some $r_i\in\Z$.
\end{lem}

\begin{proof}
Let $\calc$ be a reduced irreducible representative of $\Gamma_\sigma$ near the
origin. By the rigid analytic Weierstrass preparation \cite[Exercise~0.A.3]{kato}, after shrinking, the first projection
$\pr_1\colon \calc\rightarrow D$
is finite and surjective onto a disk $D$ centered at $0$.

Since $\calc$ is contained in the graph of $\sigma$, every geometric
fiber of $\pr_1$ consists of a single point. If the finite extension
\[
\Frac(\calo(\calc))/\Frac(\calo(D))
\]
had separable degree greater than $1$, then, after normalizing $\calc$
and shrinking $D$, its separable part would define a finite \'etale cover
of degree greater than $1$. A general geometric fiber would then contain
more than one point, a contradiction: hence the extension is purely
inseparable.

Let $y_i\in \calo(\calc)$ denote the $i$-th coordinate function. For each
$i=2,\ldots,n$, there is therefore an integer $m_i\ge0$ such that
\[
y_i^{p^{m_i}}\in \Frac(\calo(D)).
\]
Indeed, since $\pr_1$ is finite, $\calo(\calc)$ is integral over $\calo(D)$. Thus $y_i^{p^{m_i}}$ is integral over $\calo(D)$, and,
since $\calo(D)$ is normal, it belongs to $\calo(D)$. Thus
\[
y_i^{p^{m_i}}=f_i(x)
\]
for some $f_i\in \calo(D)$. Thus $f_i(x)=\sigma_i(x)^{p^{m_i}}$ on $D$.

Shrinking $D$ if necessary, we  assume that it is stable under addition
and multiplication. Since $\sigma_i$ is a field automorphism,
\[
f_i(x+x')=f_i(x)+f_i(x')
\qquad
\text{and}
\qquad
f_i(xx')=f_i(x)f_i(x')
\]
for $x,x'\in D$.
Write $f_i(T)=\sum_{j\ge0}a_{i,j}T^j$.
Additivity implies that
$f_i(T)=\sum_{s\ge0}b_{i,s}T^{p^s}$.
Comparing coefficients in
$f_i(XY)=f_i(X)f_i(Y)$
shows that $b_{i,s}b_{i,t}=0$ for $s\ne t$, and $b_{i,s}^2=b_{i,s}$. Since $f_i\ne0$, it follows that $f_i(T)=T^{p^{s_i}}$ for some $s_i\ge0$. Therefore
$\sigma_i(x)^{p^{m_i}}=x^{p^{s_i}}$ on $D$, and hence
\[
\sigma_i(x)=\Frob^{\,s_i-m_i}(x)
\]
on $D$. For arbitrary $x\in\C_\infty$, choose $c\ne0$ such that
$c,cx\in D$. Then
$\sigma_i(x)=
\frac{\sigma_i(cx)}{\sigma_i(c)}=
\Frob^{\,s_i-m_i}(x)$.
\end{proof}

\begin{rmk}\label{r:localform}
The same coefficient argument shows that every nonzero  germ $\sigma\colon (\C_\infty,0)\rightarrow(\C_\infty,0)$
which preserves addition and multiplication is equal to
$\sigma(z)=z^{p^m}$ for some $m\ge0$.
\end{rmk}

We record the following simple fact.

\begin{lem}\label{l:unipotent-commuting}
Let $G=\pgl_{2,\C_\infty}$, and let $u\in G$ be a nontrivial
unipotent element. Then $\{g\in G:gu=ug\}$
is a maximal unipotent subgroup of $G$.
\end{lem}

\begin{proof}
After conjugation, assume
$u=\begin{bmatrix}1&1\\0&1\end{bmatrix}$.
Let $g\in G$ be represented by
$A=\begin{bmatrix}a&b\\c&d\end{bmatrix}$.
If $gu=ug$, then
$A\begin{bmatrix}1&1\\0&1\end{bmatrix}=
\lambda
\begin{bmatrix}1&1\\0&1\end{bmatrix}A$
for some $\lambda\in\C_\infty^\times$.
Comparing entries and using $\det A\neq0$ gives
$\lambda=1$, $c=0$, and $a=d$. Hence
$\{g\in G:gu=ug\}=\left\{
\begin{bmatrix}1&x\\0&1\end{bmatrix}:
x\in\C_\infty
\right\}$
which is a maximal unipotent subgroup.
\end{proof}

For an algebraic subgroup $H\subset G$, we denote by $N_G(H):=\{g\in G:gHg^{-1}=H\}$
its normalizer in $G$.

\begin{lem}\label{l:frobenius-goursat}
Assume that $q$ is odd. Let
$S\subset\pgl_{2,\C_\infty}^{\,2}$
be a proper irreducible subgroup germ whose projections onto both factors
are dominant. Then, after possibly exchanging the two factors, there exist
$\delta\in\pgl_2(\C_\infty)$ and $m\ge0$
such that $S$ is contained in the germ at $(e,e)$ of
\[
\Gamma_{\delta,m}
:=
\left\{
\bigl(\gamma,
\operatorname{Inn}(\delta)\circ\Frob^m(\gamma)\bigr):
\gamma\in\pgl_2(\C_\infty)
\right\}.
\]
In particular, $S$ is contained in a proper algebraic subgroup of
$\pgl_{2,\C_\infty}^{\,2}$.
\end{lem}
\begin{proof}
Let us set $G:=\pgl_{2,\C_\infty}$ and $\gotg:=\Lie(G)$.
We recall that, as we consider $q$ odd, $\gotg\simeq\mathfrak{sl}_2$ is simple.

Since $S$ is reduced, its smooth locus is dense. Translating a smooth
point of $S$ to $(e,e)$, we see that $S$ is smooth at the identity.
Set
\[
\mathfrak s:=T_{(e,e)}S\subset\gotg\oplus\gotg.
\]
For $i=1,2$, the subspace $d\pr_i(\mathfrak s)$ is invariant under
$\text{Ad}(\pr_i(S))$. Since $\pr_i(S)$ is dominant, it is Zariski dense in
$G$; hence $d\pr_i(\mathfrak s)$ is an ideal of $\gotg$. At least one tangent projection is nonzero. After exchanging the factors, we may assume that
\[
d\pr_1(\mathfrak s)=\gotg.
\]
The kernel of $d\pr_1|_{\mathfrak s}$ projects to an ideal of $\gotg$. If it
were nonzero, it would be $\gotg$, and then
$\dim\mathfrak s=6$, contradicting the properness of $S$. Thus
$d\pr_1|_{\mathfrak s}\colon \mathfrak s\rightarrow\gotg$ is an isomorphism. By the rigid analytic implicit function theorem \cite[\S10, Proposition~10.8]{abhyankar}, $\pr_1|_S$ is an isomorphism of germs. Therefore $S$ is the graph of a
dominant local homomorphism
\[
\varphi\colon (G,e)\rightarrow(G,e).
\]
Consider
\[
u_+(x):=
\begin{bmatrix}
1&x\\0&1
\end{bmatrix},
\qquad
u_-(y):=
\begin{bmatrix}
1&0\\y&1
\end{bmatrix},
\qquad
t(a):=
\begin{bmatrix}
a&0\\0&a^{-1}
\end{bmatrix}.
\]
The restrictions of $\varphi$ to $U_+$ and $U_-$ are nontrivial:
otherwise, using the local decomposition $U_-TU_+$ (see \cite[Proposition~1.4.11]{conrad})
the image of
$\varphi$ would have dimension at most $2$, contradicting dominance.

Every nontrivial element of $U_\pm$ has order $p$, and every nontrivial element of order $p$ in $G$ is unipotent.
Choose nontrivial elements
$u'_\pm\in\varphi(U_\pm)$.
Since $U_\pm$ are commutative, every element of $\varphi(U_\pm)$
commutes with $u'_\pm$. Set
$U'_\pm:=\{g\in G:gu'_\pm=u'_\pm g\}$. By Lemma~\ref{l:unipotent-commuting}, $U'_\pm$ is a maximal unipotent subgroup of $G$, so that \[
\varphi(U_+)\subset U'_+
\qquad
\text{and}
\qquad
\varphi(U_-)\subset U'_-.
\]
These two subgroups are distinct. Indeed, suppose that
$U'_+=U'_-=:U'$. Since $\varphi|_{U_+}$ is nontrivial, its image is
Zariski dense in the one-dimensional group $U'$. As $T$ normalizes
$U_+$, the group $\varphi(T)$ normalizes $U'$. Therefore $\varphi(T)\subset N_G(U')$,
and hence
$\varphi(U_-TU_+)\subset N_G(U')$.
But $U_-TU_+$ is a neighbourhood of $e$, whereas $N_G(U')$ is a
proper Borel subgroup of $G$, contradicting the dominance of
$\varphi$.

Finally, any ordered pair of distinct maximal unipotent subgroups of
$\pgl_2$ is conjugate to $(U_+,U_-)$. It is well-known that every maximal unipotent subgroup of $G=\pgl_2$ fixes a unique point of
$\mathbb P^1$. Since $U'_+\neq U'_-$, their fixed points are distinct.
As $G$ acts transitively on ordered pairs of distinct points of $\mathbb P^1$, there exists $\delta\in G$ such that $\delta U'_+\delta^{-1}=U_+$ and $\delta U'_-\delta^{-1}=U_-$.
Choose such a $\delta\in G$.
Applying the automorphism
$\id_G\times\operatorname{Inn}(\delta)$
to $G\times G$, we may replace $S$ by its image and $\varphi$ by
$\varphi':=\operatorname{Inn}(\delta)\circ\varphi$.
This preserves all the hypotheses, and now
$\varphi'(U_+)\subset U_+$ and $\varphi'(U_-)\subset U_-$.
We again write $S$ and $\varphi$ for these normalized objects.
Since $T=N_G(U_+)\cap N_G(U_-)$, we also have $\varphi(T)\subset T$.

Thus, locally,
\[
\varphi(u_+(x))=u_+(f(x)),
\qquad
\varphi(u_-(y))=u_-(g(y)),
\qquad
\varphi(t(a))=t(h(a))
\]
where $f,g$ are nonzero additive rigid analytic germs and $h$ is multiplicative.

Consider the following decomposition
\begin{equation}\label{e:decomposition-UTU}
u_+(x)u_-(y)
=
u_-\left(\frac{y}{1+xy}\right)
t(1+xy)
u_+\left(\frac{x}{1+xy}\right)
\end{equation}
which holds whenever $1+xy\neq0$, and follows from \cite[Example~4.2.5 and Theorem~4.2.6(2)]{conrad}. 
By uniqueness of the decomposition \eqref{e:decomposition-UTU}, we have
$t\bigl(h(1+xy)\bigr)=t\bigl(1+f(x)g(y)\bigr)$.
Since both arguments are sufficiently close to $1$, the map $a\mapsto t(a)$ is locally injective\footnote{Indeed, $t(a)=t(b)$ in $\pgl_2$ if and only if
$a^2=b^2$. Since the characteristic is odd, this implies $a=\pm b$;
on a sufficiently small neighbourhood of $1$, the possibility $a=-b$ is excluded.}, and therefore
\begin{equation}\label{e:frobenius-goursat}
h(1+xy)=1+f(x)g(y).
\end{equation}

Define the analytic germ
\[
\sigma(z):=h(1+z)-1.
\]
Choose $y_0\neq0$ sufficiently small with $g(y_0)\neq0$. By
\eqref{e:frobenius-goursat}, we have $\sigma(z)=f(z/y_0)g(y_0)$ for $z$ sufficiently small; hence $\sigma$ is additive. Moreover,
multiplicativity of $h$ gives
\[
1+\sigma(z+w+zw)
=
\big(1+\sigma(z)\big)\big(1+\sigma(w)\big).
\]
Using additivity, we obtain
\[
\sigma(zw)=\sigma(z)\sigma(w).
\]

By Lemma~\ref{l:analytic-field-automorphism} and Remark~\ref{r:localform}, there is $m\ge 0$ such that
$\sigma(z)=z^{p^m}$.
Equation~\eqref{e:frobenius-goursat} now gives
\[
f(x)g(y)=\sigma(xy)=\sigma(x)\sigma(y).
\]

Since $h(1+xy)=1+(xy)^{p^m}$,
equation~\eqref{e:frobenius-goursat} gives
$f(x)g(y)=x^{p^m}y^{p^m}$.
Choose $y_0$ sufficiently small with $g(y_0)\neq0$, and set
$c:=\frac{y_0^{p^m}}{g(y_0)}\in\C_\infty^\times$.
Taking $y=y_0$, we obtain
$f(x)=c\,x^{p^m}$,
and substituting this back gives
$g(y)=c^{-1}y^{p^m}$.
Moreover, for $a$ sufficiently close to $1$,
$h(a)=1+(a-1)^{p^m}=a^{p^m}$.
Since $\C_\infty$ is algebraically closed, choose
$s\in\C_\infty^\times$ such that $s^2=c^{-1}$, and set
\[
\varphi':=\operatorname{Inn}(t(s))\circ\varphi.
\]
Using $t(s)u_+(x)t(s)^{-1}=u_+(s^2x)$ and
$t(s)u_-(y)t(s)^{-1}=u_-(s^{-2}y)$,
and the fact that $t(s)$ centralizes $T$, we obtain
\[
\varphi'(u_+(x))=u_+(x^{p^m}),
\qquad
\varphi'(u_-(y))=u_-(y^{p^m}),
\qquad
\varphi'(t(a))=t(a^{p^m}).
\]
We again write $\varphi$ for $\varphi'$; the conjugation introduced here
will be absorbed into the element $\delta$ at the end of the proof.
Since $U_-TU_+$ is a neighbourhood of $e$, it follows that, before the
inner conjugations,
\[
\varphi=\text{Inn}(\delta)\circ\Frob^m
\]
as germs, for some $\delta\in G(\C_\infty)$.
Therefore $S$ is contained in the germ of $\Gamma_{\delta,m}$.  Let
$G_0:=\pgl_{2,\F_p}$.  The $m$-fold relative $p$-Frobenius of $G_0$,
after the canonical identification of its Frobenius twist with $G_0$,
base-changes to the algebraic group endomorphism $\Frob^m\colon G\rightarrow G$. Hence $\Gamma_{\delta,m}$ is a connected algebraic subgroup of $G\times G$, isomorphic to $G$ by the first projection. In particular, $\dim\Gamma_{\delta,m}=3<6$, so it is proper.
\end{proof}

\begin{lem}\label{l:henselfinite}
Let $f\colon X\rightarrow S$ be a separated morphism of rigid-analytic spaces, and
let $x\in X$ be a point at which $f$ is quasi-finite. Then there are
affinoid neighbourhoods $x\in U\subset X$ and $f(x)\in V\subset S$ such that
\[
f|_U\colon U\rightarrow V
\]
is finite.
\end{lem}

\begin{proof}
This follows from \cite[Theorem~A.1.3]{conrad-ampleness}.
\end{proof}

Next result, perhaps well-known to the experts,  adapts \cite[Lemme~p.444]{loeser-igusa} to our setting.

\begin{lem}\label{l:analytic-image-slice}
Let $f\colon X\rightarrow Y$ be a separated morphism of reduced
rigid-analytic spaces over $\C_\infty$, with $X$ of pure dimension.
Suppose that, on a neighbourhood of $v\in X$, the dimension
$\dim_w f^{-1}(f(w))=c$
is constant. Then $f$ has a well-defined reduced analytic image
germ at $f(v)$, of dimension $\dim X-c$. 
\end{lem}

\begin{proof}
Let $d=\dim X$ and shrink around $v$ so that the hypothesis
holds. We choose $c$ general local linear forms
$\ell=(\ell_1,\ldots,\ell_c)$, vanishing at $v$, whose common zero locus is of dimension $d-c$ and cuts the fiber through $v$ in an isolated point. By Lemma~\ref{l:henselfinite}, the map $(f,\ell)$
admits a finite representative
\[
(f,\ell)\colon U\rightarrow T\times D
\]
where $U$ is an  affinoid neighbourhood of $v$,
$T$ is an affinoid neighbourhood of $f(v)$, and $D$ is a
$c$-dimensional polydisk around $0$. We choose it with fiber
over $(f(v),0)$ supported at $v$. Therefore we obtain that $W:=\bigl(U\cap\ell^{-1}(0)\bigr)_{\mathrm{red}}$  maps finitely and surjectively onto its reduced closed
analytic image $S\subset T$, which is 
of dimension $d-c$.
Let $Z\subset T\times D$ be the reduced closed image of
$(f,\ell)$. For every $a\in T$ with
$Z_a\neq \emptyset$, the induced map
$f^{-1}(a)\cap U\rightarrow Z_a$ is finite and surjective. Its source has dimension $c$ by hypothesis, so $\dim Z_a=c$.
Since a proper closed analytic subset of the
$c$-dimensional polydisk $D$ has smaller dimension,
$(Z_a)_{\mathrm{red}}=D$.
In particular, $(a,0)\in Z$. Surjectivity of $U\rightarrow Z$
therefore gives a point $w\in W$ with $f(w)=a$.
Thus every point of $f(U)$ belongs to $f(W)=S$,
and consequently $f(U)=S$.

Finally, $W\rightarrow S$ is finite and surjective, with fiber over
$f(v)$ supported at $v$.
By \cite[Lemma~A.1.4]{conrad-ampleness}, every admissible neighbourhood of $v$ in $W$ maps onto a neighbourhood of $f(v)$ in $S$. Hence, for every sufficiently small
admissible neighbourhood $U'\subset U$ of $v$, the image $f(U')$ is contained in $S$ and contains a neighbourhood of $f(v)$ in $S$.
Thus $(S,f(v))$ is the required reduced analytic image germ of dimension $d-c$.
\end{proof}

\begin{lem}\label{l:pgl-p-sparse}
Assume that $q$ is odd. For every $n\ge1$, then
$\pgl_{2,\C_\infty}^{\,n}$
is $p$-sparse.
\end{lem}
\begin{proof}
Put $G:=\pgl_{2,\C_\infty}$, and let
$\Sigma\subset\bigl((G^n)^{\mathrm{an}},e\bigr)$
be a subgroup germ with $\dim\Sigma<3n$.

Reduced subgroup germs are smooth at $e$, by translating
a nearby smooth point to the identity.

We first treat $n=1$. Let
$\goth:=T_e\Sigma\subset\gotg:=\Lie(G)$.
If $\goth=0$, then $\Sigma=\{e\}$. Otherwise, conjugation by
$\Sigma$ preserves $\goth$, and hence
$\Sigma\subset\bigl(N_G(\goth)^{\mathrm{an}},e\bigr)$, where  $N_G(\goth):=\{g\in G: \text{Ad}(g)\goth=\goth\}$.
Since $q$ is odd, $\gotg:=\mathfrak{pgl_2}\simeq\mathfrak{sl}_2$ is simple.
Thus $N_G(\goth)\neq G$, for otherwise $\goth$ would be a
nonzero proper ideal of $\gotg$. Therefore $\Sigma$ is contained
in the germ of the proper connected algebraic subgroup
$N_G(\goth)^\circ$.

Now let $n\ge2$ and proceed by induction.
By Lemma~\ref{l:analytic-image-slice}, coordinate projections of
$\Sigma$ have reduced, irreducible, analytic subgroup image germs and they satisfy the fiber dimension formula (as their local fibers are translates of their kernels).

If the image germ of $\pr_{1,\ldots,n-1}|_\Sigma$ is proper, then by induction hypothesis  it is contained in the germ of a proper algebraic subgroup $H'\subset G^{n-1}$.
Then $\Sigma$ is contained in the germ of $H'\times G$.
Similarly, if the image germ of $\pr_n|_\Sigma$ is proper, then, by the case $n=1$ , we have that
$\Sigma$ is contained in the germ of
$G^{n-1}\times H_n$.
Henceforth we assume that the image germs of
$\pr_{1,\ldots,n-1}|_\Sigma$ and $\pr_n|_\Sigma$ coincide with
$\bigl((G^{n-1})^{\mathrm{an}},e\bigr)$ and
$(G^{\mathrm{an}},e)$, respectively.

We now adapt the kernel argument of the classical Goursat lemma.
First, every reduced normal subgroup germ of $G$ is either
$\{e\}$ or the full germ of $G$: indeed, it is smooth\footnote{Indeed, reducedness gives a dense smooth locus and translation by the inverse of a sufficiently nearby smooth
point identifies the germ at that point with the germ at $e$.} and by normality its tangent space is an ideal of $\gotg\simeq\mathfrak{sl}_2$.
Simplicity therefore gives dimension either $0$ or $3$, which proves the assertion.

Consider the subgroup germ of $G^{n-1}$ defined as $K:=\bigl(\Sigma\cap(G^{n-1}\times\{e\})\bigr)_{\mathrm{red}}$. By (local) surjectivity, every $a\in G^{n-1}$ near $e$ has a lift $(a,b)\in\Sigma$ near $e$. For $k\in K$ near $e$, we have
$(a,b)(k,e)(a,b)^{-1}=(aka^{-1},e)\in\Sigma$.
Thus $aka^{-1}\in K$, so $K$ is normal.
The fiber dimension formula for $\pr_n|_\Sigma$ gives
$\dim K=\dim\Sigma-3<3(n-1)$.
Thus $K$ is proper, so, for some $i$, the factor $G_i\subset G^{n-1}$ is not contained in $K$,
otherwise $K$ would contain the product of all the factors.

The reduced subgroup germ $(K\cap G_i)_{\mathrm{red}}$ is normal
in $G_i$, so it is either trivial or full. Since $G_i\not\subset K$, it must be trivial.
For $k\in K$ and $g\in G_i$ sufficiently close to the identity,
the commutator $kgk^{-1}g^{-1}$ belongs to $K\cap G_i$, so it is the identity; thus $\pr_i(K)$ centralizes the full germ of $G$.
Since a neighbourhood of the identity is Zariski dense in $G$
and $G$ has trivial center, we see that
$\pr_i(K)=\{e\}$.

Every point of $\Sigma$ near $e$ whose $n$-th coordinate is $e$
belongs to $K$, so its $i$-th coordinate is also $e$,
since $\pr_i(K)=\{e\}$.
Thus $\ker(\pr_n|_\Sigma)$ and $\ker(\pr_{i,n}|_\Sigma)$
have the same points near $e$, and hence the same reduced germ.
Therefore
$\bigl(\ker(\pr_{i,n}|_\Sigma)\bigr)_{\mathrm{red}}=K$, and the fiber dimension formula gives
$\dim \pr_{i,n}(\Sigma)=\dim\Sigma-\dim K=3$.
Both coordinate projections of $\pr_{i,n}(\Sigma)$ are full. 
Hence by Lemma~\ref{l:frobenius-goursat}  $\pr_{i,n}(\Sigma)$ sits in the germ of a graph $\Gamma_{\delta,m}$.
The inverse image of this algebraic subgroup under
$\pr_{i,n}\colon G^n\rightarrow G^2$
contains $\Sigma$ and is isomorphic to
$G^{n-2}\times\Gamma_{\delta,m}$.
It is therefore a proper connected algebraic subgroup of $G^n$.
\end{proof}

\subsection{Abstract Ax--Schanuel}

\subsubsection{Transverse image germs and enlarged foliation}\label{ss:enlarged}

Let $\pi\colon P\rightarrow Y$ be a principal homogeneous left $G$-space equipped with its $G$-equivariant Hasse--Schmidt foliation $\calf$. We work over a smooth affine base chart with fixed \'etale coordinates, as in
Lemma~\ref{l:hs-base-change}, and set $r:=\dim Y$. 
By Proposition~\ref{p:atypical-closed}, there is a non-empty smooth
Zariski open subset $V^\circ\subset V$ on which
\[
c:=\dim_v(V\cap\lcal_v)
\]
is constant, where $\lcal_v$ is the leaf through $v$.
We shall work at points of this fixed Zariski open subset.

For $v\in V^\circ$, choose a left $G$-equivariant horizontal local
trivialization
\[
\tau\colon P|_B\xrightarrow{\sim}B\times G,
\]
over an analytic neighborhood $B$ of $\pi(v)$. In these coordinates, $G$ acts on the second factor by left multiplication and the local Hasse--Schmidt leaves are the sets $B\times\{g\}$. For $\tau(v)=(\pi(v),g_v)$,
define the (normalized) {\em transverse projection}
\[
\lambda_v(x):=\pi_G(\tau(x))g_v^{-1}
\]
where $\pi_G$ is the projection onto the group coordinate.
Then $\lambda_v(v)=e$, and the fibers of $\lambda_v$ are the local  leaves.  Two such horizontal trivializations differ by
right multiplication by a constant element of $G$; hence the germ of $\lambda_v$ is independent of the chosen trivialization.

Let us identify $B$ with an analytic neighborhood in $\A^{r,\mathrm{an}}_{\C_\infty}$ using the chosen base coordinates.
For $\mathbf t=(t_1,\ldots,t_r)$, we consider
\[
\theta_{\mathbf t}(v)
:=
\tau^{-1}\bigl(\pi(v)+\mathbf t,g_v\bigr).
\]
The Hasse--Schmidt expansion gives
\[
f\bigl(\theta_{\mathbf t}(v)\bigr)
=
\sum_{\alpha\in\N^r}
D_\alpha(f)(v)\mathbf t^\alpha
\]
for every  regular function $f$.
We also set
\[
\boldsymbol g_v(\mathbf t,g)
:=g\cdot\theta_{\mathbf t}(v)
\]
which gives local analytic coordinates on $P$ near $v$. Indeed, in the horizontal trivialization,
$\tau\bigl(\boldsymbol g_v(\mathbf t,g)\bigr)
=
\bigl(\pi(\theta_{\mathbf t}(v)),\,g g_v\bigr)$. The first coordinate parametrizes the base, and the second is right multiplication by $g_v$; both operations are  indeed
invertible, by substracting $\pi(v)$ in the base coordinate and multiply by $g_v^{-1}$ in the group coordinate. Thus every $x$  close to $v$ has a unique expression $x=g\cdot\theta_{\mathbf t}(v)$;
its base point determines $\mathbf t$, and
$g=\lambda_v(x)$. In particular,
$\lambda_v\bigl(\boldsymbol g_v(\mathbf t,g)\bigr)=g$.
By $G$-equivariance, fixing $g$ and varying $\mathbf t$ parametrizes a local leaf.

Choose \'etale coordinates
$\mathbf u=(u_1,\ldots,u_{\dim G})$ centered at $e\in G$, and let $f_1,\ldots,f_s$
generate the ideal of $V$ on an affine chart.
Restrict $V^\circ$ to this affine open.

We now define an {\em enlarged} Hasse--Schmidt foliation
$\widetilde{\calf}$ on $P\times G$.
Take the product foliation on $P\times G$, whose local leaves are $\lcal_v\times G$: the first coordinate moves along $\lcal_v$, while the second varies  near $e$.
Use the Hasse--Schmidt operators of $\calf$ in the first
factor and Taylor expansion in $\mathbf u$ in the second. Transport this product
foliation through the algebraic isomorphism
$P\times G\rightarrow P\times G$ defined by $(x,g)\mapsto(g\cdot x,g)$,
and denote the resulting foliation by $\widetilde{\calf}$.
We consider each $f_i$ as a function on $P\times G$ simply by
$(x,g)\mapsto f_i(x)$.
Enlarging the constructions of Section~\ref{ss:hs-foliation},
we obtain  the  intersection ideal of $V\times G$ with the
leaf of $\widetilde{\calf}$ through $(v,e)$ 
\begin{equation}\label{e:mixed-leaf-ideal}
I_v:=
\bigl(
\operatorname{jet}^{\infty}_{\widetilde{\calf},(v,e)}(f_i):
1\le i\le s
\bigr)
\end{equation}

Since this leaf is parametrized by
\[
(\mathbf t,g)\mapsto
\bigl(g\cdot\theta_{\mathbf t}(v),g\bigr),
\]
restricting $f_i$ to it gives
$f_i(g\cdot\theta_{\mathbf t}(v))$.
This shows that
$\operatorname{jet}^{\infty}_{\widetilde{\calf},(v,e)}(f_i)$
is  the  expansion of
$f_i\circ\boldsymbol g_v$ at $(\mathbf t,g)=(0,e)$.
More explicitly, first expand the algebraic action at $g=e$:
\[
f_i(g\cdot x)
=
\sum_\beta c_{i,\beta}(x)\mathbf u(g)^\beta.
\]
Substituting $x=\theta_{\mathbf t}(v)$ and expanding along
the original Hasse--Schmidt leaf gives
\begin{equation}\label{e:jet-alg}
\operatorname{jet}^{\infty}_{\widetilde{\calf},(v,e)}(f_i)
=
\sum_{\alpha,\beta}
D_\alpha(c_{i,\beta})(v)\,
\mathbf t^\alpha\mathbf u^\beta.
\end{equation}
These coefficients are regular functions of $v$ on $V^\circ$.
Before evaluation at $v$, the series in \eqref{e:jet-alg} have coefficients
$D_\alpha(c_{i,\beta})|_{V^\circ}$ and belong to
$\calo_V(V^\circ)[\![\mathbf t,\mathbf u]\!]$.
The ideal they generate specializes to $I_v$ after evaluation at $v$.
At each finite order, their truncations are polynomial equations with coefficients regular on $V^\circ$.

For $N\ge0$, let
$j_N\colon
\C_\infty[\![\mathbf t,\mathbf u]\!]
\rightarrow
\C_\infty[\mathbf t,\mathbf u]/
(\mathbf t,\mathbf u)^{N+1}$
be the $N$-truncation map, and consider $I_{v,N}:=j_N(I_v)$. 
For $F\in\calo^{\mathrm{an}}_{G,e}$, we define
\[
J_v:=
\left\{
F\in\calo^{\mathrm{an}}_{G,e}:
j_N(F)\in I_{v,N}\ \text{for every }N\ge0
\right\}.
\]
By Krull's intersection theorem, this condition is equivalent to $F(\mathbf u)\in I_v$. Since $\boldsymbol g_v$ is an isomorphism and $\lambda_v(\boldsymbol g_v(\mathbf t,g))=g$, it is also
equivalent to $F\circ\lambda_v=0$ on $(V^{\mathrm{an}},v)$. So we can also write $J_v=
\ker(\calo^{\mathrm{an}}_{G,e}
\xrightarrow{\lambda_v^*}
\calo^{\mathrm{an}}_{V,v})$.
We define
\[
\Sigma_v(V)\subset (G^{\mathrm{an}},e)
\]
as the closed rigid-analytic subspace germ corresponding to 
$\calo^{\mathrm{an}}_{G,e}
\rightarrow
\calo^{\mathrm{an}}_{G,e}/J_v$. Since $v\in V^\circ$ is smooth,
$\calo^{\mathrm{an}}_{V,v}$ is a domain.
As $J_v$ is a kernel, we get that $J_v$ is prime, so $\Sigma_v(V)$ is reduced and irreducible.

\begin{lem}\label{l:group-projection-drop}
Let $V\subset P$ be irreducible, let $v\in V^\circ$ , and let $\lcal_v$ be the Hasse--Schmidt leaf through $v$. If $\dim_v(V\cap\lcal_v)\ge m$, then
\[
\dim\Sigma_v(V)\le \dim V-m.
\]
\end{lem}

\begin{proof}
Choose a smooth neighborhood 
$U\subset(V^\circ)^{\mathrm{an}}$ of $v$ on which
$\lambda_v$ is defined.
Its fibers are intersections with local leaves, so, for $x\in U$, 
\[
\dim_x\bigl((\lambda_v|_U)^{-1}(\lambda_v(x))\bigr)
=\dim_x(V\cap\lcal_x)=c.
\]
Lemma~\ref{l:analytic-image-slice} applies and gives a reduced analytic image germ of dimension $\dim V-c$.
An analytic function germ vanishes on this image if and only if its pullback vanishes on $(U,v)$. Hence its defining ideal is
$\ker(\lambda_v^*)=J_v$, and the image germ is
$\Sigma_v(V)$. Therefore
\[
\dim\Sigma_v(V)=\dim V-c\le\dim V-m,
\]
since $c=\dim_v(V\cap\lcal_v)\ge m$.
\end{proof}

\begin{lem}\label{l:translated-image-jets-algebraic}
Let $V\subset P$ be irreducible. There exists a non-empty smooth Zariski open subset $V^\circ\subset V$ such that,
for every fixed $v_0\in V^\circ$, setting
$\Sigma_0(V):=\Sigma_{v_0}(V)$, the locus
\[
\bigl\{v\in V^\circ:
\Sigma_v(V)\subseteq\Sigma_0(V)\bigr\}
\]
is Zariski closed in $V^\circ$.
\end{lem}

\begin{proof}
For $V^\circ$ as above, we have that
$c=\dim_v(V\cap\lcal_v)$ is constant, and we set $d:=\dim V$. By the proof of
Lemma~\ref{l:group-projection-drop}, every $\Sigma_v(V)$ is reduced and irreducible of dimension $d-c$. Fix $v_0\in V^\circ$.

Consider $\widetilde{\calf}\times\widetilde{\calf}$
on $(P\times G)^2$ and the closed subvariety
\[
W:=
\bigl\{((x_0,g),(x,g)):x_0,x\in V,\ g\in G\bigr\}.
\] 
The intersection of $W$ with the product leaf through $((v_0,e),(v,e))$ is the germ
\[
\bigl\{(x_0,x)\in
(V^{\mathrm{an}},v_0)\times(V^{\mathrm{an}},v):\lambda_{v_0}(x_0)=\lambda_v(x)\bigr\}.
\]
As in the proof of Lemma~\ref{l:analytic-image-slice}, consider $c$-tuples of functions $\ell_0$ and $\ell$ vanishing at $v_0$ and $v$ respectively, such that 
\[
(\lambda_{v_0},\ell_0)\colon
(V^{\mathrm{an}},v_0)\rightarrow
\Sigma_0(V)\times(\A^{c,\mathrm{an}}_{\C_\infty},0)
\quad\text{and}\quad
(\lambda_v,\ell)\colon
(V^{\mathrm{an}},v)
\rightarrow\Sigma_v(V)\times(\A^{c,\mathrm{an}}_{\C_\infty},0)
\]
are represented on a sufficiently small neighborhood by  finite surjective maps whose fibers over $(e,0)$ consist of $v_0$ and $v$, respectively.

Restrict the product of these maps to the germ at $(v_0,v)$ of  
$\bigl\{(x_0,x):\lambda_{v_0}
(x_0)=\lambda_v(x)\bigr\}$.
We can write the product map as
$(x_0,x)\mapsto
\bigl(\lambda_{v_0}(x_0),\ell_0(x_0),\ell(x)\bigr)$ taking values in
\[
\bigl(\Sigma_0(V)\cap\Sigma_v(V)\bigr)
\times(\A^{2c,\mathrm{an}}_{\C_\infty},0).
\]
This map is the base change\footnote{Which we recall preserves finite morphisms.}
of the product of the two finite surjections, hence is finite and surjective. Its fiber over $(e,0,0)$ is supported at $(v_0,v)$, so it induces the required finite surjection of germs. Therefore
\begin{equation}\label{e:dim}
\dim\bigl(
(\widetilde{\calf}\times\widetilde{\calf})\cap W
\bigr)_{((v_0,e),(v,e))}
=2c+\dim\bigl(\Sigma_0(V)\cap\Sigma_v(V)\bigr).
\end{equation}
Since $\Sigma_v(V)$ is reduced and irreducible of
dimension $d-c$, we obtain
\begin{equation}\label{e:dim2}
\Sigma_v(V)\subseteq\Sigma_0(V)
\quad\text{if and only if}\quad
\dim\bigl(\Sigma_0(V)\cap\Sigma_v(V)\bigr)=d-c
\end{equation}
By \eqref{e:dim} and \eqref{e:dim2}, containment holds when the leaf intersection has dimension $(d-c)+2c=d+c$. As the intersection of the image germs has dimensions at most $d-c$, that dimension cannot exceed $d+c$. Thus
\[
\Sigma_v(V)\subseteq\Sigma_0(V)
\quad\text{if and only if}\quad
((v_0,e),(v,e))
\in A_{d+c}
\bigl(W,\widetilde{\calf}\times\widetilde{\calf}\bigr)
\]
where $A_{d+c}
\bigl(W,\widetilde{\calf}\times\widetilde{\calf}\bigr)=\{w\in W: \dim(W\cap \widetilde{\calf}\times\widetilde{\calf})_w\ge d+c\}$, which, by Proposition~\ref{p:atypical-closed}, is Zariski closed.
We conclude by pulling it back along the algebraic map
$V^\circ\rightarrow W$ given by $v\mapsto((v_0,e),(v,e))$.
\end{proof}

\begin{lem}\label{l:translated-germ-constant}
Let $V\subset P$ be irreducible, and let
$\calv\subset V^{\mathrm{an}}$ be an irreducible Zariski dense analytic
subset contained in one Hasse--Schmidt leaf. After possibly shrinking $\calv$, there exists a subgroup germ
$\Sigma\subset(G^{\mathrm{an}},e)$
such that 
\[
\Sigma_v(V)=\Sigma
\]
for every $v\in V$.
\end{lem}

\begin{proof}
Let $V^\circ$ be the fixed open subset of
Lemma~\ref{l:translated-image-jets-algebraic}.
Fix $v_0\in\calv\cap V^\circ$ and set
$\Sigma:=\Sigma_{v_0}(V)$.
Choose a neighbourhood $U$ of $e$ and a reduced irreducible representative
$S=V(F_1,\ldots,F_s)\subset U$ of $\Sigma$ such that $SS^{-1}\subset U$.
By Lemma~\ref{l:analytic-image-slice}, there is a neighbourhood $W\subset(V^\circ)^{\mathrm{an}}$ of $v_0$,
contained in a horizontal chart, such that
$\lambda_{v_0}(W)\subset S$ and $\lambda_{v_0}(W)$ contains a neighbourhood of $e$ in $S$.

We shrink, if necessary, $\calv$ to a non-empty  open contained in $W$ and in the local leaf through $v_0$; it remains Zariski dense in $V$.
For every $v\in\calv$, we have
$\lambda_{v_0}(v)=e$ and $\lambda_v=\lambda_{v_0}$.
Hence $\Sigma_v(V)\subseteq(S,e)=\Sigma$.
Both germs are reduced and irreducible of dimension
$\dim V-c$, so $\Sigma_v(V)=\Sigma$.

By Lemma~\ref{l:translated-image-jets-algebraic}, the locus
$\{w\in V^\circ:\Sigma_w(V)\subseteq\Sigma\}$ is Zariski
closed. It contains the Zariski dense subset $\calv$,
so it is all of $V^\circ$. Equality of dimensions gives once more
\[
\Sigma_w(V)=\Sigma
\qquad\text{for every }w\in V^\circ.
\]
We now show that $\Sigma$ has a group structure. For $a\in S^{\mathrm{sm}}$ sufficiently close to $e$,
choose $w\in W$ with $\lambda_{v_0}(w)=a$.
By Lemma~\ref{l:analytic-image-slice}, the local image of $\lambda_{v_0}$ at $w$ is a reduced analytic subgerm of $(S,a)$ of dimension $\dim V-c$. Since $(S,a)$ is irreducible of the same dimension, this local image is  $(S,a)$. As $a=\lambda_{v_0}(w)=g_wg_{v_0}^{-1}$, for $x$ near
$w$ we have
$\lambda_{v_0}(x)a^{-1}=
\pi_G(\tau(x))g_{v_0}^{-1}g_{v_0}g_w^{-1}=
\lambda_w(x)$.
Thus normalizing at $w$ translates the image germ $(S,a)$ by $a^{-1}$, and  we obtain
\[
(Sa^{-1},e)=\Sigma_w(V)=\Sigma.
\]
Thus $F_i(ba^{-1})$ vanishes for $b$ near $a$ in $S$. The identity principle, first in $b$ and then in $a$, gives $F_i(ba^{-1})=0$ for all $a,b\in S$. Hence $S$ is locally stable under right division. Taking $b=e$ gives inversion, and therefore also
stability under $(a,b)\mapsto a^{-1}b$.
\end{proof}

\begin{thm}\label{t:ax-schanuel}
Assume that $q$ is odd. Let $V_0\subset \widetilde{\caly}^{\,n}$ be an irreducible algebraic subvariety, let $\lcal$ be a product Ramanujan leaf, and let $\calv\subset V_0\cap \lcal$ be an irreducible analytic subvariety. Consider
$V:=\overline{\calv}^{\mathrm{Zar}}\subset V_0$. If
\[
\dim V<\dim \calv +3n
\]
then the Zariski closure of $\pi_j^n(V)$ in $\A^n_{\C_\infty}$ is contained in
a proper  HS-special subvariety.
\end{thm}

\begin{proof}
Set $G:=\pgl_{2,\C_\infty}^{\,n}$.  On the actual Ramanujan leaf one has $h_i\neq0$ for every $i$.  If $g_i$ vanishes
identically on $\calv$ for some $i$, then $j_i=0$ on
$X_0:=\overline{\pi_j^n(V)}^{\,\mathrm{Zar}}$.  Then
$X_0\subset\{x_i=0\}$, so $X_0$ is HS-non-generic by the
convention above and is contained in an HS-special subvariety by
Noetherianity, so the conclusion follows.  We may therefore replace
$\calv$ by a non-empty analytic open on which every $g_ih_i\neq0$.

Apply the finite quotient \eqref{e:mu2-quotient} coordinatewise.  After
shrinking the analytic branch, it is an isomorphism on that branch.  Set
\[
\calv^\sharp:=\vartheta^n(\calv)
\qquad
\text{and}
\qquad
V^\sharp:=\overline{\calv^\sharp}^{\,\mathrm{Zar}}
\]
where $\vartheta$ is the map defined in \eqref{e:mu2-quotient}. Because $\vartheta^n$ is finite,
$\dim\calv^\sharp=\dim\calv$
and $\dim V^\sharp=\dim V$,
and, by \eqref{e:quotient-j},
\[
X_0:=\overline{\pi_j^n(V)}^{\,\mathrm{Zar}}
=
\overline{(\pi_j^\sharp)^n(V^\sharp)}^{\,\mathrm{Zar}}.
\]
No coordinate of $X_0$ is identically zero, so choose a non-empty smooth
open $X^\circ\subset X_0\cap\G_m^n$.  For the remainder of the argument we
write $X:=X^\circ$; this shrinking does not change the Hasse--Schmidt
Galois group of $X_0$.  We use the frame space
\[
P:=P_X=X\times G\rightarrow X
\]
from \eqref{e:frame-space-X}.  Let $\widehat V\subset P$ be the
transform of $V^\sharp$ under the birational map of Lemma~\ref{l:local-principal-model}, and let $\widehat{\calv}$ be the
corresponding analytic branch.  Birationality at the general point gives
\[
\dim\widehat V=\dim V
\qquad
\text{and}
\qquad
\dim\widehat{\calv}=\dim\calv
\]
while $\widehat{\calv}$ remains Zariski dense in $\widehat V$ and lies in
one Hasse--Schmidt leaf.  Replacing $(V,\calv)$ by
$(\widehat V,\widehat{\calv})$, Corollary~\ref{c:zariski-closure-leaf}
gives
\[
V\subset A_{\dim \calv}(V,\calf_X).
\]

If $\dim X=0$, then the Hasse--Schmidt leaves over $X$ are points, so
$\Gal_{\mathrm{HS}}(X)=\{e\}$ and $X$ is HS-non-generic. We may
therefore assume that $\dim X>0$.
After shrinking $\calv$ as in
Lemma~\ref{l:translated-germ-constant}, let
\[
\Sigma\subset(G^{\mathrm{an}},e)
\]
be the subgroup germ satisfying $\Sigma_v(V)=\Sigma$ for every
$v\in\calv$. Choose
$v_0\in\calv$ general, and let $\lcal_{v_0}$ be the Hasse--Schmidt leaf
through $v_0$.  Since $V\subset A_{\dim\calv}(V,\calf_X)$, we have
$\dim_{v_0}(V\cap\lcal_{v_0})\ge \dim\calv$.
Thus Lemma~\ref{l:group-projection-drop}  yields
\[
\dim\Sigma\le\dim V-\dim\calv<3n=\dim G.
\]
Applying Lemma~\ref{l:pgl-p-sparse} to the associated reduced rigid-analytic germ, there exists a proper connected algebraic subgroup $H\subsetneq G$ such that $\Sigma\subset(H^{\mathrm{an}},e)$.  Let
$\rho\colon P\rightarrow H\backslash P$
be the quotient by the restricted left action of $H$.  In a horizontal trivialization it is
\[
B\times G\rightarrow B\times(H\backslash G),
\qquad
(x,g)\mapsto(x,Hg).
\]
By Lemma~\ref{l:hs-quotient}, $\calf_X$ descends to an algebraic Hasse--Schmidt foliation $\calf_{H\backslash P}$ on $H\backslash P$; in
the displayed trivialization its leaves are $B\times\{Hg\}$.

Consider
\[
Y:=\overline{\rho(V)}^{\mathrm{Zar}}\subset H\backslash P.
\]
Since $\calv$ is Zariski dense in $V$, its image is Zariski dense in $Y$.  Choose $v_1\in\calv$ such that $\rho(v_1)$ is a general point of $Y$.  We have $\Sigma_{v_1}(V)=\Sigma\subset(H^{\mathrm{an}},e)$, so the germ of $Y$ at $\rho(v_1)$ is contained in a quotient leaf.  The projection $Y\rightarrow X$ is dominant, hence $\dim Y\ge\dim X$; the quotient leaf has dimension $\dim X$.  Therefore
\[
\dim Y=\dim X.
\]
It follows that a dense open subset of $Y$ is contained in
$A_{\dim X}(Y,\calf_{H\backslash P})$.  By
Proposition~\ref{p:atypical-closed}, this locus is closed in $Y$ and
therefore
\[
Y\subset A_{\dim X}(Y,\calf_{H\backslash P}).
\]
Lemma~\ref{l:hs-invariance} now shows that $Y$ is Hasse--Schmidt
invariant.  Therefore $\rho^{-1}(Y)\subset P$ is Hasse--Schmidt
invariant.  It is proper.  Indeed, for a general $x\in X$, the fiber
$Y_x\subset H\backslash G$ is finite, and hence
$\rho^{-1}(Y)_x\subset G$ is a finite union of left $H$-cosets.  Since
$\dim H<\dim G$, this finite union is a proper subset of $G$.

Fix $u\in \rho^{-1}(Y)$, let $\lcal_u$ be the Hasse--Schmidt leaf through $u$,
and set
\[
Z:=\overline{\lcal_u}^{\mathrm{Zar}}\subset \rho^{-1}(Y).
\]
By Lemma~\ref{l:hs-base-change}, every maximal Hasse--Schmidt leaf over
$X$ dominates $X$, and
Lemma~\ref{l:hs-minimal-leaf} shows that $Z$ is minimal
Hasse--Schmidt invariant. By Lemma~\ref{l:hs-galois-principal}, the map
$Z\rightarrow X$ is a principal homogeneous $\Gal_{\mathrm{HS}}(Z)$-space over $X$, where
$\Gal_{\mathrm{HS}}(Z)\subseteq G$.
If $\Gal_{\mathrm{HS}}(Z)=G$, then every fiber $Z_x$ equals the whole principal homogeneous
$G$-space $P_x$,
so $Z=P$. This contradicts
$Z\subset \rho^{-1}(Y)\subsetneq P$.
Therefore $\Gal_{\mathrm{HS}}(Z)\subsetneq G$,
and hence $X_0$ is HS-non-generic.

Among the irreducible HS-non-generic subvarieties containing $X_0$,
pick one of maximal dimension; it is maximal for inclusion, so it is
HS-special. This HS-special subvariety is proper because $\A^n_{\C_\infty}$ is HS-generic: the quotient lift of the full branch is Zariski dense in the frame space and is stabilized by $(\Gamma^+)^n$, which is Zariski dense in $G$ by Lemma~\ref{l:dense}. Therefore $X_0$, which consists of the Zariski closure
of the original $j$-image, is contained in a proper HS-special subvariety.
\end{proof}

\subsection{From HS-special to weakly-special}
\label{s:HS-to-weakly-corrected}

Henceforth we shall assume that $q$ is odd. 

This final section is devoted to the proof of the following result.

\begin{thm}\label{t:HS-weakly}
Every proper HS-special subvariety of $\A^n_{\C_\infty}$ is contained in a proper weakly-special subvariety.
\end{thm}

Let us briefly  recall the architecture of the proof. We first show that the analytic monodromy of a proper HS-special
subvariety has proper Zariski closure, while each of its nonconstant
coordinate projections is Zariski dense.  Iterated Goursat then reduces
the problem to a proper two-coordinate projection. This explains why
the main arithmetic argument below is carried out in
$\pgl_{2,\C_\infty}^{\,2}$: Goursat identifies such a projection with a
Frobenius-twisted graph, and congruence counting eliminates the twist.
The Tate conjecture (a theorem in our setting) turns the resulting relation into an isogeny, hence a Hecke relation.  The proof of the theorem is completed in Section~\ref{ss:proof-HS-weakly}.

\subsubsection{Analytic monodromy}

Put $G:=\pgl_{2,\C_\infty}^{\,n}$, and let $\Gamma^+$ denote the image of
$\Sl_2(A)$ in $\pgl_2(A)$.  On the
open $jD_1j\neq0$, consider
\[
u(z):=\frac{h(z)}{g(z)}
\qquad
\text{and}
\qquad
\widetilde C^\sharp(z):=(z,E(z),u(z),g(z)h(z))
\in\widetilde{\caly}^{\,\sharp}.
\]
Using the inverse in Lemma~\ref{l:local-principal-model}, define the frame
lift
\[
\widehat C(z)
:=
\left(
 j(z),
 \left[
 \begin{matrix}
 u(z)^{-1}(1-\widetilde\pi E(z)z)&z\\
 -\widetilde\pi E(z)u(z)^{-1}&1
 \end{matrix}
 \right]
\right)
\in \G_m\times\pgl_{2,\C_\infty}.
\]
For $z=(z_1,\ldots,z_n)$, write
$\widehat C^n(z):=(\widehat C(z_i))_{i=1}^n$.
Let $\Omega^\times:=\{z\in\Omega:j(z)\neq0\}$
and let
\[
\mathcal L^{\mathrm{fr}}_0
:=
\widehat C(\Omega^\times)
\subset (P^{\mathrm{fr}})^{\mathrm{an}}
\]
be the basic global frame leaf.  Let
$X\subset\A^n_{\C_\infty}$ be irreducible and not contained in a coordinate hyperplane $x_i=0$, let
$P_X=X^\circ\times G$ be the frame space
\eqref{e:frame-space-X}, and let
\[
\cala\subset
\boldsymbol j^{-1}\bigl((X^\circ)^{\mathrm{an}}\bigr)
\]
be an irreducible analytic component.
The frame lift identifies
$\boldsymbol j^{-1}\bigl((X^\circ)^{\mathrm{an}}\bigr)$ with 
$P_X^{\mathrm{an}}\cap\bigl(\mathcal L^{\mathrm{fr}}_0\bigr)^n$.
Hence $\widehat C^n(\cala)$ is a global leaf of $\calf_X$.  We set
\[
\lcal_\cala:=\widehat C^n(\cala)
\subset P_X^{\mathrm{an}}.
\]

Define the {\em analytic monodromy} and its Zariski closure as
\[
\Delta_\cala^+:=\Stab_{(\Gamma^+)^n}(\cala)
\qquad
\text{and}
\qquad
H_\cala:=\overline{\Delta_\cala^+}^{\,\mathrm{Zar}}
\subset G.
\]
Equivariance of $\Phi$, together with \eqref{e:rt}, gives, wherever the
frame lift is defined,
$\widehat C^n(\gamma z)=\gamma\cdot\widehat C^n(z)$ for $\gamma\in(\Gamma^+)^n$.
Hence
\begin{equation}\label{e:rt-cala}
\gamma\cdot\lcal_\cala=\lcal_{\gamma\cala}.
\end{equation}

\begin{lem}\label{l:HS-monodromy-proper}
Let $X\subset\A^n_{\C_\infty}$ be irreducible, HS-non-generic, and not
contained in a coordinate hyperplane.  For every irreducible analytic
component $\cala\subset\boldsymbol j^{-1}(X)$,
\[
H_\cala\subsetneq
\pgl_{2,\C_\infty}^{\,n}.
\]
\end{lem}

\begin{proof}
All closures are taken in the frame space $P_X$.  Here
$\lcal_\cala:=\widehat C^n\bigl(\cala\cap
\boldsymbol j^{-1}((X^\circ)^{\mathrm{an}})\bigr)$.
The restriction is nonempty, and the lifts of its irreducible components
are leaves with the same Zariski closure in $P_X$, by the identity principle
on $\cala$. The corresponding Galois group is independent of $X^\circ$ by Lemma~\ref{l:hs-shrinking}. Set $Z_\cala:=\overline{\lcal_\cala}^{\,\mathrm{Zar}}\subset P_X$.
By Lemma~\ref{l:hs-minimal-leaf}, $Z_\cala$ is minimal invariant, and
$\Gal_{\mathrm{HS}}(Z_\cala)$ represents
$\Gal_{\mathrm{HS}}(X)$.  If $\gamma\in\Delta_\cala^+$, then
\eqref{e:rt-cala} gives
$\gamma\cdot\lcal_\cala=\lcal_\cala$.  Since left translation on $P_X$ is regular,
we have $\gamma Z_\cala=Z_\cala$,
so that $\Delta_\cala^+\subset\Gal_{\mathrm{HS}}(Z_\cala)$. The latter is algebraic by Lemma~\ref{l:hs-galois-principal}; taking
Zariski closures and using HS-non-genericity yields
$\overline{\Delta_\cala^+}^{\,\mathrm{Zar}}
\subset
\Gal_{\mathrm{HS}}(Z_\cala)
\subsetneq G$.
\end{proof}

\subsubsection{Running notation}\label{running}

Choose once and for all a non-zero proper ideal $\gotn\subset A$ such that
the full level-$\gotn$ moduli problem is fine; see
\cite[Chapter~II]{gekeler}. Set
\[
\kappa:=\kappa_\gotn\colon
\widetilde\Gamma(\gotn)\xrightarrow{\sim}\Gamma(\gotn).
\]

For every non-zero ideal $\gota\subset A$, abbreviate
$\nu_\gota
:=
\nu_{\gotn\gota,\gotn}
\colon
Y(\gotn\gota)\rightarrow Y(\gotn)$.
Then
\[
(\nu_\gota^N)^{\mathrm{an}}\circ\pi_{\gotn\gota}^N
=
\pi_\gotn^N.
\]

We now fix the notation used for the congruence tower. Let
\[
B\subset Y(\gotn)^N_{\C_\infty}
\]
be smooth, irreducible and locally closed. Let
\[
\cala'
\subset
(\pi_\gotn^N)^{-1}(B^{\mathrm{an}})
\]
be a connected analytic component, and set
$\Delta_{\cala'}:=
\Stab_{\Gamma(\gotn)^N}(\cala')$.

For every non-zero ideal $\gota\subset A$, set
\[
B(\gota)
:=
B\times_{Y(\gotn)^N}Y(\gotn\gota)^N.
\]
Let $B(\gota)_{\cala'}$ be the algebraic connected component of
$B(\gota)$ whose analytification contains the image of $\cala'$ under
\[
\pi_{\gotn\gota}^N\colon
\Omega^N\rightarrow Y(\gotn\gota)^{N,\mathrm{an}}.
\]

\subsubsection{Coordinate monodromy}

\begin{lem}\label{l:congruence-tower-comparison}

With the notations from Section~\ref{running}, we have
\[
\Stab_{(\Gamma(\gotn)/\Gamma(\gotn\gota))^N}\bigl(B(\gota)_{\cala'}\bigr)
=
\mathrm{Im}\bigl(\Delta_{\cala'}\rightarrow (\Gamma(\gotn)/\Gamma(\gotn\gota))^N\bigr).
\]
\end{lem}

\begin{proof}
By Lemma~\ref{l:level-tower-galois}, the map $B(\gota)\rightarrow B$ is a finite \'etale
$(\Gamma(\gotn)/\Gamma(\gotn\gota))^N$-cover. The connected component of $B(\gota)^{\mathrm{an}}$ containing
$\pi_{\gotn\gota}^N(\cala')$ is the analytification of a unique algebraic
connected component of $B(\gota)$ by 
\cite[Theorem~2.3.1]{conrad-irred}.

To simplify the notation, let us set $\pi_\gota:=\pi_{\gotn\gota}^N$. We first show that
\begin{equation}\label{e:analyticimage-downstairs}
\pi_\gota(\cala')=B(\gota)_{\cala'}^{\mathrm{an}}.
\end{equation}
Let $C$ be a connected analytic component of
$\pi_\gota^{-1}\bigl(B(\gota)^{\mathrm{an}}\bigr)$.
By admissible-local triviality of $\pi_\gota$, for every $x\in C$ there is a
connected admissible open neighbourhood
$U\subset B(\gota)^{\mathrm{an}}$ of $\pi_\gota(x)$ such that the connected component of
$\pi_\gota^{-1}(U)$ containing $x$ maps isomorphically onto $U$. This component
is contained in $C$, because it is connected and meets $C$. Hence
$U\subset \pi_\gota(C)$, and therefore $\pi_\gota(C)$ is admissible open in
$B(\gota)^{\mathrm{an}}$.  If $C_1,C_2$ are two such components and
$\pi_\gota(C_1)\cap\pi_\gota(C_2)\neq\emptyset$, choose $x_\nu\in C_\nu$ with
$\pi_\gota(x_1)=\pi_\gota(x_2)$. Since the fibers of $\pi_\gota$ are the
$\Gamma(\gotn\gota)^N$-orbits, there exists $\eta\in\Gamma(\gotn\gota)^N$ such
that $x_2=\eta x_1$. Hence $\eta C_1$ and $C_2$ are connected components which
meet, so $\eta C_1=C_2$. Since $\eta$ acts trivially after applying
$\pi_\gota$, we get
$\pi_\gota(C_1)=\pi_\gota(C_2)$.
Hence the images of the connected components are pairwise either equal or
disjoint. Admissible-locally on $B(\gota)^{\mathrm{an}}$ this is a disjoint
admissible decomposition. Therefore $\pi_\gota(\cala')$ is admissible open and
admissible closed in the connected space $B(\gota)_{\cala'}^{\mathrm{an}}$. It
is non-empty, hence
$\pi_\gota(\cala')=B(\gota)_{\cala'}^{\mathrm{an}}$.

The inclusion
\[
\im\bigl(\Delta_{\cala'}\rightarrow (\Gamma(\gotn)/\Gamma(\gotn\gota))^N\bigr)
\subseteq
\Stab_{(\Gamma(\gotn)/\Gamma(\gotn\gota))^N}\bigl(B(\gota)_{\cala'}\bigr)
\]
follows from \eqref{e:analyticimage-downstairs}: if
$\delta\in\Delta_{\cala'}$, then $\delta\cala'=\cala'$, hence the image of
$\delta$ stabilizes $\pi_\gota(\cala')=B(\gota)_{\cala'}^{\mathrm{an}}$, and therefore
stabilizes the algebraic component $B(\gota)_{\cala'}$. For the converse inclusion, let $\bar\gamma\in (\Gamma(\gotn)/\Gamma(\gotn\gota))^N$ stabilize
$B(\gota)_{\cala'}$, and choose a lift $\gamma\in\Gamma(\gotn)^N$. Pick
$x\in\cala'$. Since $\bar\gamma$ stabilizes $B(\gota)_{\cala'}$, the point
$\pi_\gota(\gamma x)=\bar\gamma\,\pi_\gota(x)$
belongs to $B(\gota)_{\cala'}^{\mathrm{an}}$. By
\eqref{e:analyticimage-downstairs},
there exists $x'\in\cala'$ such that
$\pi_\gota(x')=\pi_\gota(\gamma x)$.
Thus $x'=\eta\gamma x$ for some $\eta\in\Gamma(\gotn\gota)^N$. Hence
$\eta\gamma\cala'$ and $\cala'$ are connected components of
$(\pi_{\gotn}^N)^{-1}(B^{\mathrm{an}})$ sharing the point $x'$. Therefore
$\eta\gamma\cala'=\cala'$, so $\eta\gamma\in\Delta_{\cala'}$. Since $\eta$ maps trivially to $(\Gamma(\gotn)/\Gamma(\gotn\gota))^N$, the class $\bar\gamma$ belongs to
$\im(\Delta_{\cala'}\rightarrow(\Gamma(\gotn)/\Gamma(\gotn\gota))^N)$. 
\end{proof}

Lemma~\ref{l:congruence-tower-comparison} implies that the monodromy group of
$B(\gota)_{\cala'}\rightarrow B$
is $\mathrm{Im}\big(\Delta_{\cala'}\rightarrow(\Gamma(\gotn)/\Gamma(\gotn\gota))^N\big)$, up to conjugacy.
Indeed, for a finite \'etale Galois cover with group $Q$, the monodromy group
of a connected component is the stabilizer of that component in $Q$. Applying
this to $B(\gota)_{\cala'}$ gives this final assertion.

\begin{lem}\label{l:dominant-pi1-open}
Let $f\colon V\rightarrow W$
be a dominant morphism of connected noetherian schemes, with $f$ of finite
type and $W$ normal. Then the induced  homomorphism
$\pi_1^{\mathrm{et}}(V)\rightarrow\pi_1^{\mathrm{et}}(W)$
has open image.
\end{lem}
\begin{proof}
Since $W$ is connected, normal and noetherian, it is integral.
One of the finitely many irreducible components of $V$ dominates $W$.
Its reduced structure has fundamental group image contained in that of $V$, so it suffices to replace $V$ by this reduced component.
Put $K:=K(W)$ and $E:=K(V)$. Choose a separable closure $E^{\sep}$ of $E$,
and let $K^{\sep}\subset E^{\sep}$ be the separable closure of $K$ inside
$E^{\sep}$. Set $L:=K^{\sep}\cap E$.
Since $E/K$ is finitely generated, $L/K$ is finite.
As $W$ is normal, by \cite[Lemma~58.10.7]{stacks}  we obtain a surjective map
$\Gal(K^{\sep}/K)\twoheadrightarrow \pi_1^{\mathrm{et}}(W)$.
The image of
$\Gal(E^{\sep}/E)\rightarrow\Gal(K^{\sep}/K)$
is $\Gal(K^{\sep}/L)$, which has index $[L:K]$. Therefore the image of
$\pi_1^{\mathrm{et}}(V)\rightarrow\pi_1^{\mathrm{et}}(W)$
has finite index. Since it is a continuous image of a profinite group, it is
closed; hence it is open.
\end{proof}

We record the following lemma.

\begin{lem}
\label{l:adelic-open-zariski}
Let $\Delta\subset \Sl_2(F)$ be a subgroup whose closure in
$\Sl_2(\mathbf A_F^\infty)$ is open. Let
$\tilde\Delta\subset \pgl_2(F)$ be its image. Then
$\tilde\Delta$ is Zariski dense in $\pgl_{2,F}$.
\end{lem}

\begin{proof}
Let
$H:=\overline{\Delta}^{\,\mathrm{Zar}}\subseteq \Sl_{2,F}$. We first show that $H=\Sl_{2,F}$. Suppose otherwise. For any finite
place $v\ne\infty$, the subset $H(F_v)\subset\Sl_2(F_v)$ has empty
interior, because $H$ is a proper Zariski closed subvariety of
$\Sl_{2,F}$.

On the other hand, the adelic closure of $\Delta$ is open in
$\Sl_2(\mathbf A_F^\infty)$. Therefore its projection to
$\Sl_2(F_v)$ contains a non-empty open subset. This projected set is
contained in the $v$-adic closure of $\Delta$, and the latter is
contained in $H(F_v)$. This contradicts the empty interior statement.
Hence $H=\Sl_{2,F}$.

Let
$\rho\colon \Sl_{2,F}\rightarrow \pgl_{2,F}$
be the natural central isogeny. Let $Y\subset\pgl_{2,F}$ be the Zariski
closure of $\tilde\Delta$. Then $\rho^{-1}(Y)$ is a closed subvariety of
$\Sl_{2,F}$ containing $\Delta$. Since $\Delta$ is Zariski dense in
$\Sl_{2,F}$, we have $\rho^{-1}(Y)=\Sl_{2,F}$.
As $\rho$ is dominant, this forces $Y=\pgl_{2,F}$. Thus
$\overline{\tilde\Delta}^{\,\mathrm{Zar},F}
=\pgl_{2}$.
\end{proof}

\begin{lem}
\label{l:coordinate-adelic-open-corrected}
Let
$B\subset Y(\gotn)^N_{\C_\infty}$
be smooth, irreducible and locally closed, and let
$\cala'\subset(\pi_{\gotn}^N)^{-1}(B^{\mathrm{an}})$
be a connected analytic component. Set $\Delta_{\cala'}:=\Stab_{\Gamma(\gotn)^N}(\cala')$.
Fix $1\le i\le N$, and assume that $p_i\colon B\rightarrow Y(\gotn)$ is nonconstant. Then
\[
\overline{\pr_i(\Delta_{\cala'})}^{\,\mathrm{Zar}}
=
\pgl_{2,\C_\infty}.
\]
\end{lem}

\begin{proof}

Since $Y(\gotn)$ is an irreducible curve and $p_i$ is nonconstant, $p_i$ is
dominant. Consider $\pi_1^{\mathrm{et}}(Y(\gotn))$ and set $H_i:=
\im\bigl(\pi_1^{\mathrm{et}}(B)\rightarrow\pi_1^{\mathrm{et}}(Y(\gotn))\bigr)$.
By Lemma~\ref{l:dominant-pi1-open}, $H_i$ is open, so that $[\pi_1^{\mathrm{et}}(Y(\gotn)):H_i]=:d$ is finite.

Fix a non-zero ideal $\gota\subset A$, consider $\Gamma(\gotn)/\Gamma(\gotn\gota)$.
By Lemma~\ref{l:level-tower-galois}, the cover
$Y(\gotn\gota)\rightarrow Y(\gotn)$
is connected finite \'etale Galois with group $\Gamma(\gotn)/\Gamma(\gotn\gota)$. Hence its monodromy
homomorphism
\[
\rho_\gota\colon \pi_1^{\mathrm{et}}(Y(\gotn))\rightarrow  \Gamma(\gotn)/\Gamma(\gotn\gota)
\]
is surjective. Consider now
\[
\varrho_{\gota,B}\colon \pi_1^{\mathrm{et}}(B)\rightarrow (\Gamma(\gotn)/\Gamma(\gotn\gota))^N
\]
the monodromy homomorphism of the finite cover $B(\gota):=B\times_{Y(\gotn)^N}Y(\gotn\gota)^N\rightarrow B$
computed from a lift lying in the component $B(\gota)_{\cala'}$. 
For each $1\le r\le N$, the $r$-th coordinate of $B(\gota)\rightarrow B$ is the
pullback of
$Y(\gotn\gota)\rightarrow Y(\gotn)$
along $p_r\colon B\rightarrow Y(\gotn)$. Hence, by functoriality of \'etale fundamental
groups, the $r$-th coordinate of $\varrho_{\gota,B}$ is the pulled-back
monodromy of this cover. In particular, for the fixed coordinate $i$,
\[
\pr_i(\im\varrho_{\gota,B})=\rho_\gota(H_i)
\]
up to conjugacy in $\Gamma(\gotn)/\Gamma(\gotn\gota)$.

On the other hand, by Lemma~\ref{l:congruence-tower-comparison}, the monodromy
group of the component
$B(\gota)_{\cala'}\rightarrow B$
is conjugate in $(\Gamma(\gotn)/\Gamma(\gotn\gota))^N$ to
$\im\bigl(\Delta_{\cala'}\rightarrow (\Gamma(\gotn)/\Gamma(\gotn\gota))^N\bigr)$.
Projecting to the $i$-th coordinate, we get that
\[
R_\gota:=
\im\bigl(\pr_i(\Delta_{\cala'})\rightarrow  \Gamma(\gotn)/\Gamma(\gotn\gota)\bigr)
\]
is conjugate in $\Gamma(\gotn)/\Gamma(\gotn\gota)$ to $\rho_\gota(H_i)$. Hence
\[
[\Gamma(\gotn)/\Gamma(\gotn\gota):R_\gota]
=
[\Gamma(\gotn)/\Gamma(\gotn\gota):\rho_\gota(H_i)]
=
[\pi_1^{\mathrm{et}}(Y(\gotn)):H_i\ker(\rho_\gota)]
\le
[\pi_1^{\mathrm{et}}(Y(\gotn)):H_i]
=
d 
\]
where the second equality holds since $\rho_\gota$ is surjective.

Let us set
$\widehat \Gamma_\gotn:=\varprojlim_\gota \Gamma(\gotn)/\Gamma(\gotn\gota)$.
The finite-level images of $\pr_i(\Delta_{\cala'})$ are  the subgroups $R_\gota$.
Therefore $\overline{\pr_i(\Delta_{\cala'})}^{\,\widehat \Gamma_\gotn}
=\varprojlim_\gota R_\gota$.
The uniform bound $[\Gamma(\gotn)/\Gamma(\gotn\gota):R_\gota]\le d$ implies
\[
\big[\widehat \Gamma_\gotn:\overline{\pr_i(\Delta_{\cala'})}^{\,\widehat \Gamma_\gotn}\big]\le d.
\]

Thus
$\overline{\pr_i(\Delta_{\cala'})}^{\,\widehat \Gamma_\gotn}$ is a closed subgroup of finite index in $\widehat \Gamma_\gotn$, hence is open.

Recall that  $\kappa^{-1}\colon \Gamma(\gotn)/\Gamma(\gotn\gota)\simeq\widetilde\Gamma(\gotn)/\widetilde\Gamma(\gotn\gota)$ by Lemma~\ref{l:level-tower-galois}.
By strong approximation for $\Sl_2$, see, for instance,
\cite[Corollary~28.3.6]{voight-quaternion},
\[
\varprojlim_\gota
\widetilde\Gamma(\gotn)/\widetilde\Gamma(\gotn\gota)
\simeq
\ker\bigl(\Sl_2(\widehat A)\rightarrow \Sl_2(A/\gotn)\bigr)=:K(\gotn).
\]
Hence the closure of
$\kappa^{-1}\bigl(\pr_i(\Delta_{\cala'})\bigr)$
is open in $K(\gotn)$, and therefore open in
$\Sl_2(\mathbf A_F^\infty)$, since $K(\gotn)$ is compact open.

The final Zariski-density assertion follows from
Lemma~\ref{l:adelic-open-zariski}, applied to
$\Delta=\kappa^{-1}\bigl(\pr_i(\Delta_{\cala'})\bigr)
\subset \Sl_2(F)$, after base change from $F$ to $\C_\infty$.
\end{proof}

\begin{cor}
\label{c:adelic-open-finite-index}
In the setting of Lemma~\ref{l:coordinate-adelic-open-corrected}, let
$\Delta'\subset\Delta_{\cala'}$ be a finite-index subgroup. Then the closure of $\kappa^{-1}\bigl(\pr_i(\Delta')\bigr)$
in $\Sl_2(\mathbf A_F^\infty)$ is open. Therefore, for every  place
$v\ne\infty$, the $v$-adic closure of $\pr_i(\Delta')$ in
$\pgl_2(F_v)$ contains a non-empty open subset, and
\[
\overline{\pr_i(\Delta')}^{\,\mathrm{Zar},F_v}
=
\pgl_{2,F_v}.
\]
\end{cor}

\begin{proof}
Since $\Delta'$ has finite index in $\Delta_{\cala'}$, the subgroup
$\pr_i(\Delta')$ has finite index in $\pr_i(\Delta_{\cala'})$. Hence
$\kappa^{-1}(\pr_i(\Delta'))$ has finite index in
$\kappa^{-1}(\pr_i(\Delta_{\cala'}))$. A finite-index subgroup of a subgroup
with open adelic closure again has open adelic closure.
Projecting to the $v$-factor and using that
$\Sl_2(F_v)\rightarrow \pgl_2(F_v)$
is open onto an open image, the $v$-adic closure of $\pr_i(\Delta')$
contains a non-empty open subset of $\pgl_2(F_v)$.

The Zariski-density assertion follows from
Lemma~\ref{l:adelic-open-zariski}, applied to
$\Delta=\kappa^{-1}\bigl(\pr_i(\Delta')\bigr)\subset\Sl_2(F)$,
and then by base change from $F$ to $F_v$.
\end{proof}

\begin{lem}\label{l:spreading-out}
There exist a finitely generated (over $F$) subfield $k\subset\C_\infty$, a smooth
geometrically irreducible $k$-variety $B_0$, and a $k$-morphism $B_0\rightarrow Y(\gotn)^N_k$ whose base change to $\C_\infty$ identifies with the morphism
$B\rightarrow Y(\gotn)^N_{\C_\infty}$.
\end{lem}

\begin{proof}
As $B$ and the morphism
$B\rightarrow Y(\gotn)^N_{\C_\infty}$ are of finite presentation, they descend
to one such subfield $k$ by \cite[Lemma~32.10.1]{stacks}. Thus there are
a $k$-variety $B_0$ and a morphism
$B_0\rightarrow Y(\gotn)^N_k$
whose base changes to $\C_\infty$ are the given data. Smoothness descends under (faithfully flat) base change, and $B_0$ is geometrically irreducible because $B_0\otimes_k\C_\infty\simeq B$ is irreducible and $\C_\infty$ is algebraically closed.
\end{proof}

From now on, fix $k$ and $B_0$ as in
Lemma~\ref{l:spreading-out}, and fix a prime $\gotp\neq\infty$ not dividing $\gotn$. Set
\[
K:=k(B_0).
\]
We fix separable closures
$F^{\sep}\subset k^{\sep}\subset K^{\sep}$
inside an algebraic closure of $K$, and set
\[
G_K:=\Gal(K^{\sep}/K)
\qquad
\text{and}
\qquad
G_K^{\mathrm{geom}}:=\Gal(K^{\sep}/Kk^{\sep}).
\]
Since $B_0$ is geometrically irreducible, $Kk^{\sep}$ is a field; in other words,
$B_{0,k^{\sep}}:=B_0\otimes_k k^{\sep}$ is irreducible with function field $Kk^{\sep}$.

For every $e\ge1$, set
\[
B_{0,e}:=
B_{0,k^{\sep}}
\times_{Y(\gotn)^N_{k^{\sep}}}
Y(\gotn\gotp^e)^N_{k^{\sep}}.
\]
Since $k^{\sep}$ is separably closed, every connected component of
$B_{0,e}$ is geometrically connected. Hence base change from
$k^{\sep}$ to $\C_\infty$ induces a bijection on connected components.

Let $C_e\subset B_{0,e}$ be the connected component whose base change to
$\C_\infty$ is the component
$B(\gotp^e)_{\cala'}$.
The transition morphism
$C_{e+1}\rightarrow C_e$
is surjective. Indeed, it is finite \'etale, and its image is a non-empty
open and closed subset of the connected scheme $C_e$.

\begin{lem}\label{l:det-finite-geometric}
With notation as in Lemma~\ref{l:spreading-out}, let
$\rho_{i,\gotp}\colon G_K\rightarrow\gl_2(F_\gotp)$
be the $\gotp$-adic Tate representation of the $i$-th coordinate Drinfeld
module. Then $\det\rho_{i,\gotp}(G_K^{\mathrm{geom}})$ is finite.
\end{lem}

\begin{proof}
The determinant representation of $\rho_{i,\gotp}$ is the $\gotp$-adic Tate
representation of the rank-$1$ determinant Drinfeld module associated with the
$i$-th rank-$2$ Drinfeld module; see \cite[Proof of Theorem~1.8]{pinkmt}. Every rank-$1$ Drinfeld $A$-module (where $A=\F_q[T]$) over an algebraically closed
field is isomorphic to the Carlitz module after a finite separable extension (see \cite{hayes}).
After replacing $K$ by a finite extension, the determinant Drinfeld module is
therefore Carlitz. Its $\gotp$-power torsion is contained in
$F^{\sep}\subset k^{\sep}$. Since $G_K^{\mathrm{geom}}$ fixes $k^{\sep}$
pointwise, the determinant character is trivial on a finite index subgroup of $G_K^{\mathrm{geom}}$. Hence its image is finite.
\end{proof}

\subsubsection{Monodromy comparison}
Let $B,\cala',\Delta_{\cala'}$ be as in Section~\ref{running}.
Fix a  prime $\gotp\neq \infty$ not dividing $\gotn$. Let $k,B_0,K$ and
$G_K^{\mathrm{geom}}$ be as in Lemma~\ref{l:spreading-out}. Recall that
\[
\rho_{r,\gotp}\colon G_K\rightarrow\gl_2(F_\gotp)
\qquad
\text{and}
\qquad
\bar\rho_{r,\gotp}\colon G_K\rightarrow\pgl_2(F_\gotp)
\]
are the $\gotp$-adic and projective $\gotp$-adic Tate representations attached
to the $r$-th coordinate Drinfeld module. For $i\ne j$ recall that $\pr_{ij}$ is the projection to the coordinates $i$, $j$. Set
\[
S_{\gotp,ij}:=
\overline{\pr_{ij}(\Delta_{\cala'})}^{\,\mathrm{Zar}}
\subset
\pgl_{2,F_\gotp}^{\,2}.
\]

\begin{lem}
\label{l:tate-comparison-identity-components}

There exist $a_i,a_j\in\pgl_2(A_\gotp)\subset\pgl_2(F_\gotp)$ such that
\[
\big(\overline{(\bar\rho_{i,\gotp},\bar\rho_{j,\gotp})(G_K^{\mathrm{geom}})}
^{\,\mathrm{Zar}}\big)^{\circ}
=
(a_i,a_j)S_{\gotp,ij}^{\circ}(a_i,a_j)^{-1}.
\]
\end{lem}

\begin{proof}
Choose a geometric point $\bar b$ of $B_{0,k^{\sep}}$. For every $e\ge1$,
choose a lift $\bar b_e$ of $\bar b$ to the connected component of
$B_{0,k^{\sep}} \times_{Y(\gotn)^N_{k^{\sep}}}Y(\gotn\gotp^e)^N$
corresponding to $\cala'$, and choose these lifts compatibly as $e$ varies.

For the $r$-th coordinate, the point $\bar b_e$ gives a full
level-$\gotn\gotp^e$ structure on the fiber of the $r$-th Drinfeld module at
$\bar b$. Since the level-$\gotn$ structure is already fixed, the remaining
$\gotp^e$-part gives a basis of the $\gotp^e$-torsion. Passing to the inverse
limit over $e$ gives an $A_\gotp$-basis of
$T_\gotp(\varphi_r)$.
Changing the compatible system of lifts changes this basis, in the $r$-th
coordinate, by an element of $\gl_2(A_\gotp)$, hence changes the projective
image by conjugation by an element
\[
a_r\in\pgl_2(A_\gotp).
\]
These are the elements $a_r$ used below.

Since $B_0$ is geometrically irreducible, $B_{0,k^{\sep}}$ is connected with
function field $Kk^{\sep}$. Consider the geometric generic point
$\bar\eta=\Spec K^{\sep}\rightarrow B_{0,k^{\sep}}$.
Then the inclusion of the generic point
$\eta=\Spec Kk^{\sep}\rightarrow B_{0,k^{\sep}}$
induces, by functoriality of the \'etale fundamental group, a morphism
\[
\pi_1^{\et}(\eta,\bar\eta)
=
\Gal(K^{\sep}/Kk^{\sep})
=
G_K^{\mathrm{geom}}
\rightarrow
\pi_1^{\et}(B_{0,k^{\sep}},\bar\eta).
\]
Since $B_{0,k^{\sep}}$ is smooth, hence regular, hence normal, this map is
surjective; see \cite[Lemma~58.10.7]{stacks}.
Therefore, at level $\gotp^e$, the image of
$G_K^{\mathrm{geom}}$ is the full monodromy group of the corresponding connected
component of
$B_{0,k^{\sep}}
\times_{Y(\gotn)^N_{k^{\sep}}}Y(\gotn\gotp^e)^N$.
By Lemma~\ref{l:congruence-tower-comparison} and
Lemma~\ref{l:spreading-out}, this group is conjugate to the image
of $\Delta_{\cala'}$ in
$\bigl(\Gamma(\gotn)/\Gamma(\gotn\gotp^e)\bigr)^N$.
Projecting to the $(i,j)$-coordinates, let $M_{ij,e}$ denote the image of
$G_K^{\mathrm{geom}}$ on the corresponding $(i,j)$-coordinate
level-$\gotn\gotp^e$ cover. Then
\begin{equation}\label{e:finite-level-comparison-ij}
M_{ij,e}
\sim
\im\left(
\pr_{ij}(\Delta_{\cala'})
\rightarrow
\bigl(\Gamma(\gotn)/\Gamma(\gotn\gotp^e)\bigr)^2
\right),
\end{equation}
where $\sim$ denotes conjugacy in
$\bigl(\Gamma(\gotn)/\Gamma(\gotn\gotp^e)\bigr)^2$.

The level-$\gotn$ structure is defined over $K$, so $G_K$ acts trivially on the
level-$\gotn$ part. Thus only the $\gotp$-power tower is seen by the
representations below. Since $\gotp\nmid\gotn$, the Chinese remainder theorem
yields
\[
\ker\bigl(\Sl_2(A/\gotn\gotp^e)\rightarrow\Sl_2(A/\gotn)\bigr)
\simeq
\Sl_2(A/\gotp^e).
\]
The reduction map from $\widetilde\Gamma(\gotn)$ onto this kernel is
surjective by strong approximation, see again
\cite[Corollary~28.3.6]{voight-quaternion}. Hence
$\widetilde\Gamma(\gotn)/\widetilde\Gamma(\gotn\gotp^e)
\simeq
\Sl_2(A/\gotp^e)$.
Via Lemma~\ref{l:level-tower-galois}, it follows that
\begin{equation}\label{e:approx-monodromy}
\Gamma(\gotn)/\Gamma(\gotn\gotp^e)
\simeq
\Sl_2(A/\gotp^e).
\end{equation}
Consider
\[
G^0:=
\big\{\sigma\in G_K^{\mathrm{geom}}:
\det\rho_{i,\gotp}(\sigma)=1=\det\rho_{j,\gotp}(\sigma)
\big\}.
\]
By Lemma~\ref{l:det-finite-geometric}, $G^0$ has finite index in
$G_K^{\mathrm{geom}}$.
For $r=i,j$, let
$\rho^{\mathrm{int}}_{r,\gotp}\colon G_K\rightarrow \gl_2(A_\gotp)$
be the representation on the Tate module
$T_\gotp(\varphi_r)$ with respect to the basis chosen above, whose scalar
extension to $F_\gotp$ is $\rho_{r,\gotp}$. If $\sigma\in G^0$, then
$\det\rho^{\mathrm{int}}_{r,\gotp}(\sigma)
=
\det\rho_{r,\gotp}(\sigma)
=
1$.
Hence, after reducing modulo $\gotp^e$, we have
$\rho^{\mathrm{int}}_{r,\gotp}(\sigma)\bmod\gotp^e
\in
\Sl_2(A_\gotp/\gotp^eA_\gotp)
=
\Sl_2(A/\gotp^e)$.
Thus, under the chosen compatible bases and the identification
\eqref{e:approx-monodromy}, the action of $G^0$ on the $(i,j)$-coordinates of
the level-$\gotn\gotp^e$ cover is represented by the reductions modulo
$\gotp^e$ of
$(\rho^{\mathrm{int}}_{i,\gotp},\rho^{\mathrm{int}}_{j,\gotp})|_{G^0}$.
After projectivising, this is the reduction modulo $\gotp^e$ of
$(\bar\rho_{i,\gotp},\bar\rho_{j,\gotp})|_{G^0}$,
up to the fixed conjugation by $(a_i,a_j)$.
By \eqref{e:finite-level-comparison-ij}, the image of $G_K^{\mathrm{geom}}$
on the same $(i,j)$-coordinate level cover is conjugate to the image of
\[
\pr_{ij}(\Delta_{\cala'})
\rightarrow
\big(\Gamma(\gotn)/\Gamma(\gotn\gotp^e)\big)^2.
\]
Let $\theta_e$ be the finite-level monodromy map of
$G_K^{\mathrm{geom}}$ on the $(i,j)$-coordinate level-$\gotn\gotp^e$ cover, in symbols, $\theta_e\colon G_K^{\mathrm{geom}}\rightarrow (\Gamma(\gotn)/\Gamma(\gotn\gotp^e))^2$.
Since $G^0\subset G_K^{\mathrm{geom}}$ has finite index,
$[\theta_e(G_K^{\mathrm{geom}}):\theta_e(G^0)]
\le
[G_K^{\mathrm{geom}}:G^0]$.
By \eqref{e:finite-level-comparison-ij}, after the fixed conjugation by
$(a_i,a_j)$, the group $\theta_e(G_K^{\mathrm{geom}})$ is the corresponding
finite-level image of $\pr_{ij}(\Delta_{\cala'})$. Hence the finite-level image
of $G^0$ has index at most $[G_K^{\mathrm{geom}}:G^0]$ inside that image.

Passing to the inverse limit over $e$, the $\gotp$-adic closure of
$(\bar\rho_{i,\gotp},\bar\rho_{j,\gotp})(G^0)$
in $\pgl_2(A_\gotp)^2$ is therefore commensurable with the $\gotp$-adic closure of
$(a_i,a_j)\pr_{ij}(\Delta_{\cala'})(a_i,a_j)^{-1}$.
Commensurable compact subgroups have the same Zariski closure identity
component. Since $G^0$ has finite index in $G_K^{\mathrm{geom}}$, replacing
$G^0$ by $G_K^{\mathrm{geom}}$ also does not change the Zariski closure identity component. 
Hence we conclude that
$\big(
\overline{(\bar\rho_{i,\gotp},\bar\rho_{j,\gotp})(G_K^{\mathrm{geom}})}
^{\,\mathrm{Zar}}
\big)^\circ=
(a_i,a_j)S_{\gotp,ij}^{\circ}(a_i,a_j)^{-1}$.
\end{proof}

\begin{lem}\label{l:normalizer-graph}
Let $u\in\pgl_2(F_\gotp)$ and consider
$D_u:=\{(g,ugu^{-1}):g\in\pgl_{2,F_\gotp}\}
\subset\pgl_{2,F_\gotp}^{\,2}$.
Then $N_{\pgl_2^2}(D_u)=D_u$.
\end{lem}

\begin{proof}
Let $(a,b)\in N_{\pgl_2^2}(D_u)$. Then we have that for every $g\in\pgl_2$ there exists
$g'\in\pgl_2$ such that
$(a,b)(g,ugu^{-1})(a,b)^{-1}=(g',ug'u^{-1})$.
The first coordinate gives $g'=aga^{-1}$, and the second coordinate gives
$bugu^{-1}b^{-1}=uaga^{-1}u^{-1}$ for all $g$.
Put $c:=u^{-1}bu$. Then $cgc^{-1}=aga^{-1}$ for all $g$, so $a^{-1}c$
centralizes $\pgl_2$. Since $\pgl_2$ has trivial center, $c=a$. Thus
$b=uau^{-1}$, and $(a,b)\in D_u$.
\end{proof}

\begin{cor}
\label{l:graph-monodromy-projective-tate}
Keep the notation of
Lemma~\ref{l:tate-comparison-identity-components}. Assume that
$S_{\gotp,ij}^{\circ}
=
D_{\delta_0}
:=
\left\{
(g,\delta_0g\delta_0^{-1})
:
g\in\pgl_{2,F_\gotp}
\right\}$
for some $\delta_0\in\pgl_2(F_\gotp)$. Then there exists
$u\in\pgl_2(F_\gotp)$ such that
\[
\bar\rho_{j,\gotp}(\sigma)
=
u\bar\rho_{i,\gotp}(\sigma)u^{-1}
\qquad
\text{for every }\sigma\in G_K.
\]
\end{cor}

\begin{proof}
Let $H:=
\overline{
(\bar\rho_{i,\gotp},\bar\rho_{j,\gotp})
(G_K^{\mathrm{geom}})
}^{\,\mathrm{Zar},F_\gotp}
\subset
\pgl_{2,F_\gotp}^{\,2}$.
By Lemma~\ref{l:tate-comparison-identity-components}, we have
$H^\circ
=
(a_i,a_j)S_{\gotp,ij}^{\circ}(a_i,a_j)^{-1}
=D_u$,
where
$u:=a_j\delta_0a_i^{-1}\in\pgl_2(F_\gotp)$.
Since $G_K^{\mathrm{geom}}$ is normal in $G_K$, the group
$(\bar\rho_{i,\gotp},\bar\rho_{j,\gotp})(G_K)$
normalizes $H^\circ=D_u$. By
Lemma~\ref{l:normalizer-graph},
we have $N_{\pgl_2^2}(D_u)=D_u$.
Thus
$(\bar\rho_{i,\gotp}(\sigma),
 \bar\rho_{j,\gotp}(\sigma))
\in D_u$
for every $\sigma\in G_K$, which is the claimed conjugacy.
\end{proof}

\begin{lem}
\label{l:projective-tate-conjugacy-isogeny-corrected}
Let $K$ be finitely generated over $F$, and let $\varphi_1,\varphi_2$ be
rank-$2$ Drinfeld $A$-modules over $K$. Let
$\gotp\ne\infty$ be a prime. If, after replacing $K$ by a finite
extension, the projective $\gotp$-adic Tate representations of $\varphi_1$ and
$\varphi_2$ are conjugate, then $\varphi_1$ and $\varphi_2$ are geometrically
isogenous.
\end{lem}

\begin{proof}
After the finite extension in the hypothesis, assume
$\bar\rho_{\varphi_2,\gotp}=\Inn(u)\circ\bar\rho_{\varphi_1,\gotp}$
for some $u\in\pgl_2(F_\gotp)$. Choose a lift $\widetilde u\in\gl_2(F_\gotp)$.
Writing $\rho_a:=\rho_{\varphi_a,\gotp}$, there is a continuous character
$\chi\colon G_K\rightarrow F_\gotp^\times$ such that
$\rho_2(\sigma)=\chi(\sigma)\widetilde u\rho_1(\sigma)\widetilde u^{-1}$
for $\sigma\in G_K$.
Taking determinants gives $\det\rho_2=\chi^2\det\rho_1$.
Let $\psi_a$ be the rank-$1$ determinant Drinfeld module attached to $\varphi_a$ (see again \cite[proof of Theorem~1.8]{pinkmt}), so that
\[
V_\gotp(\psi_a)\simeq\bigwedge^2 V_\gotp(\varphi_a)
\]
$G_K$-equivariantly. Since $A=\F_q[T]$, the
two rank-$1$ modules $\psi_1$ and $\psi_2$ are geometrically isogenous. After a
further finite extension of $K$, their rational Tate modules are isomorphic, so
$\det\rho_1=\det\rho_2$. Therefore $\chi^2=1$. Since $q$ is odd, $\chi$ has finite image. After one more finite extension, $\chi=1$.

Thus $V_\gotp(\varphi_1)\simeq V_\gotp(\varphi_2)$ as $F_\gotp[G_K]$-modules.
The Tate conjecture for Drinfeld modules over finitely generated fields \cite{taguchi,tamagawa} gives
\[
\Hom_K(\varphi_1,\varphi_2)\otimes_A F_\gotp
\simeq
\Hom_{F_\gotp[G_K]}\bigl(V_\gotp(\varphi_1),V_\gotp(\varphi_2)\bigr).
\]
 The
right-hand side is non-zero, hence $\Hom_K(\varphi_1,\varphi_2)\ne0$ after the
finite extension already made. In general, if a non-zero homomorphism is obtained
after some finite extension $K'/K$, then it is a non-zero element of
$\Hom_{\overline K}(\varphi_1,\varphi_2)$,
where $\overline K$ is an algebraic closure of $K$. Thus the two modules are
geometrically isogenous. Any non-zero homomorphism between Drinfeld modules of the
same rank is an isogeny.
\end{proof}

\subsubsection{The two-coordinate criterion}

For a field extension $K/F$, an element $\delta\in\pgl_2(K)$, and an
integer $m\ge0$, set
\[
\Gamma_{\delta,m,K}
:=
\left\{
\bigl(g,\delta\,\Frob^m(g)\,\delta^{-1}\bigr):
g\in\pgl_{2,K}
\right\}
\subset \pgl_{2,K}^{\,2}.
\]
Here $\Frob^m$ denotes the $m$-fold Frobenius endomorphism of  $\pgl_2$. On $K$-points it raises matrix coefficients to their
$p^m$-th powers.

\begin{lem}
\label{l:F-frobenius-graph}
Let
$J_F\subset \pgl_{2,F}^{\,2}$
be a closed subgroup. Assume that, after base change to $\C_\infty$,
we get $(J_F)_{\C_\infty}=
\Gamma_{\delta,m,\C_\infty}$
for some $\delta\in\pgl_2(\C_\infty)$ and $m\ge0$. Then there exists
$\delta_0\in\pgl_2(F)$ such that
\[
J_F=\Gamma_{\delta_0,m,F}.
\]
\end{lem}

\begin{proof}
Set $G_F:=\pgl_{2,F}$. Since $(J_F)_{\C_\infty}$ is a graph over
the first factor, the first projection
$\pr_1\colon J_F\rightarrow G_F$
becomes an isomorphism after base change to $\C_\infty$. Hence it is
already an isomorphism over $F$. Thus $J_F$ is the graph of the $F$-homomorphism
\[
\alpha:=\pr_2\circ\pr_1^{-1}\colon G_F\rightarrow G_F.
\]
After base change to $\C_\infty$, one has
$\alpha_{\C_\infty}=
\Inn(\delta)\circ\Frob^m$,
so $\alpha$ is an isogeny. By
\cite[Proposition~1.6]{pinkcompact}, the group $G_F$ admits no non-standard isogenies. Hence
\cite[Theorem~1.7(b)]{pinkcompact} gives an integer $r\ge0$ and an
$F$-isomorphism $\beta\colon(\uptau^r)^*G_F\xrightarrow{\sim}G_F$ such that
\[
\alpha=\beta\circ\Frob^r.
\]
Since $G_F$ is the base change of $\pgl_{2,\F_p}$, we canonically
identify $(\uptau
^r)^*G_F$ with $G_F$, and regard $\beta$ as an
$F$-automorphism of $G_F$.

We claim that $r=m$. Since $\dim\pgl_2=3$, the morphism
$\Frob^s\colon\pgl_{2,\C_\infty}\rightarrow
\pgl_{2,\C_\infty}$
has purely inseparable degree $p^{3s}$, while an automorphism has
inseparable degree $1$. Therefore
\[
p^{3r}
=
\deg_{\mathrm{insep}}(\alpha_{\C_\infty})
=
p^{3m},
\]
and hence $r=m$.

Finally, every $F$-automorphism of $G_F$ is inner. Indeed,
\cite[Proposition~7.1.6]{conrad} describes $\Aut_F(G_F)$ as the
semidirect product of $G_F(F)$ with the automorphism group of its based
root datum; the latter is trivial for type $A_1$. Thus $\beta=\Inn(\delta_0)$
for some $\delta_0\in\pgl_2(F)$. Thus
$\alpha=\Inn(\delta_0)\circ\Frob^m$.
Since $J_F$ is the graph of $\alpha$, this gives $J_F=\Gamma_{\delta_0,m,F}$.
\end{proof}

\begin{lem}\label{l:aut}
  The  automorphism group of $\pgl_2(\C_\infty)$ is naturally isomorphic to $\pgl_2(\C_\infty) \rtimes \Aut(\C_\infty)$.
\end{lem}

\begin{proof}
See \cite{svdw}.
\end{proof}

\begin{lem}
\label{l:step3-frobenius-relation}
Let $S\subset\pgl_2(F)^2$ be a subgroup. Assume
$\overline S^{\,\mathrm{Zar},\C_\infty}
\subsetneq
\pgl_{2,\C_\infty}^{\,2}$
and that both coordinate projections of this Zariski closure are $\pgl_{2,\C_\infty}$. Then, after possibly exchanging the two coordinates, there exist a finite index subgroup $S^0\subset S$, an element $\delta_0\in\pgl_2(F)$, and an integer $m\ge0$ such that
\[
s_2=\delta_0\Frob^m(s_1)\delta_0^{-1}
\]
for every $s=(s_1,s_2)\in S^0$.
\end{lem}

\begin{proof}
Set $G:=\pgl_{2,\C_\infty}$ and
$H:=\overline S^{\,\mathrm{Zar},\C_\infty}\subset G^2$.
Let $H^\circ$ be the identity component of $H$, and consider
$S^0:=S\cap H^\circ(\C_\infty)$.
Since $H^\circ$ has finite index in $H$, the subgroup $S^0$ has finite
index in $S$ and is Zariski dense in $H^\circ$. The two projections of
$H^\circ$ are still equal to $G$: their images have finite index in the
corresponding projections of $H$, and $G$ is connected.
Consider the (abstract) subgroup
$H^\circ(\C_\infty)
\subset
G(\C_\infty)\times G(\C_\infty)$.
We set
\[
N_1:=\bigl\{g\in G(\C_\infty):(g,e)\in H^\circ(\C_\infty)\bigr\}
\qquad
\text{and}
\qquad
N_2:=\bigl\{g\in G(\C_\infty):(e,g)\in H^\circ(\C_\infty)\bigr\}.
\]
The surjectivity of the two projections shows that $N_1$ and $N_2$ are
normal subgroups of $G(\C_\infty)$. Since $\C_\infty$ is algebraically
closed, $G(\C_\infty)=\PSL_2(\C_\infty)$, which is a simple
group. Neither $N_1$ nor $N_2$ can be all of $G(\C_\infty)$, for then the
surjectivity of the other projection would give $H^\circ=G^2$, contrary
to the hypothesis. Hence $N_1=N_2=\{e\}$. Goursat's lemma implies that
\[
H^\circ(\C_\infty)
=
\bigl\{(g,\psi(g)):g\in G(\C_\infty)\bigr\}
\]
for $\psi\in\Aut\bigl(G(\C_\infty)\bigr)$.

By Lemma~\ref{l:aut} there exist
$\delta\in\pgl_2(\C_\infty)$ and $\sigma\in\Aut(\C_\infty)$
such that
\[
\psi=\Inn(\delta)\circ\sigma
\]
where $\sigma$ acts on matrix coefficients. Set $H':=(\id_G\times\Inn(\delta^{-1}))(H^\circ)$.
Then
\[
H'(\C_\infty)
=
\bigl\{(g,\sigma(g)):g\in G(\C_\infty)\bigr\}.
\]
Let $U_+\subset G$ be the standard upper unipotent subgroup and write
$u_+(x):=
\begin{bmatrix}
1&x\\0&1
\end{bmatrix}$.
Thus $C:=H'\cap(U_+\times U_+)$
is an algebraic subgroup of $U_+\times U_+$, and
$C(\C_\infty)=\bigl\{\bigl(u_+(x),u_+(\sigma(x))\bigr):x\in\C_\infty\bigr\}$.
Its first projection is surjective. Since $C$ has finitely many connected
components, the image of the identity component $C^\circ$ is
positive dimensional and therefore equal to $U_+$. Since $C$ is a proper
subgroup of the two-dimensional group $U_+\times U_+$, it follows that
$\dim C^\circ=1$. Moreover, the graph above has exactly one point over
each $x\in\C_\infty$; hence the surjectivity of
$C^\circ\rightarrow U_+$ shows that $C^\circ(\C_\infty)$ is the whole
graph. Thus, since $U_+\simeq\G_a$, its germ at the identity is 
the germ at the origin of the graph of $\sigma$.
Lemma~\ref{l:analytic-field-automorphism} therefore gives
\[
\sigma=\Frob^r
\]
for some $r\in\Z$.

If $r\ge0$, put $m:=r$. If $r<0$, exchange the two factors, set
$m:=-r$, and replace $\delta$ by $\Frob^m(\delta)^{-1}$. In either case,
after possibly exchanging the two factors, the $\C_\infty$-points of
$H^\circ$ and $\Gamma_{\delta,m,\C_\infty}$ are equal. Both are reduced
closed algebraic subgroups of $G^2$, and therefore
\[
H^\circ
=
\Gamma_{\delta,m,\C_\infty}
\]
for some $\delta\in\pgl_2(\C_\infty)$ and $m\ge0$. Now set
\[
J_F:=
\overline{S^0}^{\,\mathrm{Zar},F}
\subset \pgl_{2,F}^{\,2}.
\]
Because $S^0\subset\pgl_2(F)^2$, base change gives
\[
(J_F)_{\C_\infty}
=
\overline{S^0}^{\,\mathrm{Zar},\C_\infty}
=
H^\circ
=
\Gamma_{\delta,m,\C_\infty}.
\]
Indeed, if a $\C_\infty$-polynomial vanishes on $S^0$, write its
coefficients in a finite dimensional $F$-subspace of $\C_\infty$ and
choose an $F$-basis; its $F$-polynomial components then vanish on $S^0$.
By Lemma~\ref{l:F-frobenius-graph}, there exists
$\delta_0\in\pgl_2(F)$ such that
$J_F=\Gamma_{\delta_0,m,F}$.
Since $S^0\subset J_F(F)$, every $s=(s_1,s_2)\in S^0$ satisfies
$s_2=\delta_0\Frob^m(s_1)\delta_0^{-1}$.
\end{proof}

\begin{lem}
\label{l:no-positive-frobenius-twist}
Let $i\ne j$, assume that $\pr_j\colon B\rightarrow Y(\gotn)$ is
nonconstant, and let $\Delta'\subset\Delta_{\cala'}$ be a finite-index
subgroup. Suppose that there exist $\delta\in\pgl_2(F)$ and $m\ge0$ such
that
\[
\pr_j(\gamma)=
\delta\,\Frob^m\big(\pr_i(\gamma)\big)\,\delta^{-1}
\qquad
\text{for every }\gamma\in\Delta'.
\]
Then $m=0$.
\end{lem}

\begin{proof}
Suppose that $m>0$, and set $a:=p^m>1$. Choose a prime
$\gotp\neq\infty$ not dividing $\gotn$ such that
$\delta\in\pgl_2(A_\gotp)$; this is possible since
$\delta\in\pgl_2(F)$ and so a rational point extends to an $A_\gotp$-point away from finitely many primes. Set
\[
K:=\pgl_2(A_\gotp)
\qquad
\text{and}
\qquad
K_e:=\ker\!\left(
K\rightarrow
\pgl_2(A_\gotp/\gotp^eA_\gotp)
\right)
\quad
\text{for}
\;\;
e\ge1.
\]
Since $A$ is a principal ideal domain, for $\gotp=(f)$, consider $|f|_\infty=|A/\gotp|=q^{\deg_T(f)}$. Since
$\Delta'\subset\Gamma(\gotn)^N$, all coordinate projections of
$\Delta'$ lie in $\pgl_2(A)\subset K$. For $\ell\in\{i,j\}$ and
$e\ge1$, let
\[
M_{\ell,e}:=
\operatorname{Im}\!\left(
\pr_\ell(\Delta')\rightarrow K/K_e
\right).
\]

By Corollary~\ref{c:adelic-open-finite-index}, the closure $H_j$ of
$\pr_j(\Delta')$ in $\pgl_2(F_\gotp)$ contains  a non-empty open subset $U$ of $K$. Choose $h\in U$. Since
$H_j$ is a subgroup, the translate
$h^{-1}U$ is an open neighborhood of the identity contained in $H_j$. It follows
that $H_j$ is open in $K$.
Hence there is an integer $s\ge1$ such that $K_s\subset H_j$. Reduction modulo $K_e$ has
finite target, so the image of $H_j$ is equal to the image of
$\pr_j(\Delta')$. Therefore, for every $e\ge s$,
\begin{equation}\label{e:frobenius-count-lower}
|M_{j,e}|
\ge
|K_s/K_e|
=
|f|_\infty^{3(e-s)}.
\end{equation}
For every $r\ge1$ one has
\begin{equation}\label{e:frobenius-depth}
\Frob^m(K_r)\subset K_{ar}.
\end{equation}
Indeed, let $I$ be the ideal of the identity in the coordinate ring of
$\pgl_{2,\F_p}$. If $g\in K_r$, then
$f(g)\in\gotp^rA_\gotp$ for every $f\in I$. The pullback by the $m$-fold
 Frobenius sends $f$ to $f^a$, and therefore
$f(\Frob^m(g))=f(g)^a\in\gotp^{ar}A_\gotp$.
This proves \eqref{e:frobenius-depth}.

Fix $e\ge1$ and consider
$r_e:=\left\lceil\frac{e}{a}\right\rceil$.
If $\gamma_1,\gamma_2\in\Delta'$ have the same $i$-th coordinate modulo
$K_{r_e}$, then
$\pr_i(\gamma_1)\pr_i(\gamma_2)^{-1}\in K_{r_e}$. Using the graph
relation, \eqref{e:frobenius-depth}, and the fact that conjugation by
$\delta\in K$ preserves $K_e$, we obtain
\[
\pr_j(\gamma_1)\pr_j(\gamma_2)^{-1}
=
\delta\,\Frob^m\!\left(
\pr_i(\gamma_1)\pr_i(\gamma_2)^{-1}
\right)\delta^{-1}
\in K_e.
\]
Thus the reduction of the $j$-th coordinate modulo $\gotp^e$ is determined
by the reduction of the $i$-th coordinate modulo $\gotp^{r_e}$. Hence
\begin{equation}\label{e:frobenius-count-upper}
|M_{j,e}|
\le
|M_{i,r_e}|
\le
|K/K_{r_e}|.
\end{equation}
Note that, for every $r\ge1$,
\[
|K/K_r|
= \left|\pgl_2(A_\gotp/\gotp^rA_\gotp)\right|
=
\left|\pgl_2(A/\gotp)\right|\cdot |f|_\infty^{3(r-1)}.
\]
As $r_e-1\le e/a$, equation~\eqref{e:frobenius-count-upper} gives
\[
|M_{j,e}|
\le
\left|\pgl_2(A/\gotp)\right|\cdot|f|_\infty^{3e/a}.
\]
Together with \eqref{e:frobenius-count-lower}, this yields, for every
$e\ge s$,
\[
|f|_\infty^{3(e-s)}
\le
\left|\pgl_2(A/\gotp)\right|\cdot|f|_\infty^{3e/a},
\]
which is impossible as $e\to\infty$, since $a>1$. Therefore $m=0$.
\end{proof}

\begin{lem}
\label{l:untwisted-graph-monodromy}
Fix $i\ne j$, and assume that $p_i,p_j\colon B\rightarrow Y(\gotn)$ are
nonconstant. Suppose that, for some finite-index subgroup
$\Delta'\subset\Delta_{\cala'}$ and some $\delta_0\in\pgl_2(F)$,
one has
$\pr_j(\gamma)=\delta_0\pr_i(\gamma)\delta_0^{-1}$
for every $\gamma\in\Delta'$. Then, for every finite prime
$\gotp\nmid\gotn$,
$S_{\gotp,ij}^{\circ}=D_{\delta_0}$.
\end{lem}

\begin{proof}
Let
$J:=\overline{\pr_{ij}(\Delta')}^{\,\mathrm{Zar},F_\gotp}$.
The assumed relation gives $J\subset D_{\delta_0}$.
By Corollary~\ref{c:adelic-open-finite-index},
$\pr_i(\Delta')$ is Zariski dense in $\pgl_{2,F_\gotp}$.
Hence the first projection of $J$ is surjective. Since the first
projection $D_{\delta_0}\rightarrow\pgl_{2,F_\gotp}$ is an isomorphism, we get
$J=D_{\delta_0}$. Since $\Delta'$ has finite index in $\Delta_{\cala'}$, the groups
$J$ and $S_{\gotp,ij}$ have the same identity component. Thus
$S_{\gotp,ij}^{\circ}=D_{\delta_0}$.
\end{proof}

\begin{lem}
\label{p:projected-pair-isogeny}
Fix $i\ne j$, and assume that $p_i,p_j\colon B\rightarrow Y(\gotn)$ are
nonconstant. Let $\varphi_i,\varphi_j$ be the corresponding generic
rank-$2$ Drinfeld modules. If
$\overline{\pr_{ij}(\Delta_{\cala'})}^{\,\mathrm{Zar},\C_\infty}$
is a proper subgroup of $\pgl_{2,\C_\infty}^{\,2}$, then
$\Hom_{\bar\eta_B}(\varphi_i,\varphi_j)\ne0$.
\end{lem}

\begin{proof}
By Lemma~\ref{l:coordinate-adelic-open-corrected}, both coordinate
projections of the Zariski closure are equal to
$\pgl_{2,\C_\infty}$.

Lemma~\ref{l:step3-frobenius-relation} therefore gives, after possibly
exchanging $i$ and $j$, a finite-index subgroup
$\Delta'\subset\Delta_{\cala'}$, an element
$\delta_0\in\pgl_2(F)$, and an integer $m\ge0$ such that
\[
\pr_j(\gamma)=
\delta_0\Frob^m(\pr_i(\gamma))\delta_0^{-1}
\qquad
\text{for}
\quad
\gamma\in\Delta'.
\]
Since $p_j$ is nonconstant,
Lemma~\ref{l:no-positive-frobenius-twist}, applied to this graph relation,
gives $m=0$.

Choose a finite prime $\gotp\nmid\gotn$ and spread out as in
Lemma~\ref{l:spreading-out}. By
Lemma~\ref{l:untwisted-graph-monodromy},
$S_{\gotp,ij}^{\circ}=D_{\delta_0}$. Hence
Corollary~\ref{l:graph-monodromy-projective-tate} shows that the two
projective $\gotp$-adic Tate representations are conjugate.
Lemma~\ref{l:projective-tate-conjugacy-isogeny-corrected} now gives a
geometric isogeny between $\varphi_i$ and $\varphi_j$. 
\end{proof}

\subsubsection{Proof of Theorem~\ref{t:HS-weakly}}\label{ss:proof-HS-weakly}
\begin{proof}
Let $X\subset Y(1)^N_{\C_\infty}$ be proper HS-special. If some coordinate of
$X$ is constant, then $X$ is contained in the weakly-special divisor
$x_i=a$, and there is nothing to prove. We therefore assume that all coordinate
projections of $X$ are nonconstant.

Fix the auxiliary level $\gotn$ chosen above, and recall
the level-forgetting map $\nu_\gotn\colon Y(\gotn)^N\rightarrow Y(1)^N$. Choose an irreducible component
$Z$ of $\nu_\gotn^{-1}(X)$ and replace it by a non-empty smooth open subset
$B\subset Z$
which dominates $X$ and lies over the \'etale locus of $\nu_\gotn$ along $X$.
Such a choice is possible because $\nu_\gotn$ is generically \'etale over $X$:
the ramification locus is supported over the elliptic divisors, and no
coordinate projection of $X$ is constant.

Let
$\cala'\subset(\pi_\gotn^N)^{-1}(B^{\mathrm{an}})$
be a connected analytic component. By Lemma~\ref{l:smooth-branch-bridge},
$\cala'$ is irreducible and lies in a unique irreducible analytic component
\[
\cala\subset \boldsymbol j^{-1}(X^{\mathrm{an}}),
\]
with
$\Delta_{\cala'}:=\Stab_{\Gamma(\gotn)^N}(\cala')
\subset
\Delta_\cala^+:=\Stab_{(\Gamma^+)^N}(\cala)$.

Since $X$ is HS-non-generic, Lemma~\ref{l:HS-monodromy-proper} gives
\[
H_\cala:=
\overline{\Delta_\cala^+}^{\,\mathrm{Zar},\C_\infty}
\subsetneq
\pgl_{2,\C_\infty}^{\,N}.
\]
On the other hand, every coordinate map $\pr_r\colon B\rightarrow Y(\gotn)$ is nonconstant.
Hence Lemma~\ref{l:coordinate-adelic-open-corrected} gives
\[
\overline{\pr_r(\Delta_{\cala'})}^{\,\mathrm{Zar}}
=
\pgl_{2,\C_\infty}
\]
for $1\le r\le N$.
Since $\Delta_{\cala'}\subset\Delta_\cala^+\subset H_\cala(\C_\infty)$, we have
$\pr_r(\Delta_{\cala'})\subset \pr_r(H_\cala)$.
The restriction $\pr_r|_{H_\cala}$ is a morphism of algebraic groups, hence
$\pr_r(H_\cala)$ is a closed algebraic subgroup of $\pgl_{2,\C_\infty}$.
Together with
$\overline{\pr_r(\Delta_{\cala'})}^{\,\mathrm{Zar}}=
\pgl_{2,\C_\infty}$,
this gives
$\pr_r(H_\cala)=\pgl_{2,\C_\infty}$.
Thus $N\ge2$. Suppose, by contradiction, that every two-coordinate projection of $H_\cala$ is $\pgl_{2,\C_\infty}^2$. We show by induction on $N$ that then $H_\cala=\pgl_{2,\C_\infty}^N$. The case
$N=2$ is immediate. For $N>2$, the induction hypothesis gives
$\pr_{1,\ldots,N-1}(H_\cala)=\pgl_{2,\C_\infty}^{N-1}$.
Apply Goursat's lemma to
\[
H_\cala\subset \pgl_{2,\C_\infty}^{N-1}\times \pgl_{2,\C_\infty}.
\]
If $H_\cala$ were proper, Goursat would give a nontrivial common quotient of $\pgl_{2,\C_\infty}^{N-1}$ and $G$. Since $\pgl_{2,\C_\infty}$ is simple, this quotient is $\pgl_{2,\C_\infty}$, and the quotient map $\pgl_{2,\C_\infty}^{N-1}\rightarrow \pgl_{2,\C_\infty}$ factors through one coordinate: the images of the factors are normal subgroups of $\pgl_{2,\C_\infty}$ that commute with one another. The corresponding two-coordinate projection of $H_\cala$ would then be proper, which is a contradiction. Hence, for some $i\ne j$,
\[
\pr_{ij}(H_\cala)\subsetneq \pgl_{2,\C_\infty}^2.
\]
Therefore
\[
\overline{\pr_{ij}(\Delta_{\cala'})}^{\,\mathrm{Zar}}
\subset
\pr_{ij}(H_\cala)
\subsetneq
\pgl_{2,\C_\infty}^{\,2}.
\]
Lemma~\ref{p:projected-pair-isogeny} now applies to $B$, $\cala'$, and the
coordinates $i,j$. Hence the two generic coordinate Drinfeld modules over $B$
are geometrically isogenous. Since $B\rightarrow X$ is dominant and generically finite,
the same is true for the corresponding generic coordinate Drinfeld modules
over $X$. Lemma~\ref{l:isogeny-hecke} therefore shows that $X$ is contained in a proper weakly-special subvariety.
\end{proof}

\subsubsection{From Ax--Schanuel to Ax--Lindemann, and Andr\'e--Oort}\label{ss:alao}

We now deduce a $n$-dimensional Ax--Lindemann for the Drinfeld $j$-function, following the argument of \cite[Corollary~1.4]{tsim}.

\begin{cor}\label{c:ax-lindemann}
Let $q$ be odd.  Let $C$ be an algebraic subvariety of $\A^n_{\C_\infty}$, and let $Z$ be a maximal irreducible algebraic subvariety contained in $\boldsymbol j^{-1}(C)$. Then $\boldsymbol j(Z)$ is a weakly-special subvariety.
\end{cor}
\begin{proof}
We set $X:=\overline{\boldsymbol j(Z)}^{\mathrm{Zar}}$, and we choose
a weakly-special $S\supset X$ of minimal dimension $r$. The case $r=0$ is immediate, so we shall deal with $r>0$. Choosing one coordinate from each
equivalence class of nonconstant coordinates linked by Hecke relations gives a finite surjective projection
$\pr_I\colon S\rightarrow Y(1)^r$. On a component of $\boldsymbol j^{-1}(S)$ containing $Z$, the
discarded coordinates are constants or fractional linear functions of the retained ones. Hence $\pr_I(Z)$ is algebraic of dimension $d:=\dim Z$. Moreover, $\pr_I(X)$ lies in no proper weakly-special
subvariety, since pulling back to $S$ would contradict minimality.

Let $\calv$ be the Ramanujan lift of $\pr_I(Z)$, and let $V:=\overline{\calv}^{\,\mathrm{Zar}}$.
Now by Theorem~A (or, which is the same,  the combination of Theorems~\ref{t:ax-schanuel} and \ref{t:HS-weakly}) we get the first inequality of
\[
d+3r\le\dim V\le d+\dim X+2r
\]
while the second follows because the $z$-coordinates vary in $\overline{\pr_I(Z)}^{\,\mathrm{Zar}}$, of dimension $d$, and the $j$-values vary in $\pr_I(X)$, of dimension $\dim X$.
Once these values are fixed, each factor contributes at most two  dimensions: after choosing $E_i$ and $h_i$, there are only finitely many possibilities for $g_i$, since
$g_i^{q+1}=-x_i h_i^{q-1}$, where $x_i$ is the fixed $j$-value. Thus $\dim X\ge r$, and consequently $X=S$. The components of $\boldsymbol j^{-1}(S)$ are algebraic and map onto $S$. Since $S=X\subset C$, maximality of $Z$ inside $\boldsymbol j^{-1}(C)$ forces $Z$ to equal the component containing it. Therefore $\boldsymbol j(Z)=S$.
\end{proof}

We recall that a point of $Y(1)^n$ is called {\em CM } if  all its coordinates correspond to Drinfeld $A$-modules with complex multiplication, i.e., if their endomorphism rings  strictly contain $A$.  A weakly-special subvariety is called {\em special} if all its constant coordinates are CM points. When there are no constant coordinates, this condition is automatically satisfied.

One can now extend (almost verbatim) our proof of the two-dimensional Andr\'e--Oort as in \cite[Theorem~3.24]{bns} to arbitrary dimensions by Corollary~\ref{c:ax-lindemann}, and obtain the following.

\begin{thm}\label{t:andre-oort}
Assume that $q$ is odd. Let $C\subset Y(1)^n$ be an irreducible closed algebraic subvariety. Then $C(\C_\infty)$
contains a Zariski dense set of CM points if and only if $C$
is special.
\end{thm}

\end{document}